\documentclass[a4paper, 11pt, twoside, leqno]{article}

\usepackage[T1]{fontenc}
\usepackage[utf8x]{inputenc}
\usepackage[UKenglish]{babel}
\usepackage{lmodern}               % Latin Modern font
\usepackage{amsmath,amsthm,amssymb}
\usepackage{mathrsfs,esint,bbm,stmaryrd,tikz-cd}
\usepackage{enumitem,graphicx,xcolor,subcaption}
\usepackage[textwidth=16cm,centering,headheight=24pt]{geometry}
\usepackage{authblk}
\newtheorem{theorem}{Theorem}
\newtheorem{corollary}[theorem]{Corollary}
\newtheorem{lemma}[theorem]{Lemma}
\newtheorem{prop}[theorem]{Proposition}

\theoremstyle{plain}
\newcommand{\thistheoremname}{}
\newtheorem*{prpr}{\thistheoremname}

\theoremstyle{remark}                  % Italic title, roman text
\newtheorem{step}{Step}

\newtheorem{remark}[theorem]{Remark}
\newtheorem{example}[theorem]{Example}

\DeclareMathOperator{\dist}{dist}

\DeclareMathOperator{\spt}{spt}
\DeclareMathOperator{\ind}{ind}

\DeclareMathOperator{\BV}{BV}

\newcommand{\abs}[1]{\left| #1 \right|}
\newcommand{\norm}[1]{\left\| #1 \right\|}
\newcommand{\one}{\mathbbm{1}}
\newcommand{\mres}
{\mathbin{\vrule height 1.6ex depth 0pt width 0.13ex\vrule
height 0.13ex depth 0pt width 1.3ex}}
\newcommand{\csubset}{\subset\!\subset}

\DeclareMathAlphabet{\mathpzc}{OT1}{pzc}{m}{it}

\renewcommand{\d}{\mathrm{d}}

\newcommand{\N}{\mathbb{N}}
\newcommand{\R}{\mathbb{R}}
\newcommand{\Z}{\mathbb{Z}}

\newcommand{\M}{\mathbb{M}}
\newcommand{\F}{\mathbb{F}}
\newcommand{\I}{\mathbb{I}}
\renewcommand{\SS}{\mathbb{S}}
\newcommand{\G}{\mathbb{G}}
\renewcommand{\P}{\mathbb{P}}

\newcommand{\NN}{\mathscr{N}}

\newcommand{\EE}{\mathscr{E}}
\newcommand{\GG}{\mathscr{G}}
\newcommand{\QQ}{\mathscr{Q}}
\renewcommand{\H}{\mathscr{H}}
\renewcommand{\L}{\mathscr{L}}

\newcommand{\eps}{\varepsilon}
\newcommand{\lb}{\llbracket}
\newcommand{\rb}{\rrbracket}
\newcommand{\Hs}{H_{\mathrm{strong}}}
\newcommand{\Hw}{H_{\mathrm{weak}}}
\newcommand{\GN}{\Gamma_p(\NN)}
\newcommand{\pN}{\pi_p(\NN)}
\newcommand{\Snice}{\S^{\mathrm{nice}}}
\newcommand{\Sgrid}{\S^{\mathrm{grid}}}

\let\oldS\S

\renewcommand{\S}{\mathbf{S}}

\definecolor{lightblue}{rgb}{0.22,0.45,0.70}   % light blue
\definecolor{darkgray}{gray}{0.4}    % dark grey
\definecolor{lightgray}{gray}{0.8}

\title{Topological point singularities of manifold-valued maps \\
and approximability by smooth maps}
\author{Giacomo Canevari and Giandomenico Orlandi\thanks{
Universit\`a di Verona, Strada le Grazie 15, 37134 Verona, Italy. \\
\emph{E-mail addresses}: \texttt{giacomo.canevari@univr.it},
\texttt{giandomenico.orlandi@univr.it}}}
\date{\today}

\begin{document}
\maketitle

\begin{abstract}
 Given a positive integer~$p$, we consider $W^{1,p}$-maps from a Euclidean domain of dimension $p+1$ into a closed Riemannian manifold $\NN$. The target manifold is required to satisfy a suitable topological condition, that is, the action of the fundamental group over the~$p$-th homotopy group must be trivial; however, we do \emph{not} assume that $\NN$ is $(p-1)$-connected.
 Using tools from geometric measure theory --- namely, chains with coefficients in the $p$-th homotopy group of~$\NN$--- we associate to each map a new object that captures its topological point singularities.
 Using this object, we characterise both the strong and weak sequential closures of smooth maps in~$W^{1,p}$.
\end{abstract}

\section{Introduction}

\paragraph{Motivation and background.}
Let~$\NN$ be a smooth, compact, connected submanifold
without boundary of some Euclidean space~$\R^m$.
Let~$\Omega\subseteq\R^n$ be a bounded Lipschitz domain,
and let~$p\geq 1$. %be \emph{an integer}.
We consider the Sobolev space of~$\NN$-valued maps on~$\Omega$,
\[
 W^{1,p}(\Omega, \, \NN)
 := \left\{u\in W^{1,p}(\Omega, \, \R^m)\colon u(x)\in\NN
 \textrm{ a.e. }x\in\Omega \right\} \! .
\]
Sobolev spaces of manifold-valued maps arise naturally in many contexts
in which admissible configurations are constrained to lie on a nonlinear target,
such as harmonic maps \cite{EellsLemaire, HeleinWood}, condensed matter physics~\cite{Mermin},
% (more precisely, in the modelling of partially ordered materials) \cite{Mermin},
and numerical graphics~\cite{CrossFields}.
% (in cross-fields algorithms for domain meshing) \cite{CrossFields}.

A natural question is whether maps in~$W^{1,p}(\Omega, \, \NN)$
can be approximated, in a suitable sense, by maps in~$C^\infty(\overline{\Omega}, \, \NN)$ --- that is, $\NN$-valued maps that can be smoothly extended to a neighbourhood of~$\overline{\Omega}$.
Classical regularisation techniques based on mollification
are not enough to address this question, because
the nonlinear constraint~$u(x)\in\NN$ is not preserved by convolution. In fact, the answer to this question
depends on the topology of~$\Omega$ and~$\NN$, as well as the Sobolev exponent~$p$.
For instance, in case~$\Omega$ is contractible, smooth maps are dense in~$W^{1,p}(\Omega, \, \NN)$ with respect to the strong~$W^{1,p}$-topology if and only if~$p\geq m$ or the homotopy group~$\pi_{\lfloor p\rfloor}(\NN)$ is trivial~\cite{SchoenUhlenbeck2, Bethuel-Density}; here~$\lfloor p\rfloor$ denotes the largest integer not greater than~$p$.
Indeed, when~$p < m$ and~$\pi_{\lfloor p\rfloor}(\NN)$ is nontrivial, elements of~$W^{1,p}(\Omega, \, \NN)$ may exhibit topological singularities that act as an obstruction toward approximability by smooth maps.
Further work on this problem (see e.g.~\cite{Hajlasz, Bethuel-DensityTrace, Bousquet2007, Mucci2009, BousquetPonceVanSchaftingen-I, BousquetPonceVanSchaftingen2014, BousquetPonceVanSchaftingen-II, BrezisMironescu2015}) eventually led to a complete characterisation~\cite{Detaille2026} of Sobolev spaces~$W^{s,p}(\Omega, \, \NN)$, with arbitrary~$s > 0$, $p\geq 1$, and~$\Omega\subseteq\R^n$, in which smooth maps are strongly dense.

% When strong density of smooth maps fails, it is natural to investigate weaker notion of approximability.
Barring density, a natural question is to characterise the elements of~$W^{1,p}(\Omega, \, \NN)$ that are approximable by smooth maps.
We denote by $\Hs^{1,p}(\Omega, \, \NN)$ and~$\Hw^{1,p}(\Omega, \, \NN)$, respectively, the strong and the sequential weak closures of smooth maps.
More precisely, we will say that a map~$u\in W^{1,p}(\Omega, \, \NN)$ belongs to~$\Hw^{1,p}(\Omega, \, \NN)$ if and only if there exists a sequence of smooth maps~$\varphi_j$ that converge to~$u$ strongly in~$L^p(\Omega)$ and satisfies~$\sup_{j\in\N}\norm{\nabla\varphi_j}_{L^p(\Omega)} < +\infty$. (This notion of weak convergence in~$W^{1,p}$ coincides with the standard one for~$p > 1$, but actually corresponds to weak$^*$ convergence in~$\BV$ when~$p=1$).
When~$p$ is \emph{not} an integer, we have
% this approach does not provide any further information, for~
$\Hw^{1,p}(\Omega, \, \NN) = \Hs^{1,p}(\Omega, \, \NN)$ \cite[Theorem~3]{Bethuel-Density}.
On the other hand, for integer~$p$, $\Hw^{1,p}(\Omega, \, \NN)$ may be strictly larger than~$\Hs^{1,p}(\Omega, \, \NN)$
% --- in fact, this is always the case if~$\pi_p(\NN)$ is nontrivial,
(see~\cite[Theorem~5]{Bethuel-Density}, \cite[Theorem~5.5]{HangLin-III}) and we may have weak sequential density of smooth maps even if strong density fails.
For instance, smooth maps are weakly sequentially dense for any integer~$p\geq 1$ and any~$(p-1)$-connected target manifold~$\NN$, that is, any~$\NN$ such that $\pi_1(\NN) = \dotsb = \pi_{p - 1} (\NN) = \{0\}$ (see~\cite[Theorem~6]{Bethuel-Density} and~\cite{Hajlasz}).
In case~$p=2$, weak sequential density holds and for an even wider class of target manifolds~\cite{PakzadRiviere}.
% Additional significant contributions to this problem include
% % the adaptation of these techniques to derive a new proof of~\eqref{eq_cahkoh3me3ohBeeQu7iengeb} in the case $p = 1$~\cite{pakzad_2003} and
% the introduction of the notion of \emph{scan}, which provides a novel framework for studying strong and weak approximation problems, both for $ W^{1,3}\brk{\Omega, \Sset^{2}} $ and, more generally, in the presence of topological obstructions arising from any non-torsion component of $\pi_{p}(\NN)$ \cite{Hardt_Riviere_2003, Hardt_Riviere_2008}.
However, smooth maps are not always weakly sequentially dense.
In a groundbreaking result, Bethuel~\cite{Bethuel-Inventiones} showed that
% for any domain~$\Omega$ of dimension~$n\geq 4$, there holds
\begin{equation} \label{nonweakdensity}
  \Hw^{1, 3} (B^4, \, \SS^2)\subsetneq W^{1, 3}(B^4), \, \SS^2).
\end{equation}
More generally, for any integer~$p\geq 2$ Detaille and Van Schaftingen~\cite{DetailleVanSchaftingen} constructed a manifold~$\NN_p$ such that
% for any domain~$\Omega\subseteq\R^n$ of dimension~$n\geq p+1$, there holds
\begin{equation} \label{nonweakdensitybis}
  \Hw^{1, p}(B^{p+1}, \, \NN_p)\subsetneq
   W^{1, p} (B^{p+1}, \, \NN_p).
\end{equation}
In the same paper, the authors also proved that
\begin{equation} \label{nonweakdensitytris}
  \Hw^{1, 4k-1} (B^{4k}, \, \SS^{2k})\subsetneq
  W^{1, 4k-1} (B^{4k}, \, \SS^{2k})
  \qquad \textrm{for any integer } k\geq 1.
\end{equation}
% and any domain~$\Omega$ of dimension~$n\geq 4k$.
These examples reveal the presence of analytical obstructions to weak sequential density in addition to topological ones, as previously observed in lifting problems~\cite{BourgainBrezisMironescu, BethuelChiron, MironescuVanSchaftingen-Lifting} and trace extension problems~\cite{BethuelDemengel, Bethuel-Extension, MironescuVanSchaftingen-Trace}.

In this work, we aim at describing the topological singularities of maps in~$W^{1,p}(\Omega, \, \NN)$ when~$p\geq 1$ is an integer and~$\Omega$ is a bounded, Lipschitz domain of dimension~$p+1$.
To each~$u\in W^{1,p}(\Omega, \, \NN)$, we associate a geometric object~$\S(u)$
--- roughly speaking, a collection of points equipped with multiplicities in a suitable group ---
that describes the singularities of~$u$ and satisfies appropriate an notion of continuity.
Objects of this kind have been introduced in the past,
to characterise~$\Hs^{1,p}(\Omega, \, \NN)$ and~$\Hw^{1,p}(\Omega, \, \NN)$ or to study variational problems featuring energy concentration and emergence of topological singularities (see e.g.~\cite{ABO2, CO2, CaselliFregugliaPicenni2025, CLOBO}).
For instance, for maps~$u\in W^{1,p}(B^{p+1},\SS^p)$, this r\^ole is played by the distributional Jacobian, i.e., the exterior differential~$\d (u^*\omega_{\SS^2})$ of the pull-back of the volume form~$\omega_{\SS^2}$.
This approach, based on differential forms, was introduced by Bethuel~\cite{Bethuel-Density} %(in case~$p=2$)
and later extended by Bethuel, Coron, D\'emengel and H\'elein~\cite{BethuelCoronDemengelHelein},
% (who considered maps~$u\in W^{1,p}(\Omega, \, \NN)$ for a larger class of~$p$ and~$\NN$),
and by Detaille, Mironescu and Xiao~\cite{DetailleMironescuXiao2025}.
The latter even considered maps~$u\in W^{s,p}(\Omega, \, \NN)$, for a larger class of manifolds~$\NN$ and~$s > 0$ possibly smaller than one.
Using a different approach, based on tools from geometric-measure theory, Hardt and Rivi\`ere~\cite{HardtRiviere} studied the rational-homotopy singularities of~$u$ --- that is, those detected by~$\mathrm{Hom}\,(\pi_p(\NN), \, \R)$ --- for maps in~$\Hw^{1,p}(\R^{p+1}, \, \NN)$.
These approaches, based on differential forms or rational homotopy, cannot describe topological obstructions carried by finite-order elements of the homotopy groups of~$\NN$.
However, there are other approaches that are compatible with the latter.
In particular, the results in~\cite{PakzadRiviere, CO1} allow to define the topological singular set of a map~$u\in W^{1,p}(\Omega, \, \NN)$ so long as the target~$\NN$ is $(\lfloor p\rfloor - 1)$-connected and $\pi_1(\NN)$ is Abelian.
% This construction is carried out by Pakzad and Rivi\`ere~\cite{PakzadRiviere}, under the additional assumption that~$1 \leq p < 2$ or~$n-1\leq p < n$, and then again in~\cite{CO1}.
In this paper, we propose a different construction, which allows for a larger class of target manifolds~$\NN$.

\paragraph{Content of this paper.}
We consider maps~$u\in W^{1,p}(\Omega, \, \NN)$ for integer~$p$ and a domain~$\Omega$ of dimension~$n = p+1$. Our construction differs from the one in~\cite{HardtRiviere} because we aim at describing all singularities of~$u$, including those that are associated with finite-order elements of~$\pi_p(\NN)$.
Contrary to~\cite{PakzadRiviere, CO1}, we do \emph{not} assume that~$\NN$ is~$(p-1)$-connected.
Instead, we make the following assumption:
\begin{enumerate}[label=(H), ref=H]
 \item the fundamental group~$\pi_1(\NN)$ acts trivially on~$\pi_p(\NN)$.
\end{enumerate}
Despite excluding a number of interesting cases, this assumption still allows for a wide range of examples.
For instance, it is satisfied when~$p = 1$ and~$\pi_1(\NN)$ is Abelian, or when~$\NN$ is simply connected, for any value of~$p$, or when~$\NN$ is a Lie group, again for all~$p$.
% --- and is less restrictive than the corresponding assumptions in~\cite{PakzadRiviere, CO1, DetailleMironescuXiao2025}.
Assumption~\eqref{hp:single-valued}, contrarily to the ones in~\cite{PakzadRiviere, CO1}, is compatible with settings in which~$\Hw^{1,p}(\Omega, \, \NN)\subsetneq W^{1,p}(\Omega, \, \NN)$ --- for instance, the case~$p=3$, $\NN = \SS^2$.
% which is studied in~\cite{Riviere, Bethuel-Inventiones}.
For a thorough discussion on this condition, we refer the reader to Sections~\ref{sect:homotopy} and~\ref{sect:assumptions}, where we introduce the relevant topological notions assuming no prior knowledge on the reader's part.

For a map~$u\in W^{1,p}(\Omega, \, \NN)$ that is continuous out of a finite set~$\Sigma(u)$ of point discontinuities, it is natural to define
\begin{equation} \label{Stop_nice_intro}
 \Snice(u) := \sum_{x\in\Sigma(u)} \sigma(u, \, x)
   \llbracket x\rrbracket,
\end{equation}
where~$\sigma(u, \, x)\in\pi_p(\NN)$ is the free homotopy class
of~$u$ on any sufficiently small sphere centred at~$x$,
whose interior contains no singular point of~$u$ other than~$x$.
(Thanks to assumption~\eqref{hp:single-valued}, we can identify
free homotopy classes of maps~$\SS^p\to\NN$ with elements
of~$\pi_p(\NN)$.)
The quantity defined in~\eqref{Stop_nice_intro} is a ``discrete measure with values in~$\pN$'' --- or, more precisely, a polyhedral $0$-chain with coefficient group~$\pN$ (see Section~\ref{sect:chains} for the definition).
This object does carry topological information on the map~$u$.
More precisely, for every pair~$(x, \, r)\in\Omega\times (0, \, +\infty)$ such that the cube~$Q^{p+1}_r(x) := x + [-r/2, \, r/2]^{p+1}$ is contained in~$\Omega$, we define the intersection index of~$\Snice(u)$ and~$Q^{p+1}_r(x)$ as
\begin{equation} \label{intro-ind}
 \ind(\Snice(u), \, Q^{p+1}_r(x))
 := \sum_{y\in \Sigma(u)\cap Q^{p+1}_r(x)} \sigma(u, \, y).
\end{equation}
Then, assumption~\eqref{hp:single-valued} implies
\begin{equation} \label{intro-indhomotopy}
 \ind(\Snice(u), \, Q^{p+1}_r(x))
 = \textrm{free homotopy class of } u \textrm{ on } \partial Q^{p+1}_r(x)
\end{equation}
for every~$(x, \, r)$ such that~$Q^{p+1}_r(x)\subseteq\Omega$ and~$u$ is continuous on~$\partial Q^{p+1}_r(x)$.
A classical result by Bethuel~\cite[Theorem~2]{Bethuel-Density} shows that maps with finitely many point discontinuities are strongly dense in~$W^{1,p}(\Omega, \, \NN)$, which raises
the question whether~$\Snice$ can be extended continuously to this space.
We construct such an extension with values in the normed group of ``zero-dimensional intersection chains'', $\I_0(\Omega; \, \pN)$, defined in Section~\ref{sect:Ichains} below.
Intersection chains are the completion of polyhedral chains with respect to a suitable norm, so that, for every~$S\in\I_0(\Omega; \, \pN)$ and almost every cube~$Q^{p+1}_r(x)$ contained in~$\Omega$, the intersection index~$\ind(S, \, Q^{p+1}_r(x))$ is well-defined and coincides with the one given in~\eqref{intro-ind} if~$S$ is polyhedral.
The examples in Section~\ref{sect:Ichains} and Proposition~\ref{prop:flat-vs-I} below show that~$\I_0(\Omega; \, \pN)$ is strictly larger than the set of flat chains and differs from other geometric objects such as Harrison's higher-order chains~\cite{Harrison, Harrison-Isomorphism} or Hardt and Rivi\`ere's scans~\cite{HardtRiviere}.

The examples in~\cite{Bethuel-Inventiones, DetailleVanSchaftingen} suggest that weak approximability by smooth maps should be related to finiteness of the flat norm~$\F_\Omega(\S(u))$, defined in~\eqref{flat} below; however, we are unable to prove or disprove such a characterisation.
Instead, we characterise~$\Hw^{1,p}(\Omega,\NN)$ in terms of a weaker norm, namely, the lower semicontinuous envelope~$\overline{\F}_\Omega$ of~$\F_\Omega$ with respect to the topology of~$\I_0(\Omega;\pN)$ (see~\eqref{relaxedflat}).
This quantity can be bounded above and below in terms of the relaxed $p$-Dirichlet energy,
\begin{equation} \label{relaxedenergyintro}
 \overline{D}_p(u) := \inf\left\{\liminf_{j\to+\infty}
 \int_\Omega \abs{\nabla\varphi_j}^p \d x\colon
 (\varphi_j)_{j\in\N}\subseteq C^\infty(\overline{\Omega},\NN),\
 \varphi_j\to u \textrm{ in } L^p(\Omega) \right\} \! .
\end{equation}
More precisely, we have the following result.

\begin{theorem} \label{th:Stop}
 Let~$p\geq 1$ be an integer, $\Omega\subseteq\R^{p+1}$ be a bounded Lipschitz domain, and~$\NN\subseteq\R^m$ be a closed, connected manifold satisfying the condition~\eqref{hp:single-valued}.
 Then, $\Snice$ extends uniquely to a continuous map
 $\S\colon W^{1, p}(\Omega, \, \NN) \to \I_0(\Omega; \, \pN)$.
 The map~$\S$ satisfies the following properties:
 \begin{enumerate}[label=(\roman*)]
  \item For almost all~$(x, \, r)$ such that~$Q_r^n(x)\subseteq\Omega$, the free homotopy class of~$u$ on~$\partial Q^{p+1}_r(x)$ is equal to~$\ind(\S(u), \, Q^{p+1}_r(x))$.

  \item A map~$u\in W^{1,p}(\Omega, \, \NN)$ belongs to~$\Hw^{1,p}(\Omega, \, \NN)$ if and only if~$\overline{\F}_\Omega(\S(u)) < +\infty$, in which case we have
  \begin{equation} \label{relaxedFintro}
   C_1\,\overline{\F}_\Omega(\S(u))
   \leq \overline{D}_p(u)
   \leq C_2 \left(\int_{\Omega} \abs{\nabla u}^p \, \d x
    + \overline{\F}_\Omega(\S(u)) \right)
  \end{equation}
  for some constants~$C_1$, $C_2$ that depend only on~$\Omega$, $\NN$, $p$.

  \item A map~$u\in W^{1,p}(\Omega, \, \NN)$ belongs to~$\Hs^{1,p}(\Omega, \, \NN)$ if and only if~$\S(u) = 0$.
 \end{enumerate}
\end{theorem}

Contrary to the proofs of~\cite{PakzadRiviere, HardtRiviere} --- which rely on the use of differential forms --- and that of~\cite{CO1} --- which relies on a projection argument, inspired by~\cite{HKL} ---, the proof of Theorem~\ref{th:Stop} is based on an approximation of~$\Snice$ using suitable grids.
The proof of a key technical estimate (Lemma~\ref{lemma:Sbdd} below) relies on extension arguments for manifold-valued maps that are reminiscent of~\cite[Lemma~1]{Luckhaus-PartialReg}.

As mentioned above, the example~\cite{Bethuel-Inventiones, DetailleVanSchaftingen} motivate the question as to whether the condition $u\in\Hw^{1,p}(\Omega, \, \NN)$ is equivalent to~$\F_\Omega(\S(u)) < +\infty$.
Statement~(ii) proves one implication: if~$\F_\Omega(\S(u)) < +\infty$ then~$\overline{\F}_\Omega(\S(u)) < +\infty$ and, hence, $u\in\Hw^{1,p}(\Omega, \, \NN)$ by Theorem~\ref{th:Stop}.
In fact, the condition~$\overline{\F}_\Omega(S) < +\infty$ is even strictly weaker than~$\F_\Omega(S) < +\infty$, in general; see Proposition~\ref{prop:Frelax2} below for an example.
However, the map~$\S$ is not surjective.
Therefore, it is possible that, for chains~$S$ \emph{in the image of~$\S$}, the condition~$\overline{\F}_\Omega(S) < +\infty$ implies~$\F_\Omega(S) < +\infty$.
This is precisely what happens when~$\NN$ is~$(p-1)$-connected, for in that case the range of~$\S$ is contained in the set of flat chains (see~\cite{PakzadRiviere, CO1}).
While we do not have a complete proof or disproof of this fact, we have at least a partial result: we show that~$\S(u)$ is a flat chain if~$u\in\Hw^{1,p}(\Omega, \, \NN)$ satisfies a suitable Orlicz-Sobolev condition.
More precisely, let~$\phi\colon [0, +\infty)\to\R$ be a nonnegative, nondecreasing, convex function that satisfies~$\phi(0) = 0$ and
\begin{equation} \label{hp:Orliczintro}
 \int_1^{+\infty} \frac{\d t}{\phi(t)} < +\infty.
\end{equation}
For instance, the functions~$\phi(t) = t^\alpha$ or~$\phi(t) = \max\{t(\log t)^\alpha, \, 0\}$, with~$\alpha > 1$, satisfy this condition.

\begin{prop} \label{prop:Orliczflat}
 Let~$u\in\Hw^{1,p}(\Omega, \, \NN)$.
 Suppose that there exists a nonnegative, nondecreasing, convex function $\phi\colon [0, \, +\infty)\to[0, \, +\infty)$ that satisfies~$\phi(0) = 0$, \eqref{hp:Orliczintro} and
 \begin{equation} \label{hp:OrliczSobolev}
  \int_{\Omega} \phi(\abs{\nabla u}^p) \, \d x < +\infty.
 \end{equation}
 Then, $\S(u)$ is a flat chain and satisfies
 \begin{equation*}
   C_1\,\F_\Omega(\S(u))
   \leq \overline{D}_p(u)
   \leq C_2 \left(\int_{\Omega} \abs{\nabla u}^p \, \d x
    + \F_\Omega(\S(u)) \right)
  \end{equation*}
  for some constants~$C_1$, $C_2$ that depend only on~$\Omega$, $\NN$, $p$.
\end{prop}

As a technical step towards the proof of Theorem~\ref{th:Stop} and Proposition~\ref{prop:Orliczflat}, we prove the following result
(which is completely independent from the assumptions~\ref{hp:Orliczintro}, \eqref{hp:OrliczSobolev}).
Here~$R^{1,p}(\Omega, \, \NN)$ denotes the class of maps $\Omega\to\NN$ with ``nice singularities'', i.e. locally Lipschitz continuous away from a finite set of point singularities, with a quantitative control on the gradient (see Section~\ref{sect:Sobolev} for the definition).

\begin{prop} \label{prop:RDbar}
 Let~$p\geq 1$ be an integer, $\Omega\subseteq\R^{p+1}$ be a bounded Lipschitz domain, and~$\NN\subseteq\R^m$ be a closed, connected manifold satisfying the condition~\eqref{hp:single-valued}.
 Then, there exists~$h_* > 0$, depending only on~$\Omega$, such that the following statement holds:
 For all~$u\in\Hw^{1,p}(\Omega, \, \NN)$, there exists a family of maps~$(u_h)_{0 < h\leq h_*}\subseteq R^{1,p}(\Omega, \, \NN)$ that converges strongly in~$W^{1,p}(\Omega)$ to~$u$ as~$h\to 0$ and satisfies
 \[
  \overline{D}_p(u_h) \leq C \overline{D}_p(u),
 \]
 for all~$h\in (0, \, h_*]$ and some constant~$C$ that depends
 only on~$\Omega$, $\NN$, $p$.
\end{prop}

The proof of Proposition~\ref{prop:RDbar} is based on techniques from~\cite{Bethuel-Density, BousquetPonceVanSchaftingen-II}, but require some additional care as we need a uniform control of~$\overline{D}_p(u_h)$.

\paragraph{Further questions.}
Theorem~\ref{th:Stop} leaves several natural questions open.
For instance, instead of~\eqref{relaxedFintro}, one might expect a more symmetric estimate to hold, such as
\begin{equation} \label{relaxedFintro-bis}
  C_1\,\overline{\F}_\Omega(\S(u))
  \leq \overline{D}_p(u) - \int_\Omega\abs{\nabla u}^p \d x
  \leq C_2 \, \overline{\F}_\Omega(\S(u))
\end{equation}
for all~$u\in\Hw^{1,p}(\Omega, \, \NN)$ and some constants~$C_1$, $C_2$ depending only on~$\Omega$, $\NN$, $p$.
It does not seem impossible that~\eqref{relaxedFintro-bis} could be proved by techniques similar to the ones developed here, at the cost of some additional technical work, but we have not pursued this question.
Similarly, the constructions used in this paper do not seem to present an immediate obstruction to an extension to higher-dimensional domains.
However, the proof of Theorem~\ref{th:Stop} relies on tools from geometric measure theory, and in particular on the notion of intersection chain; extending this framework to higher dimensions would require further development.
% More generally, several questions concerning the structure and properties of intersection chains remain open, including their possible relation with other geometric objects.
Another open question concerns the characterisation of maps~$u\in\Hw^{1,p}(\Omega, \, \NN)$ in terms of the flat norm of~$\S(u)$ --- in particular, whether the condition~$\F_\Omega(\S(u)) < +\infty$ is satified all~$u\in\Hw^{1,p}(\Omega, \, \NN)$, independently of~\eqref{hp:OrliczSobolev}.
As observed above, this is precisely what happens under stronger topological conditions on the target manifold~$\NN$, namely, when~$\NN$ is~$(p-1)$-connected~\cite{PakzadRiviere, CO1}.
Moreover, Hardt and Rivi\`ere's results~\cite{HardtRiviere} imply that all ``rational components of~$\S(u)$'' --- that is, all real-valued chains of the form~$\S^z(u) := z(\S(u))$ for~$z\in\operatorname{Hom}(\pN; \, \R)$ ---
have finite flat norm when~$u$ is weakly approximable by smooth maps.
However, these result do not immediately carry over to our setting.
Finally, if the assumption~\eqref{hp:single-valued} is not satisfied, then free homotopy classes of maps~$\SS^p\to\NN$ still carry a natural operation (see Section~\ref{sect:homotopy}), but this operation is multi-valued.
This prevents us from applying the existing machinery from geometry measure theory.
While it would be tempting to introduce chains with coefficients in a ``multi-valued group'', it is not clear whether such a modification alone would be sufficient to address this more general setting, because of the possible occurrence exotic phenomena, such as the ``tethering of defects'' (see e.g.~\cite[Section~VI]{Mermin}).
We leave these questions for future investigation.

\paragraph{Statement on AI use.}
The arguments in Section~\ref{sect:chains} were developed by the authors with the assistance of ChatGPT (OpenAI), using the GPT-5.6 Sol model. ChatGPT was used primarily to help develop ideas proposed by the authors and to fill in technical details. The only exception is the counterexample in Proposition~\ref{prop:Frelax2}, which was proposed by ChatGPT after the authors suggested that the conjectured statement might be false. In all cases, the statements and proofs were independently checked, revised, and adapted by the authors.
In all the other sections, the authors did not make use of AI tools, except for linguistic and stylistic assistance and to generate TikZ code.

\paragraph{Plan of the paper.}
The paper is organised as follows. 
Section~\ref{sect:chains} contains preliminary results in geometric measure theory.
In particular, we discuss zero-dimensional flat (Subsection~\ref{sect:flatchains}) and intersection chains (Subsection~\ref{sect:Ichains}); while the results in Subsection~\ref{sect:flatchains} are mostly classical~\cite{Fleming, White-Rectifiability, White-Deformation}, those in Subsection~\ref{sect:Ichains} are new.
In Section~\ref{sect:N}, we recall some preliminary material on Sobolev spaces of manifold-valued maps (Subsection~\ref{sect:Sobolev}) and free homotopy classes of maps~$\SS^p\to\NN$ (Subsection~\ref{sect:homotopy}).
In particular, we introduce the group~$\pi_p(\NN)$ assuming no prior knowledge of algebraic topology on the reader's side.
Section~\ref{sect:Sgrid} introduces an approximation of~$\Snice$ based on grids and contains key technical estimates toward the proof of Theorem~\ref{th:Stop} (see Lemma~\ref{lemma:Sbdd}).
Section~\ref{sect:RS} is devoted to the proof of Proposition~\ref{prop:RDbar}.
Finally, we complete the proof of Theorem~\ref{th:Stop} and Proposition~\ref{prop:Orliczflat} in Section~\ref{sect:proofmainthm}.

\numberwithin{equation}{section}
\numberwithin{theorem}{section}

\section{Zero-dimensional flat and intersection chains}
\label{sect:chains}

\subsection{Zero-dimensional flat chains}
\label{sect:flatchains}

\paragraph{Polyhedral chains.}
Let~$\G$ be a nontrivial normed Abelian group --- that is, an Abelian group~$\G\neq \{0\}$ (which we will write in additive notation) together with a non-negative function $|\,\cdot\,|\colon\G\to [0, \, +\infty)$ that satisfies
\begin{enumerate}[label=(\roman*)]
 \item $|\sigma| = 0$ if and only if~$\sigma=0$
 \item $|-\sigma| = |\sigma|$ for all~$\sigma\in\G$
 \item $|\sigma_1 + \sigma_2|\leq |\sigma_1|+|\sigma_2|$ for all~$\sigma_1\in\G$, $\sigma_2\in\G$.
\end{enumerate}
We assume that
\begin{equation} \label{alpha_G}
 \alpha_{\G} := \inf\left\{\abs{\sigma}\colon \sigma\in\G, \ \sigma\neq 0 \right\} > 0.
\end{equation}
For~$d\in\Z$, $0 \leq d\leq n$, a polyhedral $d$-chain in~$\R^n$ with coefficients in~$\G$ is a linear combination, with coefficients in~$\G$, of compact, convex, oriented, non-degenerate, $d$-dimensional polyhedra in~$\R^n$,
modulo a suitable equivalence relation~$\sim$.
We define~$\sim$ by requiring $-K\sim K^\prime$ if the polyhedra $K^\prime$ and~$K$ only differ for the orientation, and $K\sim K_1 + K_2$ if $K$ is obtained by gluing $K_1$, $K_2$ along a common face (with the correct orientation).
The set~$\P_d(\R^n; \, \G)$ of polyhedral $d$-chains with coefficients in~$\G$ is a group, with a naturally defined addition operation.
Every element~$S\in\P_d(\R^n; \, \G)$ can be represented as a finite sum
\begin{equation} \label{polyhedral}
 S = \sum_{i = 1}^q \sigma_i\llbracket K_i\rrbracket,
\end{equation}
where $\sigma_i\in\G$, the $K_i$'s are compact, convex, non-overlapping $d$-dimensional polyhedra, and $\llbracket\,\cdot\,\rrbracket$ denotes the equivalence class modulo the relation~$\sim$ defined above.
(We say that the polyhedra are non-overlapping if, for each pair of indices~$i\neq j$, the intersection~$K_i\cap K_j$ only contains points that belong to the boundary faces of~$K_i$ and~$K_j$.)
% Thus, $S$ may be identified with a finite collection of polyhedra as above, endowed with multiplicities in~$\G$.

Polyhedral chains come naturally with notion of boundary, mass, and support.
The boundary operator is a group homomorphism $\partial\colon\P_d(\R^n; \, \G)\to\P_{d-1}(\R^n; \, \G)$ identified by its actions on single polyhedra,
as follows. For all~$\sigma\in\G$ and all $d$-dimensional polyhedron~$K$ with~$d\geq 1$, let~$e_1$, \ldots, $e_q$ denote the $(d-1)$-dimensional faces of~$K$; then, we define~$\partial\left(\sigma\llbracket K\rrbracket\right) := \sum_{i=1}^q \sigma \llbracket e_i \rrbracket$. We also define $\partial S := 0$ for all $0$-dimensional polyhedral chains.
A classical computation shows that $\partial(\partial S) = 0$ for any polyhedral $d$-chain~$S$ with~$d\geq 1$.
The mass of a polyhedral chain~$S\in\P_d(\R^n; \, \G)$,
written in the form~\eqref{polyhedral} for non-overlapping~$K_i$, is defined by
\begin{equation} \label{mass}
 \M(S) := \sum_{i=1}^q |\sigma_i| \H^d(K_i)
\end{equation}
The mass of a polyhedral chain does not depend on the choice of its representation. %in the form~\eqref{polyhedral}.
Finally, the support of a chain~$S$, written again as as~\eqref{polyhedral}, is defined by~$\spt S := \cup_{i=1}^q K_i$.
Given a set~$E\subseteq\R^n$, we will write~$\P_d(E; \, \G)$ for the subgroup of chains~$S\in\P_d(\R^n; \, \G)$ such that~$\spt S\subseteq E$.
Given polyhedral chains $A$, $B$ and a set~$E\subseteq\R^n$, we will write~$A = B$ in~$E$ as a synonym for~$\spt(A - B)\subseteq\R^n\setminus E$.

\paragraph{The flat norm.}

From now on, we will consider $0$-dimensional chains in a bounded open set~$\Omega\subseteq\R^n$.
We consider the group~$\P_0(\Omega; \, \G)$ and equip it with several different norms --- i.e., non-negative functionals satisfying Properties~(i)--(iii) above.
The first norm we consider is the flat norm, which was introduced by Whitney~\cite{Whitney-GIT}
and later played an important role in geometric measure
theory (see e.g.~\cite{Federer, Fleming, White-Rectifiability}).
For any polyhedral chain~$S\in\P_0(\Omega; \, \G)$, we define its \emph{flat norm relative to~$\Omega$} as
\begin{equation} \label{flat}
 \F_\Omega(S) := \inf\left\{\M(A) + \M(B)\colon
  A\in\P_0(\R^n; \, \G), \ B\in\P_1(\R^n; \, \G), \
  S = A + \partial B \textrm{ in } \Omega \right\} \! .
\end{equation}

\begin{remark} \label{rk:dipolarOmega}
 A truncation argument shows that the infimum in~\eqref{flat} does not change if we only consider chains~$A$, $B$ such that
 $\spt A \subseteq\Omega$, $\spt B\subseteq\overline{\Omega}$, and~$\spt B\cap\partial\Omega$ contains at most finitely many points (the endpoints of some of the segments of~$B$).
\end{remark}

The definition immediately implies that~$\F_\Omega$ is a seminorm and that $\F_\Omega(S) \leq \M(S)$ for all~$S$.
In fact, $\F_\Omega$ is a norm.
This is a well-known fact (see e.g.~\cite[Theorem~2.2]{Fleming}), but we include a proof just for the reader's convenience.

\begin{lemma} \label{lemma:Fnorm}
 We have~$\F_{\Omega}(S) > 0$
 for all nonzero~$S\in\P_0(\Omega; \, \G)$.
\end{lemma}

In the proof, we will make use of the augmentation homomorphism
$\chi\colon\P_0(\Omega; \, \G)\to\G$.
For a chain~$S\in\P_0(\Omega; \, \G)$ written in the form~$S = \sum_{i=1}^N \sigma_i \lb x_i\rb$, with~$\sigma_i\in\G$ and~$x_i\in\Omega$ distinct, we define
\begin{equation} \label{eps*}
 \chi(S) := \sum_{i=1}^N \sigma_i.
\end{equation}
For future reference, we point out that
\begin{equation} \label{eps*bd}
 \chi(\partial B) = 0 \qquad\textrm{for all } B\in\P_1(\R^n; \, \G)
 \textrm{ such that } \spt B \subseteq\Omega.
\end{equation}
The equality~\eqref{eps*bd} is immediate in case~$B = \tau \lb L\rb$ for some line segment~$L\subseteq\Omega$; the general case follows by linearity.

\begin{proof}[Proof of Lemma~\ref{lemma:Fnorm}]
 Let~$S\in\P_0(\Omega; \, \G)$ be such that~$\F_\Omega(S) = 0$.
 Assume that the support of~$S$ is not empty, take~$x\in\spt S$ and write~$S = \sigma\lb x\rb + T$, where~$\sigma\in\G$, $\sigma\neq 0$ and~$T$ is a chain with support in~$\R^n\setminus\{x\}$.
 We define
 \begin{equation*}
  \rho := \min\left\{ 1, \, \dist(x, \, \partial\Omega), \,
   \min_{y\in \spt T} \abs{y - x} \right\} \! .
 \end{equation*}
 Since~$\F_\Omega(S) = 0$, we can write~$S = A + \partial B$ in~$\Omega$, where~$A$ and~$B$ are polyhedral chains such that
 $\M(A) + \M(B) \leq \frac{1}{2}\rho\alpha_{\G}$,
 where~$\alpha_{\G}$ is given by~\eqref{alpha_G}.
 As~$\rho/2 < 1$, the inequality~\eqref{alpha_G}
 implies that~$A = 0$. Moreover, the coarea formula implies
 \[
  \int_0^{\rho} \H^0\!\left(\partial B_r(x)
   \cap\spt B\right) \, \d r \leq \frac{\M(B)}{\alpha_{\G}}
   \leq \frac{\rho}{2}.
 \]
 Therefore, there exists~$r > 0$ such that~$\partial B_r(x)$
 contains no point of~$\spt B$.
 Writing~$B = \sum_{j=1}^q\tau_j \lb D_j\rb$ for~$\tau_j\in\G$ and straight line segments~$D_j$, each~$D_j$ must be either contained
 in~$B_\rho(x)$ or in~$\R^n\setminus\overline{B}_r(x)$.
 Let~$B^\prime := \sum_{j\colon D_j\subseteq B_r(x)} \tau_j \lb D_j\rb$.
 By examining the equality~$\sigma \lb x \rb + T = \partial B$ in~$\Omega$, and restricting both sides to~$B_r(x)$, we deduce
 \begin{equation} \label{norms}
  \sigma \lb x\rb = \partial B^\prime,
 \end{equation}
 since $B_r(x)$ is contained in~$\Omega$ and does not
 contain any point of the support of~$T$. However,
 by applying the homomorphism~$\chi$ defined in~\eqref{eps*}
 to both sides of~\eqref{norms}, and recalling~\eqref{eps*bd},
 we obtain~$\sigma = 0$.
 This is a contradiction.
 Therefore, the support of~$S$ is empty, that is~$S = 0$.
\end{proof}

We define~$\F_0(\Omega; \, \G)$ as the metric-space completion
of~$\P_0(\Omega; \, \G)$ with respect to~$\F_\Omega$.
The group operation on~$\P_0(\Omega; \, \G)$ extend by
(Lipschitz) continuity to~$\F_0(\Omega; \, \G)$,
so that the latter is again an Abelian group.

\begin{remark}
 Flat chains relative to~$\Omega$ were defined in~\cite{CO1} as the quotient of~$\F_0(\R^n; \, \G)$ with respect to the subgroup of chains supported in~$\R^n\setminus\Omega$.
 However, this definition is equivalent to the one we have given here; this follows, e.g., from~\cite[Lemma~2]{CO1}.
\end{remark}

\paragraph{Restriction of flat chains.}

Let~$U\subseteq\Omega$ be a Borel set.
Given a polyhedral chain~$S = \sum_{j=1}^N \sigma_j \lb x_j\rb$,
where the~$x_j$'s are distinct points in~$\Omega$
and~$\sigma_j\in\G$, the restriction of~$S$ to~$U$
is defined by
\begin{equation} \label{restrV}
 S\mres U := \sum_{j\in\{1, \, \ldots, \, N\}
  \colon x_j\in U} \sigma_j \lb x_j\rb .
\end{equation}
If the set~$U$ is open, then we have the estimate
\begin{equation} \label{restr_flat}
 \F_U(S\mres U) \leq \F_\Omega(S).
\end{equation}
Indeed, a decomposition~$S = A + \partial (\sum_{j=1}^q \tau_j \lb L_j \rb)$ in~$\Omega$, where~$\tau_j\in\G$ and each~$L_j$ is a straight line segment, induces a decomposition~$S\mres U = A\mres U + \partial (\sum_{j=1}^q \tau_j \lb L_j^\prime \rb)$ in~$U$, where each~$L_j^\prime$ is a connected component of~$L_j\cap U$.
(The endpoints of~$L_j^\prime$ that are not endpoints of~$L_j$, i.e.~those induced by the restriction to~$U$, belong to~$\partial U$ and hence, do not affect the relative flat norm~$\F_U(S\mres U)$.)
%
% \begin{remark}
%  It is \emph{not} true in general
%  that~$\F_\Omega(S\mres E)\leq \F_\Omega(S)$. For instance,
%  if~$S_k = \lb 1/k \rb - \lb - 1/k \rb\in\P_0(\R; \, \Z)$
%  and~$E := (0, \, +\infty)$, then $\F_\R(S_k) = 1/k\to 0$
%  as~$k\to+\infty$, but~$\F_\R(S_k\mres E) = 1$ for all~$k\geq 1$.
%  In fact, the restriction~$S\mres E$ need not even be
%  well-defined as an element of~$\F_0(\Omega; \, \G)$.
%  However, if $S_j$ is a sequence of chains that
%  $\F_\Omega$-converges to~$S$ and~$f\colon\R^n\to\R$ is a Lipschitz function,
%  then for almost every~$t\in\R$
%  the restrictions $S_j\mres f^{-1}(-\infty, \, t)$,
%  $S\mres f^{-1}(-\infty, \, t)$ are well-defined and
%  $\F(S_j\mres f^{-1}(-\infty, \, t) - S\mres f^{-1}(-\infty, \, t))\to 0$
%  as~$j\to+\infty$ (see e.g. \cite[Theorem~5.2.3.(2)]{DePauwHardt}.
% \end{remark}
Therefore, the restriction induces a continuous group homomorphism~$\F_0(\Omega; \, \G)\to\F_0(U; \, \G)$.
For simplicity, we will still use the notation~$S\mres U$ to denote the image of~$S\in\F_0(\Omega; \, \G)$ under such homomorphism.

\begin{lemma} \label{lemma:localisation}
 Let $U_1,\ldots,U_N$ be open, pairwise disjoint polyhedral sets contained in~$\Omega$. Then, we have
 \begin{equation*} %\label{flat-localisation}
  \sum_{i=1}^N \F_{U_i}(S\mres U_i)
  \leq \F_\Omega(S)
 \end{equation*}
 for every $S\in\F_0(\Omega; \, \G)$.
\end{lemma}
\begin{proof}
 We first consider a polyhedral chain $S$ and a decomposition
 $S=A+\partial B$ in~$\Omega$.
 Up to subdividing $A$ and $B$ along the boundaries of the $U_i$'s,
 we have
 \[
  S\mres U_i=(A\mres U_i)+\partial(B\mres U_i)
  \qquad\hbox{in }U_i.
 \]
 Indeed, the additional boundary terms produced by restricting $B$
 are supported on $\partial U_i$ and are therefore irrelevant for
 equality relative to $U_i$.
 Hence we have
 \[
  \F_{U_i}(S\mres U_i)
  \leq \M(A\mres U_i)+\M(B\mres U_i)
 \]
 for all~$i\in\{1, \, \ldots, \, N\}$ and,
 since the $U_i$'s are pairwise disjoint,
 \[
  \sum_{i=1}^N \F_{U_i}(S\mres U_i)
  \leq \M(A)+\M(B).
 \]
 Taking the infimum over all decompositions of $S$ proves
 the claim for polyhedral chains.
 The general case follows by a density argument.
\end{proof}

\paragraph{Intersection index for flat chains.}

Let~$S = \sum_{j=1}^N \sigma_j \lb x_j\rb$ be a polyhedral chain, where the~$x_j$'s are distinct points in~$\Omega$ and~$\sigma_j\in\G$.
The intersection index of~$S$ and a Borel set~$U$ is defined as
\begin{equation} \label{intersection}
 \ind(S, \, U) := \chi(S\mres U)
  = \sum_{i\in\{1, \, \ldots, \, N\}\colon x_i\in U} \sigma_i,
\end{equation}
where~$\chi$ is the augmentation, given in~\eqref{eps*}, and~$\cdot\,\mres U$ is the restriction, defined in~\eqref{restrV}.
We will be mainly interested in the intersection of flat chains with cubes of the form~$Q^n_r(x) := x + [-r/2, \, r/2]^n$, for~$x\in\R^n$ and~$r > 0$.
% Note that for a chain~$S$ supported in~$\Omega$, $\ind(S, \, Q^n_r(x)) = \ind(S, \, Q^n_r(x)\cap\Omega)$.
We also define
\begin{equation} \label{Q_Omega}
 \QQ_\Omega := \left\{(x, \, r)\in\Omega\times (0, \, +\infty)\colon
 Q^n_r(x)\subseteq\Omega \right\} \! .
\end{equation}

\begin{lemma} \label{lemma:I}
 For all polyhedral chain~$S\in\P_0(\Omega; \, \G)$, there holds
 \begin{equation} \label{boundI}
  \int_{\QQ_\Omega} \abs{\ind\!\left(S, \, Q^n_r(x)\right)} \d r \, \d x \leq \max\!\left(\abs{\QQ_{\Omega}}\!, \, 2\abs{\Omega}\right) \F_\Omega(S).
 \end{equation}
\end{lemma}

Lemma~\ref{lemma:I}, combined with a density argument, implies that~$\ind$ induces a continuous map $\F_0(\Omega; \, \G)\to L^1(\QQ_{\Omega}; \, \G)$, for all~$R\geq 1$.
In particular, the intersection index~$\ind(S, \, Q^n_r(x))$ is well-defined as an element of~$\G$ for all~$S\in\F_0(\Omega; \,  \G)$ and almost all~$(x, \, r)\in \QQ_\Omega$.
The estimate~\eqref{boundI} remains valid for an arbitrary~$S\in\F_0(\Omega; \,  \G)$.

In the proof of Lemma~\ref{lemma:I}, we will use the following notation.
Let~$L\subseteq\overline{\Omega}$ be an oriented straight line segment,
% such that~$L\cap\partial\Omega$ contains at most one point,
and let~$\lb L\rb\in\P_1(\R^n; \, \Z)$ be the associated $1$-polyhedral chain with (integer) unit multiplicity.
Let~$x\in\Omega$ and $r > 0$ be such that~$Q^n_r(x)\subseteq\Omega$ and~$L\cap \partial Q^n_r(x)$ is a finite set.
(If~$L$ is not parallel to any of the coordinate hyperplanes~$\{x_j = 0\}$, then this condition is satisfied for all~$x$ and~$r$; otherwise, for all~$x$ there are at most two values of~$r$ such that this condition fails.)
We define a multiplicity~$\eps(y)\in\Z$ for each~$y\in L\cap \partial Q_r^n(x)$ according to the orientation of~$L$;
more precisely, we let~$\eps(y) = 1$ if~$L$ is pointing out of~$Q^n_r(x)$ at~$y$ and~$\eps(y) := -1$ otherwise.
We define the slice of~$\lb L\rb$ along~$\partial Q_r^n(x)$ as
\begin{equation} \label{slice}
 \langle \lb L\rb, \, x, \, r\rangle :=
 \sum_{y\in L\cap \partial Q^n_r(x)}
  \eps(y) \lb y\rb \in \P_0(\Omega; \, \Z).
\end{equation}
The coarea formula, applied to the $1$-Lipschitz function
$d_\infty(y) := \abs{y - x}_{\infty} =
\max_{1 \leq i \leq n} \abs{y_i - x_i}$,
implies that
\begin{equation} \label{massslice}
 \int_0^{2\dist_\infty(x, \, \partial\Omega)}
  \M(\langle \lb L\rb, \, x, \, r\rangle) \, \d r
 = \int_0^{2\dist_\infty(x, \, \partial\Omega)}
  \#(L\cap d_\infty^{-1}(r/2)) \, \d r
 \leq 2\H^1(L),
\end{equation}
where~$\dist_\infty(x, \, \partial\Omega) := \inf_{y\in\partial\Omega} d_\infty(y)$
and~$\#$ denotes the cardinality of a set.
The following property of slices will also be useful.

\begin{lemma} \label{lemma:Islice}
 If~$Q^n_r(x)\subseteq\Omega$ and~$L\cap \partial Q_r^n(x)$ is finite, then
 \begin{align}
  \ind(\partial \lb L \rb, \, Q^n_r(x))
  &= -\chi\!\left(\langle \lb L\rb, \, x, \, r\rangle\right) \! , \label{Islice1} \\
  \abs{\ind(\partial \lb L \rb, \, Q^n_r(x))}
  &\leq \M\!\left(\langle \lb L\rb, \, x, \, r\rangle\right) \! , \label{Islice}
 \end{align}
\end{lemma}
\begin{proof}
 The estimate~\eqref{Islice} follows from~\eqref{Islice1}, by definition of~$\chi$ and the triangle inequality.
 The set~$L\cap\partial Q^n_r(x)$ contains either zero, one, or two points.
 If~$L\cap\partial Q^n_r(x)$ is empty, then the right-hand side of~\eqref{Islice1} is zero, but so is the left-hand side, because~$\partial \lb L \rb \mres Q^n_r(x)$ contains either zero points or two points with opposite multiplicities.
 If~$L\cap\partial Q^n_r(x)$ contains exactly two points~$y_1$ and~$y_2$, then we must have~$\eps(y_1) = -\eps(y_2)$ by reasons of orientation, so the right-hand side of~\eqref{Islice1} is zero. In this case, however, there cannot be any endpoint of~$L$ inside~$Q^n_r(x)$, hence the left-hand side is zero as well.
 Finally, if~$L\cap\partial Q^n_r(x)$ contains exactly one point, then there must be exactly one endpoint of~$L$ inside~$Q^n_r(x)\subseteq\Omega$.
 Then, the absolute values of both sides of~\eqref{Islice1} are both equal to~$1$, and the negative sign in front of the right-hand side depends on the orientation.
\end{proof}

\begin{proof}[Proof of Lemma~\ref{lemma:I}]
 Let~$S\in\P_0(\Omega; \, \G)$ be a chain, written in the form
 \begin{equation} \label{I0}
  S = A + \partial B
  = A + \sum_{j=1}^q \tau_j \, \partial\lb L_j\rb
  \qquad \textrm{in } \Omega,
 \end{equation}
 where~$A$, $B$ are polyhedral chains,
 $\tau_j\in\G$, and the~$L_j$'s are oriented line segment contained in~$\overline{\Omega}$, which are pairwise disjoint except possibly for their endpoints.
 Let~$x\in\Omega$ and~$r\in (0, \, R)$.
 Since~$\ind(\cdot, \, Q^n_r(x))$ is a group homomorphism~$\P_0(\Omega; \, \G)\to\G$, we have
 \[
  \abs{\ind(S, \, Q^n_r(x))}
  \leq \abs{\ind(A, \, Q^n_r(x))}
   + \sum_{j=1}^q \abs{\tau_j}
   \abs{\ind(\partial\lb L_j\rb, \, Q^n_r(x))} \! .
 \]
 The first term in the right-hand side can be bounded from above as $\abs{\ind(A, \, Q^n_r(x)\cap\Omega)}\leq \M(A)$, by the triangle inequality.
 The other terms can be estimated by applying Lemma~\ref{lemma:Islice}, since for all~$x$, there are at most finitely many values of the radius~$r > 0$ such that the assumptions of Lemma~\ref{lemma:Islice} fail.
 Therefore, we deduce
 \[
  \abs{\ind(S, \, Q^n_r(x))}
  \leq \M(A) + \sum_{j=1}^q \abs{\tau_j}
   \M\!\left(\langle \lb L_j\rb, \, x, \, r\rangle\right)
 \]
 for all~$x$ and almost all~$r$.
 By taking the integral over all~$x\in\Omega$ and all~$r$ with~$0 < r < 2\dist_\infty(x, \ \partial\Omega)$, and applying the estimate~\eqref{massslice}, we obtain
 \begin{equation} \label{Iflat-global}
  \begin{split}
   \int_{\QQ_\Omega}
    \abs{\ind(S, \, Q^n_r(x))} \d r \, \d x
   &\leq \abs{\QQ_\Omega} \M(A)
    + 2\abs{\Omega} \sum_{j=1}^q \abs{\tau_j} \H^1(L_j)
   = \abs{\QQ_\Omega} \M(A)
    + 2\abs{\Omega} \M(B).
  \end{split}
 \end{equation}
 Taking the infimum over all possible~$A$ and~$B$ satisfying~\eqref{I0}, and recalling Remark~\ref{rk:dipolarOmega}, the lemma follows.
\end{proof}

\begin{remark} \label{rk:Iflat}
 Let~$(S_j)\subseteq\P_0(\Omega; \, \G)$ be a sequence that converges to~$S$ in the~$\F_\Omega$-norm.
 By Lemma~\ref{lemma:I}, we can extract a (non-relabelled) subsequence in such a way that for almost every~$(x, \, r)\in\QQ_\Omega$, the indices~$\ind(S_j, \, Q^n_r(x))$ converge to~$\ind(S, \, Q^n_r(x))$ as~$j\to+\infty$.
\end{remark}

\begin{remark} \label{rk:localflat}
 By reasoning exactly as in Lemma~\ref{lemma:I}, we prove the estimate
 \begin{equation} \label{Iflat-local}
  \begin{split}
   \int_0^{2\dist_\infty(x, \, \partial\Omega)}
    \abs{\ind(S, \, Q^n_r(x))} \d r
   &\leq 2\dist_\infty(x, \, \partial\Omega) \M(A) + 2\M(B)
  \end{split}
 \end{equation}
 for all polyhedral chain~$S$ written in the form~$S = A + \partial B$ in~$\Omega$ and all~$x\in\Omega$.
 (In fact, Equation~\eqref{Iflat-local} is completely analogous to~\eqref{Iflat-global}, except that we do not integrate over~$x\in\Omega$.)
 Using~\eqref{Iflat-local}, we can compute the (relative) flat norm of a few simple polyhedral chains.
 For instance, the flat norm of a ``monopole''~$M = \sigma\lb x \rb$, with~$\sigma\in\G$ and~$x\in\Omega$, is
 \[
  \F_\Omega(M) = \abs{\sigma} \min\!\left(\dist_\infty(x, \, \partial\Omega), \, 1\right) .
 \]
 The upper bound is an immediate consequence of the definition~\eqref{flat}, while the lower bound follows from~\eqref{Iflat-local}.
 For a dipole along the first coordinate direction, say~$D = \sigma\lb x + \rho e_1\rb - \sigma\lb x\rb$ for~$\sigma\in\G$, $x\in\Omega$ and~$0 < \rho < \min(\dist_\infty(x, \, \partial\Omega), \, 2)$, we have
 \[
  \F_\Omega(D) = \abs{\sigma} \rho,
 \]
 and a ``quadrupole'' $Q = \sigma\lb x + \rho e_1 + \rho e_2 \rb - \sigma \lb x + \rho e_1 \rb - \sigma \lb x + \rho e_2\rb + \sigma \lb x \rb$ with~$0 < \rho < \min(\dist_\infty(x, \, \partial\Omega)/2, \, 2)$ satisfies
 \[
  \F_\Omega(Q) = 2\abs{\sigma} \rho.
 \]
 For the latter, the upper bound follows immediately from~\eqref{flat}, while the lower bound follows by applying Lemma~\ref{lemma:localisation} to the cubes~$Q^n_{\rho}(x)$, $Q^n_{\rho}(x + \rho e_1)$, $Q^n_{\rho}(x + \rho e_2)$, $Q^n_{\rho}(x + \rho e_1 + \rho e_2)$.
\end{remark}

\paragraph{Projections onto a grid.}
\label{sect:grids}

The approximation of chains by chains supported on a fixed grid is a standard tool in geometric measure theory, most notably in the proof of the polyhedral deformation theorem~\cite{FedererFleming}.
In the zero-dimensional setting considered here, this construction takes a particularly simple form and will be quite convenient for our purposes.
First, let us introduce some notations
for cubical grids in~$\R^n$, which will be used later.
Let~$Q^n_h := Q^n_h(0) = [-h/2, \, h/2]^n$ for~$h > 0$ and~$Q^n := Q^n_1$.
Given~$h > 0$ and~$y\in\R^n$, we define the grid~$\GG(h, \, y)$
of size~$h > 0$ translated by~$hy$ with respect to the origin
as a collection of closed cubes of the form
\begin{equation} \label{grid}
 \GG(h, \, y) := \left\{h y + h z + Q^n_h \colon
  z\in\Z^{n}\right\} \! .
\end{equation}
For~$j\in\N$, $0 \leq j\leq n$, we denote by~$\GG_j(h, \, y)$
the collection of the (closed)~$j$-cells of~$\GG(h, \, y)$,
namely, the~$j$-dimensional faces of the cubes in~$\GG(h, \, y)$.
We define the~$j$-skeleton %of~$\GG(h, \, y)$
as~$R_j(h, \, y) := \cup_{K\in\GG_j(h, \, y)} K$.
% Given an open set~$\Omega\subseteq\R^n$, we denote
% by~$\GG(h, \, y, \, \Omega)$ the set of ($n$-dimensional)
% cubes~$K\in\GG(h, \, y, \, \Omega)$ such
% that~$K\cap\Omega\neq\varnothing$.
% We denote by~$\GG_j(h, \, y, \, \Omega)$
% the collection of~$j$-cells of~$\GG(h, \, y, \, \Omega)$
% and by~$R_j(h, \, y, \, \Omega)$ its~$j$-skeleton.
% $R_k(h, \, y, \, \Omega)\supseteq R_k(h, \, y)\cap \Omega$.
Given a grid~$\GG(h, \, y)$, we define its dual as
\[
 \GG^\prime(h, \, y) := \left\{h y + h z + [0, \, h]^n \colon
  z\in\Z^{n}\right\} \! .
\]
The dual grid has the following property:
for each $j$-cell~$K$ of~$\GG(h, \, y)$
there exists a unique~$n - j$-cell~$K^\prime$ of~$\GG^\prime(h, \, y)$,
called the dual cell of~$K$, such that~$K\cap K^\prime\neq\varnothing$.
We will denote by~$R^\prime_j(h, \, y)$ the~$j$-skeleton of~$\GG(h, \, y)$.
Some additional notation for grids will be introduced at the beginning of Section~\ref{sect:RS}.

We recall a well-known, but useful, observation.

\begin{lemma} \label{lemma:Fubini}
 For any open set~$U\subseteq\R^n$, any integrable function~$f\colon U\to\R$,
 any integer~$j\in\{0, \, 1, \, \ldots, \, n\}$ and any~$h > 0$, there holds
 \[
  \int_{Q^n} h^{n - j} \left(\int_{R^j(h, \, y)\cap U} f \, \d\H^j \right) \d y
  = \binom{n}{j} \int_U f(x) \, \d x.
 \]
\end{lemma}
\begin{proof}
 Let~$P$ be a~$j$-dimensional coordinate subspace,
 that is, a vector $j$-subspace spanned by a subset of
 the canonical basis~$\{e_1, \, \ldots, \, e_n\}$ of~$\R^n$,
 and let~$P^\perp$ be its orthogonal complement.
 Let~$R_P^j(h, \, y)$ be the subset of~$R^j(h, \, y)$
 containing the~$j$-cells that are parallel to~$P$.
 We have~$R_P^j(h, \, y) = R_P^j(h, \, y^\prime)$
 if~$y - y^\prime$ is parallel to~$P$ and, hence,
 \[
  \begin{split}
   \int_{Q^n} h^{n - j}
    \left(\int_{R_P^j(h, \, y)\cap U} f \, \d\H^j \right) \d y
   &= \int_{P^\perp\cap Q^n} h^{n - j}
    \left(\int_{R_P^j(h, \, y)\cap U} f \, \d\H^j \right) \d y \\
   &= \int_{P^\perp\cap Q^n_h}
    \left(\int_{(P + \xi\Z)\cap U} f \, \d\H^j \right) \d \xi
   = \int_U f(x) \, \d x,
  \end{split}
 \]
 where we have made the change of variable~$hy = \xi$
 and applied Fubini theorem.
 The lemma follows by taking the sum over
 all~$j$-dimensional coordinate subspaces~$P$.
\end{proof}

Let~$\GG(h, \, y)$ be a grid and let~$S\in\P_0(\Omega; \, \G)$ be a polyhedral chain.
We define the approximation of~$S$ induced by~$\GG(h, \, y)$ as
\begin{equation} \label{deformation}
 P(S, \, h, \, y) := \sum_{K\in\GG(h, \, y)\colon K\subseteq\Omega}
  \ind(S, \, K) \lb c_K\rb,
\end{equation}
where~$c_K$ is the centre of~$K$.
The following statement is a particular case of the deformation theorem (see~\cite{FedererFleming, White-Deformation}), which holds true for chains of arbitrary dimension.
However, since the latter is typically formulated in terms of the ``absolute'' norm~$\F_{\R^n}$ instead of its relative version~$\F_\Omega$, we provide a proof.

\begin{prop} \label{prop:deformation}
 For all~$h > 0$, the map~$P(\cdot, \, h, \, \cdot)$ extends to a continuous group homomorphism $\F_0(\Omega; \, \G)\to L^1(Q^n; \, \F_0(\Omega; \, \G))$ that satisfies
 \begin{equation} \label{Pflat}
  \int_{Q^n} \F_{\Omega}\!\left(P(S, \, h, \, y)\right)
    \d y \leq \left(\sqrt{n} + \frac{n}{2}\right) \F_{\Omega}(S)
 \end{equation}
 for all~$S\in \F_0(\Omega; \, \G)$.
 In particular, for all~$S\in \F_0(\Omega; \, \G)$, all~$h > 0$ and almost all~$y\in Q^n$, the image~$P(S, \, h, \, y)$ is well-defined and satisfies~\eqref{deformation}. Moreover, we have
 \begin{equation} \label{Pflatconv}
  \int_{Q^n} \F_{\Omega}\!\left(S - P(S, \, h, \, y)\right)
    \d y \to 0 \qquad \textrm{as } h\to 0
 \end{equation}
 for all~$S\in\F_0(\Omega; \, \G)$.
\end{prop}

In the proof of Proposition~\ref{prop:deformation}, we will make use of the following technical estimate.
For all~$h > 0$, all~$y\in Q^n$ and all~$x\in\R^n\setminus R_{n-1}(h, \, y)$, we denote by~$K_y(x)$ the unique cube of~$\GG(h, \, y)$ that contains~$x$, and by~$\pi_y(x)$ its centre.
We also define a ``truncated'' projection as a chain-valued operator, $\Pi_{h,y}\colon\Omega\setminus R_{n-1}(h, \, y)\to\F_0(\Omega; \, \Z)$, in the following way:
\[
 \Pi_{h,y}(x):=
 \begin{cases}
  \llbracket \pi_{h,y}(x)\rrbracket
    & \textrm{if } K_{h,y}(x)\subseteq\Omega,\\
  0 & \textrm{otherwise}.
 \end{cases}
\]
Note that~$\Pi_{h,y}(x)$ has integer multiplicity, no matter what~$\G$ is.

\begin{lemma} \label{lemma:modified-projection-estimate}
 For all~$h > 0$ and all~$x\in\R^n$, $x^\prime\in\R^n$, there holds
 \[
  \int_{Q^n} \F_\Omega\!\left(\Pi_{h,y}(z)-\Pi_{h,y}(w)\right) \, \d y
  \leq \left(\sqrt{n} + \frac{n}{2}\right) \abs{x - x^\prime} \! .
 \]
\end{lemma}
\begin{proof}
 First, we claim that for all~$x\in\R^n$, $x^\prime\in\R^n$, there holds
 \begin{equation} \label{average-grid-projection}
  \int_{Q^n} \abs{\pi_{h,y}(x) - \pi_{h,y}(x^\prime)} \d y
  \leq \abs{x-x^\prime}_1 = \sum_{i=1}^n \abs{x_i - x_i^\prime}
  \leq \sqrt{n} \abs{x - x^\prime} \! .
 \end{equation}
 When~$n = 1$, the inequality follows from the identity
 \begin{equation} \label{average-grid-projection-1}
  \int_0^1 \abs{\lfloor a-y\rfloor-\lfloor b-y\rfloor} \, \d y = \abs{a - b}
 \end{equation}
 for all~$a\in\R$, $b\in\R$. (Here~$\lfloor t\rfloor$ denotes the integer part of~$t$, that is, the largest integer not exceeding~$t$.)
 For~$n\geq 2$, the estimate~\eqref{average-grid-projection} follows by applying~\eqref{average-grid-projection-1} to each coordinate, and then summing over~$j\in\{1, \, \ldots, \, n\}$.
 In light of~\eqref{average-grid-projection}, it suffices to prove
 \begin{equation} \label{modproj0}
  \F_\Omega\!\left(\Pi_{h,y}(x) - \Pi_{h,y}(x^\prime)\right) \, \d y
  \leq \left(1+\frac{\sqrt n}{2}\right) \abs{\pi_{h,y}(x) - \pi_{h,y}(x^\prime)}
 \end{equation}
 for all~$x\in\R^n$, $x^\prime\in\R^n$ and almost all~$y\in Q^n$.

 Given~$x$, $x^\prime$, let~$y\in Q^n$ be such that~$x\notin R_{n-1}(h, \, y)$, $x^\prime\notin R_{n-1}(h, \, y)$.
 If~$K_{h,y}(x)\subseteq\Omega$ and~$K_{h,y}(x)\subseteq\Omega$, then $\Pi_{h,y}(x) - \Pi_{h,y}(x^\prime) = \lb \pi_{h,y}\rb - \lb \pi_{h,y}(x^\prime)$ and~\eqref{modproj0} follows immediately.
 If~$K_{h,y}(x)\not\subseteq\Omega$ and~$K_{h,y}(x)\not\subseteq\Omega$, then $\Pi_{h,y}(x) - \Pi_{h,y}(x^\prime) = 0$ and there is nothing to prove.
 The only case left to consider is when exactly one of the two cubes, say~$K_{h,y}(x)$, is contained in $\Omega$.
 Choose a point~$q\in K_{h,y}(x^\prime)\setminus\Omega$, and let~$L$ be the oriented line segment from~$q$ to~$\pi_y(x)$.
 Then,
 \[
  \begin{split}
   \F_\Omega\left(\Pi_{h,y}(z)-\Pi_{h,y}(w)\right)
   = \F_\Omega\left(\lb\pi_{h,y}(z)\rb\right)
   \leq \abs{\pi_{h,y}(x) - q} \! .
  \end{split}
 \]
 Since~$q\in K_{h,y}(x^\prime)$ and the diameter of~$K_{h,y}(x^\prime)$ is~$\sqrt n\,h$, we obtain
 \begin{equation} \label{modproj1}
  \begin{split}
   \F_\Omega\left(\Pi_{h,y}(z)-\Pi_{h,y}(w)\right)
   \leq \abs{\pi_{h,y}(x) - \pi_{h,y}(x^\prime)} + \sqrt{n} \, h.
  \end{split}
 \end{equation}
 At the same time, $\pi_{h,y}(x)$ and~$\pi_{h,y}(x^\prime)$ are distinct points of the lattice $hy+h\mathbb Z^n$, because $K_{h,y}(x)\subseteq\Omega$ while
 $K_{h,y}(x^\prime)\not\subseteq\Omega$. Therefore, we have
 \begin{equation} \label{modproj2}
  \begin{split}
   \abs{\pi_{h,y}(x) - \pi_{h,y}(x^\prime)} \geq h.
  \end{split}
 \end{equation}
 Combining~\eqref{modproj1} with~\eqref{modproj2}, we obtain~\eqref{modproj0}.
\end{proof}

\begin{proof}[Proof of Proposition~\ref{prop:deformation}]
 We split the proof into several steps.

 \medskip
 \noindent
 \textit{Proof of~\eqref{Pflat}.}
 Let~$S\in\P_0(\Omega; \, \G)$. We consider a decomposition of~$S$ in the form
 \[
  S = \sum_{i=1}^N \sigma_i \lb x_i\rb
   + \sum_{j=1}^M \tau_j \left( \lb z_j \rb - \lb w_j\rb \right)\qquad \textrm{in } \Omega,
 \]
 where the coefficients~$\sigma_i$, $\tau_j$ belong to~$\Omega$, each~$x_i$ is a point in~$\Omega$, and~$z_j$, $w_j$ are points in~$\overline{\Omega}$.
 For almost every~$y\in Q^n$, none of the points~$x_i$, $z_j$, $w_j$
 belongs to~$R_{n-1}(h, \, y)$.
%  With the convention that a point supported in~$\R^n\setminus\Omega$
%  does not contribute to a chain relative to~$\Omega$,
 Comparing the definitions of~$P$ and~$\Pi_{h,y}$, we obtain
 \[
  P(S, \, h, \, y)
  = \sum_{i=1}^N \sigma_i \, \Pi_y(x_i)
   + \sum_{j=1}^M \tau_j
   \left(\Pi_{h,y}(z_j) - \Pi_{h,y}(w_j)\right) .
 \]
 Therefore, the triangle inequality and Lemma~\ref{lemma:modified-projection-estimate} imply
 \[
  \begin{split}
   \int_{Q^n} \F_\Omega\!\left(P(S, \, h, \, y)\right) \d y
   &\leq \sum_{i=1}^N \int_{Q^n} \M\!\left(\sigma_i \, \Pi_{h,y}(x_i)\right) \d y
   + \sum_{j=1}^M \int_{Q^n} \F_\Omega\!\left(\tau_j
   \Pi_{h,y}(z_j) - \tau_j\Pi_{h,y}(w_j)\right) \d y \\
   &\leq \sum_{i=1}^N \abs{\sigma_i}
   + \left(\sqrt{n} + \frac{n}{2} \right) \sum_{j=1}^M \abs{\tau_j}
     \abs{z_j - w_j} \! .
  \end{split}
 \]
 Taking the infimum over all decompositions of~$S$ as above and recalling
 Remark~\ref{rk:dipolarOmega}, we conclude that~\eqref{Pflat} is satisfied by polyhedral chains.
 As a consequence, for every~$h>0$,
 the homomorphism $S\mapsto P(S, \, h, \, \cdot)$
 is a Lipschitz map from~$\P_0(\Omega; \, \G)$, equipped with the
 $\F_\Omega$-norm, into
 $L^1(Q^n; \, \F_0(\Omega; \, \G))$.
 It therefore extends uniquely to a continuous group homomorphism
 \[
  P(\cdot, \, h, \, \cdot)\colon
  \F_0(\Omega; \, \G)
  \longrightarrow L^1(Q^n; \, \F_0(\Omega; \, \G)),
 \]
 and~\eqref{Pflat} follows by density.
 The extension agrees with~\eqref{deformation} for a.e.~$y\in Q^n$, because the intersection index satisfies continuity properties with respect to~$\F_\Omega$-convergence (see Remark~\ref{rk:Iflat}).

 \medskip
 \noindent
 \textit{Proof of~\eqref{Pflatconv} when~$S$ is polyhedral.}
 We consider a polyhedral chain
 $S = \sum_{i=1}^N \sigma_i\lb x_i\rb$,
 where the~$x_i$'s are distinct.
 We may assume that no~$x_i$ belongs to the skeleton of~$\GG(h, \, y)$,
 since this holds for almost every~$y$.
 Let~$K_{h,y}(x_i)$ be the unique cube of~$\GG(h, \, y)$ containing~$x_i$.
 If~$K_{h,y}(x_i)$, we define~$L_i$ as the oriented line segment from the centre~$\pi_{h,y}(x_i)$ of~$K_i$ to~$x_i$.
 Otherwise, we define~$L_i$ as an oriented line segment from an arbitrary point~$q\in K_{h,y}(x_i)\setminus\Omega$ to~$x_i$.
 In both cases, setting~$B_y := \sum_{i=1}^N \sigma_i \lb L_i\rb$, we have
 $S-P(S, \, h, \, y)=\partial B_y$ in~$\Omega$.
 Since the diameter of every cube of the grid is~$\sqrt n\,h$, we have
 \begin{equation} \label{Pflatmassconv}
  \int_{Q^n}
   \F_\Omega(S-P(S, \, h, \, y))\,\d y
  \leq \sqrt n\,h\,\M(S).
 \end{equation}
 In particular, $S$ satisfies~\eqref{Pflatconv}.

 \medskip
 \noindent
 \textit{Proof of~\eqref{Pflatconv} for arbitrary~$S$.}
 Fix~$S\in\F_0(\Omega; \, \G)$ and let~$S_j\in\P_0(\Omega; \, \G)$ be
 such that~$\F_\Omega(S-S_j)\to 0$ as~$j\to=\infty$.
 By~\eqref{Pflat} and the triangle inequality,
 \[
  \begin{split}
   \int_{Q^n}
    \F_\Omega(S-P(S, \, h, \, y))\,\d y
   &\leq \F_\Omega(S-S_j)
    + \int_{Q^n}
      \F_\Omega(P(S-S_j, \, h, \, y))\,\d y \\
   &\quad
    + \int_{Q^n}
      \F_\Omega(S_j-P(S_j, \, h, \, y))\,\d y \\
   &\leq \left(1+\sqrt n + \frac{n}{2}\right)\F_\Omega(S-S_j)
    + \sqrt n\,h\,\M(S_j).
  \end{split}
 \]
 Taking the limit first as~$h\to0$, then as~$j\to+\infty$,
 proves~\eqref{Pflatconv}.
\end{proof}

We point out a consequence of Proposition~\ref{prop:deformation} which will be useful later.

\begin{corollary} \label{cor:zeroI}
 Let~$S\in \F_0(\Omega; \, \G)$ be such that $\ind(S, \, Q_r^n(x)) = 0$ for almost all~$(x, \, r)\in\Omega\times (0, \, +\infty)$ such that~$Q^n_r(x)\subseteq\Omega$. Then, $S = 0$.
\end{corollary}
\begin{proof}
 By Fubini theorem, there exists a positive sequence~$r_j\to 0$ such that
 $\ind(S, \, Q^n_{r_j}(x))=0$ for almost every $x\in\Omega$ such that~$Q^n_{r_j}(x)\subseteq\Omega$ and every~$j$. Set~$h_j:=r_j$.
 With this choice of~$h_j$, we have~$\ind(S, \, K) = 0$ for all~$K\in\GG(h_j, \, y)$ satisfying~$K\subseteq\Omega$ and almost every~$y\in Q^n$.
 Then, Equation~\eqref{deformation} --- which remains valid for non-polyhedral chains~$S\in\F_0(\Omega; \, \G)$, thanks to Proposition~\ref{prop:deformation} --- implies that $P(S, \, h_j, \, y) = 0$ for a.e.~$y\in Q^n$.
 From~\eqref{Pflatconv}, it follows that~$S=0$.
\end{proof}

\subsection{Zero-dimensional intersection chains}
\label{sect:Ichains}

Let~$\Omega\subseteq\R^n$ be a bounded open set.
Given a polyhedral chain~$S\in\P_0(\Omega; \, \G)$, we define its ``intersection norm'' as
\begin{equation} \label{Inorm}
 \begin{split}
  \I_\Omega(S)
  :=&
   \int_{\QQ_\Omega} \min\left\{1, \,
   \frac{1}{\alpha_{\G}} \abs{\ind(S, \, Q^n_r(x))}\right\} \d r \, \d x \\
  =& \abs{\left\{(x,\, r)\in\QQ_\Omega\colon
   \ind(S, \, Q^n_r(x)) \neq 0\right\}} \! ,
 \end{split}
\end{equation}
where~$\alpha_{\G}$ is the number defined by~\eqref{alpha_G}
and~$\QQ_\Omega$ is defined in~\eqref{Q_Omega}.
Lemma~\ref{lemma:I} immediately implies
\begin{equation} \label{IF}
 \I_\Omega(S) \leq C_{\Omega,\G} \, \F_\Omega(S)
\end{equation}
for all~$S\in\P_0(\Omega; \, \G)$ and some constant~$C_{\Omega,\G}$ depending only on~$\Omega$ and~$\G$, namely, $C_{\Omega,\G} := \frac{1}{\alpha_{\G}}\max(\abs{\QQ_\Omega}, \, 2\abs{\Omega})$.

\begin{lemma} \label{lemma:Iisanorm}
 The function~$\I_\Omega$ is a group norm on $\P_0(\Omega; \, \G)$.
\end{lemma}
\begin{proof}
 Nonnegativity and symmetry are immediate.
 The function~$t\mapsto\min\{1, \, t\}$ is nondecreasing and subadditive.
 Therefore, for all~$\alpha > 0$ and~$\sigma_1\in\G$, $\sigma_2\in\G$, we have
 \[
 \min\{1, \, \alpha^{-1}|\sigma_1 + \sigma_2|\}
 \leq \min\{1, \, \alpha^{-1}|\sigma_1| + \alpha^{-1}|\sigma_2|\}
 \leq \min\{1, \, \alpha^{-1}|\sigma_1|\} + \min\{1, \, \alpha^{-1}|\sigma_2|\}.
 \]
 The triangle inequality follows after integration.
 It remains to prove that~$\I(S) \neq 0$ if~$S\neq 0$.
 Let $S\neq 0$ and choose a point~$x_0\in\spt S\subseteq\Omega$.
 Since $S$ has finite support there is a set of positive measure
 of pairs $(x, \, r)$ for which~$Q_r^n(x)$ is contained in~$\Omega$,
 contains $x_0$ and no other atom of $S$.
 For such~$(x, \, r)$ we have~$\ind(S, Q^n_r(x))\neq 0$, hence~$\I(S) > 0$.
\end{proof}

We define~$\I_0(\Omega; \, \G)$ as the metric space completion of~$\P_0(\Omega; \, \G)$ with respect to distance induced by the~$\I_\Omega$-norm.
The group operation on polyhedral chains extends by Lipschitz continuity
to~$\I_0(\Omega; \, \G)$, making it a complete normed Abelian group.
Because of the inequality~\eqref{IF}, the identity on polyhedral chains extends by continuity to a group homomorphism~$\iota\colon\F_0(\Omega; \, \G)\to\I_0(\Omega; \, \G)$.
Moreover, Corollary~\ref{cor:zeroI} implies that~$\iota$ is injective.
(In Proposition~\ref{prop:flat-vs-I} below, we will see that~$\iota$ is never surjective, unless~$n = 1$ \emph{and} $\G$ is bounded.)
From now on, we will identify~$\F_0(\Omega; \, \G)$ with its image under the map~$\iota$, by abuse of notation.

The following lemma characterises Cauchy sequences (of polyhedral chains) with respect to the~$\I_\Omega$-norm.

\begin{lemma} \label{lemma:ICauchy}
 Let~$(S_j)_{j\in\N}$ be a sequence in~$\P_0(\Omega; \, \G)$.
 Then, $(S_j)_{j\in\N}$ is a Cauchy sequence with respect to the~$\I_\Omega$-norm if and only if there exists a measurable function~$I\colon\QQ_\Omega\to\G$ such that
 \begin{equation} \label{ICauchy}
  \abs{\left\{(x, \, r)\in\QQ_\Omega\colon
  \ind(S_j, \, Q_r^n(x)) \neq I(x, \, r)\right\}} \to 0
 \end{equation}
 as~$j\to +\infty$.
\end{lemma}
\begin{proof}
 Let~$I_j(x, \, r) := \ind(S_j, \, Q^n_r(x))$.
 By~\eqref{Inorm}, $(S_j)_{j\in\N}$ is a Cauchy sequence with respect to the~$\I_\Omega$-norm if and only if
 \begin{equation} \label{ICauchybis}
  \abs{\{(x, \, r)\in\QQ_\Omega\colon I_j(x, \, r)\neq I_k(x, \, r)\}}\to 0 \qquad \textrm{as } j, \, k\to+\infty.
 \end{equation}
 If there exists a function~$I\colon\QQ_\Omega\to\G$ satisfying~\eqref{ICauchy}, then~\eqref{ICauchybis} follows immediately.
 Conversely, suppose~\eqref{ICauchybis} holds.
 Then, we can extract a subsequence~$j_m\to+\infty$ such that the sets $E_m := \{I_{j_m} \neq I_{j_{m+1}}\}\subseteq\QQ_\Omega$ satisfy~$\abs{E_m} \leq 2^{-m}$.
 The set~$\limsup_{m\to+\infty} E_m := \bigcap_{M\geq 1} \bigcup_{m\geq M} E_m$ is negligible by Borel-Cantelli lemma, and the sequence~$(I_{j_m}(x, \, r))_{m\in\N}$ is eventually constant for all~$(x,\, r)\notin E$.
 As a consequence, $(I_{j_m})_{m\in\N}$ converges pointwise almost everywhere to a measurable function~$I\colon\QQ_\Omega\to\G$.
 Now, let~$\eps > 0$. Since we have assumed that~\eqref{ICauchybis} holds, there exists an integer $N \geq 1$ such that~$\abs{\{I_j\neq I_k\}}\leq\eps$ for all~$j\geq N$, $k\geq N$. Given~$j\geq N$, for large enough~$m$ we have
 \[
  \abs{\{I_j\neq I\}}
  \leq \abs{\{I_{j}\neq I_{j_m}\}} + \abs{\{I_{j_m}\neq I\}}
  \leq \eps + \abs{\{I_{j_m}\neq I\}}
 \]
 However, $\{I_{j_m}\neq I\}\to 0$ as~$m\to+\infty$, by dominated convergence.
%  This implies $\limsup_{j\to+\infty} \abs{\{I_j\neq I\}} \leq \eps$ and, since~$\eps > 0$ is arbitrary, the
 Then, taking the limit (superior) first as~$j\to+\infty$, then as~$\eps\to 0$, the claim~\eqref{ICauchy} follows.
\end{proof}

Lemma~\ref{lemma:ICauchy} implies that for all~$S\in\I_0(\Omega; \, \G)$ and \emph{almost all} $(x, \, r)\in\QQ_\Omega$, the intersection index of~$S$ and~$Q^n_r(x)$ is well-defined as~$\ind(S, \, Q^n_r(x)) := I(x, \, r)$, where~$I$ is the (essentially unique) function satisfying~\eqref{ICauchy} for a Cauchy sequence~$(S_j)_{j\in\N}$ of polyhedral chains approximating~$S$.
The function~$I$ does not depend on the choice of~$(S_j)_{j\in\N}$, as two sequences converging to the same limit~$S$ can be re-structured as a single sequence that still converges to~$S$.
Therefore, intersection chains have well-defined intersection indices on \emph{almost} all cubes contained in~$\Omega$ with sides parallel to the coordinate hyperplanes.
The expression~\eqref{Inorm} for the $\I_\Omega$-norm remains valid for a general~$S\in\I_0(\Omega; \, \G)$, by continuity.

\begin{remark} \label{rk:indconvergesae}
 The arguments in the proof of Lemma~\ref{lemma:ICauchy} also imply the following statement, as a byproduct.
 Given a sequence~$(S_j)_{j\in\N}\subseteq\I_0(\Omega; \, \G)$ that converges to~$S\in\I_0(\Omega; \, \G)$ with respect to the~$\I_\Omega$-norm, there exists a subsequence~$j_m\to+\infty$ with the following property: for almost every~$(x, \, r)\in\QQ_\Omega$, there is an index~$M(x, \, r)$ such that for all~$m\geq M(x, \, r)$, we have $\ind(S_{j_m}, \, Q^n_r(x)) = \ind(S, \, Q^n_r(x))$.
%  Equation~\eqref{deformation} shows that, along the same subsequence, for almost every~$(h, \, y)$ we have $P(S_{j_m}, \, h, \, y) = P(S, \, h, \,  y)$ if~$m$ is larger than some index~$M(h, \, y)$.
\end{remark}

\begin{remark} \label{rk:index-subdivision}
 Let~$K_1, \, \ldots, \, K_{q^n}$ be a subdivision of~$Q^n_r(x)$ into~$q^n$ cubes, with~$q\geq 1$ an integer. More precisely, each~$K_i$ is a cube of the form~$Q^n_{r/q}(x_i)$ for some~$x_i\in Q^n_r(x)$, contained in~$Q^n_r(x)$, and with pairwise disjoint interiors.
 Then, for all~$S\in\I_0(\Omega; \, \G)$ and almost all~$(x, \, r)\in\QQ_\Omega$, we have
 \begin{equation} \label{index-subdivision}
  \ind(S, \, Q^n_r(x)) =
  \sum_{i = 1}^{q^n} \ind(S, \, K_i).
 \end{equation}
 Equation~\eqref{index-subdivision} is an immediate consequence of~\eqref{intersection} in case~$S$ is polyhedral; the general case follows by a density argument, using Remark~\ref{rk:indconvergesae}.
\end{remark}

The following observation is elementary, but will be of crucial importance later on.

\begin{lemma} \label{lemma:Iae}
 Let $(S_j)_{j\in\N}$ be a sequence in~$\F_0(\Omega; \, \G)$.
 Suppose that, for almost every $(x, \, r)\in\QQ_\Omega$,
 there exists an integer~$j_*(x, \, r)$ such that~$\ind(S_j, \, Q_r^n(x)) = \ind(S_k, \, Q_r^n(x))$ for all~$j\geq j_*(x, \, r)$, $k\geq j_*(x, \, r)$.
 Then, $(S_j)_{j\in\N}$ is a Cauchy sequence in~$\I_{\Omega}$.
\end{lemma}
\begin{proof}
 The lemma is a direct consequence of Lebesgue's dominated convergence theorem, applied to~\eqref{Inorm}.
\end{proof}

\paragraph*{Examples.}

The following examples illustrate some of the features of the~$\I_\Omega$-norm.

\begin{example}[$\I_\Omega$-norm of a dipole in~$\R^n$]
\label{ex:Idipole}
 Let~$a$, $b$ be given points in~$\Omega$, let~$\sigma\in\G\setminus\{0\}$, and let~$D := \sigma\,\lb b \rb - \sigma\,\lb a \rb$.
 Let~$r_\Omega(a, \, b)$ be the supremum of all~$r > 0$ such that~$Q^n_r(a) \cup Q^n_r(b)\subseteq\Omega$.
 If the points~$a$, $b$ satisfy~$\abs{b - a} \leq r_\Omega(a, \, b)/2$, then we have
 \begin{equation} \label{Idipole}
  C_n \, r_\Omega(a, \, b)^n  \abs{b - a}
  \leq \I_{\Omega}(D)
  \leq C_\Omega \, \abs{a - b} \! ,
 \end{equation}
 where~$C_n$ is a constant depending on~$n$ only and~$C_\Omega := \max(\abs{\QQ_\Omega}, \, 2\abs{\Omega})$.
 The upper estimate follows from~\eqref{IF}, upon observing that~$\I_\Omega(D)$ does not depend on the choice of the multiplicity~$\sigma\in\G\setminus\{0\}$.
 Indeed, $\min\{1,\frac{1}{\alpha_{\G}}\abs{\ind(D, \, Q^n_r(x))}\} = 1$ whenever $\ind(D, \, Q^n_r(x))\neq 0$.
 Therefore, applying~\eqref{IF} to a dipole~$D^\prime = \tau\lb b \rb - \tau\lb a\rb$ for an arbitrary nonzero multiplicity~$\tau$, we obtain $\I_{\Omega}(D) = \I_{\Omega}(D^\prime) \leq C_\Omega \abs{a - b}$.

 For the lower estimate, we observe that $\ind(D, \, Q^n_r(x)) \neq 0$ if and only if~$x$ belongs to the symmetric difference $Q^n_r(a)\triangle Q^n_r(b) := (Q^n_r(a)\setminus Q^n_r(b)) \cup (Q^n_r(b)\setminus Q^n_r(a))$.
 Moreover, we have
 \[
  \begin{split}
   \abs{Q^n_r(a)\triangle Q^n_r(b)}
   \geq 2 r^{n-1} \min\{\abs{b - a}_\infty, \, r\}
   \geq \frac{2r^{n-1}}{\sqrt n} \min\{\abs{b - a}, \, r\},
  \end{split}
 \]
 where~$\abs{b - a}_\infty := \max_{1 \leq i \leq n} \abs{b_i - a_i}$.
 If~$r_\Omega(a, \, b)/4 < r < r_\Omega(a, \, b)/2$ and $x\in Q^n_r(a)\triangle Q^n_r(b)$, then $Q^n_r(x)\subseteq Q^n_{2r}(a)$ or $Q^n_r(x)\subseteq Q^n_{2r}(b)$, and hence $Q^n_r(x)\subseteq\Omega$.
 Moreover, since $r>r_\Omega(a, \, b)/4$ and $\abs{b-a}\leq r_\Omega(a, \, b)/2$, we have $\min\{\abs{b-a}, \, r\}\geq \abs{b-a}/2$. Then, we obtain
 \[
  \int_{\{x\in\Omega\colon Q^n_r(x)\subseteq\Omega\}}
  \min\!\left\{1, \, \frac{1}{\alpha_{\G}} \abs{\ind(D, \, Q^n_r(x))}\right\} \d x
  \geq \frac{r^{n-1}}{\sqrt n} \abs{b - a} \! .
 \]
 The lower estimate in~\eqref{Idipole} follows by integrating with respect to $r\in (r_\Omega(a, \, b)/4, \, r_\Omega(a, \, b)/2)$.
\end{example}

\begin{example}[$\I_\Omega$-norm of a quadrupole]
\label{ex:Iquadrupole}
 We denote by~$e_j := (0, \, \ldots, \, 0, \, 1, \, 0, \, \ldots, \, 0)$ the~$j$-th element of the canonical basis of~$\R^n$.
 Assume that $n\geq 2$. Let $a\in\Omega$, let~$\sigma\in\G\setminus\{0\}$,
 let~$R>0$ be such that $a+([-R, \, 1+R]^2\times[-R, \, R]^{n-2})\subseteq\Omega$.
 Then, for~$0<\rho<\min\{1, \, R/2\}$, the ``quadrupole'' of size~$\rho$,
 \begin{equation} \label{Q_quadrupole}
  \begin{split}
  Q :=
    \sigma \lb a + \rho \, e_1 + \rho \, e_2\rb
  - \sigma \lb a + \rho \, e_1\rb
  - \sigma \lb a + \rho \, e_2\rb
  + \sigma \lb a\rb ,
  \end{split}
 \end{equation}
 satisfies
 \begin{equation}
  \label{Iquadrupole-two-sided}
  c_n \, R^{n-1} \, \rho^2
  \leq \I_\Omega(Q)
  \leq C_{\Omega,n}\, \rho^2
 \end{equation}
 for some constants~$c_n$ (depending on~$n$ only) and~$C_{\Omega, n}$ (depending on~$\Omega$, $n$ only).
 This is in sharp contrast with the behaviour of the flat norm, which satisfies $\F_\Omega(Q)= 2\abs{\sigma} \rho$ for small enough~$\rho$ (see Remark~\ref{rk:localflat}).
 In order to prove~\eqref{Iquadrupole-two-sided}, we consider the one-dimensional dipole (with integer coefficients)
 \[
  D:=\lb \rho \rb - \lb 0 \rb\in\P_0(\R; \, \Z).
 \]
 Writing $x=(x_1, \, \ldots, \, x_n)$, $a=(a_1, \, \ldots, \, a_n)$, we have
 \begin{equation} \label{Iquadrupole-factorisation}
 \begin{split}
  \ind(Q, \, Q_r^n(x))
  ={}&
  \ind\!\left(D, \, \left[x_1-a_1-\frac{r}{2}, \, x_1-a_1+\frac{r}{2}\right]\right) \\
  &\hspace{7mm} \cdot
  \ind\!\left(D, \, \left[x_2-a_2-\frac{r}{2}, \, x_2-a_2+\frac{r}{2}\right]\right)
  \cdot \left(\prod_{i=3}^n \one_{\{|x_i-a_i|\leq r/2\}} \right) \sigma.
 \end{split}
 \end{equation}
 Here, $\one_{\{|x_i-a_i|\leq r/2\}}$ is defined to be equal to either one or zero, depending on whether~$x_i$ satisfies the displayed inequality or not.
 All the factors in the right-hand side of~\eqref{Iquadrupole-factorisation} but the last one are integer-valued, so the product is well-defined.
 For $0<r<\rho$, the set of $t\in\R$ such that
 $\ind(D, \, [t-r/2, \, t+r/2])\neq 0$
 has measure $2r$, while for $r>\rho$ it has measure $2\rho$.
 Therefore, keeping~\eqref{Iquadrupole-factorisation} into account, we deduce
 \begin{equation*} %\label{quadrupole-detection-measure}
  \abs{\left\{x\in\R^n\colon \ind(Q, \, Q_r^n(x))\neq 0\right\}}
  =
  \begin{cases}
   4r^n    &\textrm{if } 0<r<\rho,\\[2mm]
   4\rho^2r^{n-2} &\textrm{if } r>\rho.
  \end{cases}
 \end{equation*}
 If $Q_r^n(x)\subseteq\Omega$, then $r\leq \mathrm{diam}\,\Omega$.
 As a consequence, we have
 %using \eqref{quadrupole-detection-measure}, we obtain
 \[
 \begin{split}
  \I_\Omega(Q)
  &\leq \int_0^{\rho} 4r^n\,\d r
  + \int_{\rho}^{\mathrm{diam}\,\Omega} 4\rho^2r^{n-2}\,\d r
  \leq C_{\Omega,n}\rho^2.
 \end{split}
 \]
 On the other hand, if $R/2< r <R$ and~$\rho < R/2$,
 then~$r>\rho$, and every cube~$Q^n_r(x)$ such that~$\ind(Q, \, Q^n_r(x))\neq 0$
 is contained in $a+([-R, \, 1+R]^2\times[-R, \, R]^{n-2})\subseteq\Omega$.
 This implies
 \[
 \begin{split}
  \I_\Omega(Q)
  &\geq 4\rho^2 \int_{R/2}^{R} r^{n-2}\,\d r
  \geq c_n R^{n-1}\rho^2,
 \end{split}
 \]
 which proves~\eqref{Iquadrupole-two-sided}.
 Notice that the upper bound in~\eqref{Iquadrupole-two-sided} does not require the assumptions $a+([-R, \, 1+R]^2\times[-R, \, R]^{n-2})\subseteq\Omega$ and~$\rho<\min\{1, \, R/2\}$: it holds for any quadrupole of the form~\eqref{Q_quadrupole} whose four vertices belong to~$\Omega$.
\end{example}

\begin{example}[$\I_\Omega$-norm of a $d$-pole in~$\R^n$]
\label{ex:Inpole}
 Example~\ref{ex:Iquadrupole} extends to higher dimensions.
 Let~$d$ be an integer, with~$1 \leq d \leq n$.
 Assume that~$\Omega\supseteq[-R, \, 1+R]^d\times[-R, \, R]^{n-d}$
 for some~$R>0$. For~$0<\rho<\min\{1, \, R/2\}$ and
 $\sigma\in\G\setminus\{0\}$, we consider a ``$d$-pole'' of size~$\rho$,
 \[
  P:=\sum_{x\in\{0, \, \rho\}^d\times\{0\}^{n-d}}\eps(x)\, \lb x\rb,
 \]
 where the multiplicity~$\eps(x)$ is defined as~$\eps(x) := \sigma$ if an even number of components of~$x$ are equal to zero, and
 $\eps(x) := -\sigma$ otherwise.
 The equality~\eqref{Iquadrupole-factorisation} generalises as
 \[
  \ind(P, \, Q^n_r(x))
  = \left(\prod_{i=1}^d \ind\!\left(D, \, \left[x_i-\frac{r}{2}, \, x_i+\frac{r}{2}\right]\right) \right)
  \cdot\left(\prod_{i=d+1}^n \mathbf{1}_{\{|x_i|\leq r/2\}}
  \right) \sigma.
 \]
 Then, by reasoning exactly as in Example~\ref{ex:Iquadrupole}, we deduce that~$\I_\Omega(P)$ is bounded from above and below by terms that scale as~$\rho^d$.
\end{example}

\begin{example}[Non-compatibility with weak$^*$ convergence]
\label{ex:noncompatible}
 When~$\G = \Z$, convergence with respect to the~$\I_\Omega$-norm is \emph{not} compatible with weak$^*$ convergence in the sense of distributions.
 For instance, take~$n = 1$, $\Omega = (-1, 1)$, and let~$S_j := j \lb 1/j \rb - j\lb 0 \rb$ for all integer~$j\geq 2$.
 For all~$\varphi\in C^\infty_{\mathrm{c}}(-1, \, 1)$, we have
 \[
  \begin{split}
   \langle S_j, \, \varphi\rangle
   = j \varphi(1/j) - j \varphi(0)
   = j \int_{0}^{1/j} \varphi^\prime(t) \, \d t
   \to \varphi^\prime(0)
   \qquad \textrm{as }  j\to+\infty.
  \end{split}
 \]
 In other words, $S_j$ converges weakly$^*$ in the sense of distributions to~$-\delta_0^\prime$, where~$\delta_0$ is Dirac's delta at~$t = 0$.
 However, the estimate~\eqref{Idipole} implies~$\I_\Omega(S_j)\leq C_\Omega/j$, so~$S_j$ converges to zero with respect to the~$\I_\Omega$-norm.
 In Lemma~\ref{lemma:Iandweak} below, we will provide a sufficient condition for a sequence~$(S_j)_{j\in\N}\subseteq\F_0(\Omega; \, \Z)$ to have the same limit in the sense of the~$\I_\Omega$-norm and weak$^*$ as distributions.
\end{example}

\begin{example} \label{ex:nonHarrison}
 Let~$\Omega = (-1, \, 1)$, $\G = \Z$, and~$0 < \rho < 1/4$.
 We consider the chain
 \[
  S = \lb 3\rho\rb -\lb 2\rho\rb - \lb \rho\rb + \lb 0\rb
  \in \P_0((-1, \, 1); \, \Z).
 \]
 Its~$\I_{(-1, \, 1)}$-norm can be computed using the change of variable~$a = x - r/2$, $b = x + r/2$.
 Setting~$F(t) := \ind(S, \, [-1, \, t)[)$, we have
 \[
  \begin{split}
   \I_{(-1, \, 1)}(S)
   &= \abs{\{(x, \, r)\in\R\times\R\colon
    -1 < x - \frac{r}{2} < x + \frac{r}{2} < 1,
    \ \ind\!\left(S, \left[x - \frac{r}{2}, \, x+\frac{r}{2}\right] \right) \neq 0\}}  \\
   &= \abs{\{(a, \, b)\in \R\times\R\colon
    -1 < a < b < 1, \ F(b) \neq F(a)\}}
%    &= \frac{1}{2} \left(4
%    - \sum_{k\in\Z}\abs{\{x\in (-1, \, 1)\colon F(x) = k\}}^2\right)
%    = 2\rho - \frac{3}{2}\rho^2.
  \end{split}
 \]
 Now, the function~$F$ is equal to~$1$ on the interval~$(0, \, \rho)$, to~$-1$ in~$(2\rho, \, 3\rho)$, and is zero elsewhere.
 Since the condition~$F(b)\neq F(a)$ is symmetric in~$(a, \, b)$, we have
 \[
  \begin{split}
   \I_{(-1, \, 1)}(S)
   &= \frac{1}{2} \abs{\{(a, \, b)\in (-1, \, 1)^2\colon F(b) \neq F(a)\}} \\
   &= \frac{1}{2} \left(4 - \rho^2 - \rho^2 - (2 - 2\rho)^2\right)
   = 4\rho - 3\rho^2.
  \end{split}
 \]
 In particular, $\I_{(-1, \, 1)}(S)\sim 4\rho$ as~$\rho\to 0$.
 On the other hand, for all~$\varphi\in C^2_{\mathrm{c}}(-1, \, 1))$
 there exist points~$\xi_1\in (0, \, \rho)$, $\xi_2\in (0, \, 2\rho)$, $\xi_3\in (0, \, 3\rho)$ such that
 \[
  \begin{split}
   \langle S, \, \varphi\rangle
   = \frac{\rho^2}{2}\left(9\varphi''(\xi_3) - 4\varphi''(\xi_2) - \varphi''(\xi_1)\right) \leq 7\rho^2 \norm{\varphi}_{C^2(-1, \, 1)}\! .
  \end{split}
 \]
 Therefore, the norm of~$S$ as an element of~$C^k_{\mathrm{c}}(-1, \, 1)^\prime$, for~$k\geq 2$, is quadratic in~$\rho$ as~$\rho\to 0$.
\end{example}

The examples presented here show that the topology induced by the norm $\I_\Omega$ is genuinely different from other well-known topologies on chains.
First, the $\I_\Omega$-norm is not equivalent to the flat topology: indeed, for the quadrupole~$Q$ considered in Remark~\ref{rk:localflat} and Example~\ref{ex:Iquadrupole}, the $\I_\Omega$-norm grows quadratically in the size~$\rho$ as~$\rho\to 0$, while~$\F_\Omega(Q)$ is a linear function of~$\rho$.
In fact, we will see in Proposition~\ref{prop:flat-vs-I} that the inclusion $\F_0(\Omega; \, \G)\subseteq\I_0(\Omega; \, \G)$ is in general strict.
The~$\I_\Omega$-norm does not coincide with Whitney's sharp norm~\cite[Chapter~V]{Whitney-GIT}, which coincides with the flat norm on zero-dimensional chains.
Nor does~$\I_\Omega$ agree with the higher-order norms introduced by Harrison~\cite{Harrison, Harrison-Isomorphism}, which are defined in case~$\G = \Z$ and induce a duality between chains and differential forms with coefficients in~$C^{k,\alpha}$; see Example~\ref{ex:nonHarrison}.
In fact, Example~\ref{ex:noncompatible} shows that the $\I_\Omega$-topology is not compatible with any norm that implies weak$^*$ convergence in the sense distributions.
Finally, the present construction appears to be quite different from Hardt and Rivi\`ere scans~\cite{HardtRiviere, DePauwHardt}, which in the zero-dimensional case coincide with rectifiable chains --- in our setting, polyhedral ones --- while~$\I_0(\Omega; \, \G)$ is in general strictly larger even than the space of flat $0$-chains.

\paragraph{Projections onto a grid.}

As we have seen, Lemma~\ref{lemma:ICauchy} implies that a chain~$S\in\I_0(\Omega; \, \G)$ has well-defined intersection indices on almost each cube~$Q^n_r(x)\subseteq\Omega$.
As a consequence, for \emph{almost all} $(h, \, y)\in (0, \, +\infty)\times Q^n$, the projection~$P(S, \, h, \, y)$ on the grid~$\GG(h, \, y)$ is well-defined, as the right-hand side of Equation~\eqref{deformation} makes sense.
(If~$h$ is too large for any cube of the grid~$\GG(h, \, y)$ to be contained in~$\Omega$, then~$P(S, \, h, \, y) = 0$.)
The projection operator is continuous in a sense made precise by the following lemma.

\begin{lemma} \label{lemma:Ideformation}
 Let $(S_j)_{j\in\N}\subset I_0(\Omega;G)$ and $S\in I_0(\Omega;G)$ be such that
 $\I_\Omega(S_j-S)\to 0$ as $j\to+\infty$.
 Then
 \[
  \abs{\left\{(h, \, y)\in(0,+\infty)\times Q^n\colon
  P(S_j, \, h, \, y)\neq P(S, \, h, \, y)\right\}}\to 0
  \qquad\text{as } j\to+\infty.
 \]
\end{lemma}
\begin{proof}
 Set
 \[
  E_j := \left\{(x, \, r)\in\QQ_\Omega\colon
  \ind(S_j, \, Q_r^n(x))\neq\ind(S, \, Q_r^n(x))\right\}.
 \]
 By Lemma~\ref{lemma:ICauchy}, $\abs{E_j}\to 0$ as~$j\to+\infty$.
 Fix $\delta > 0$. If $h\geq\delta$ and
 $P(S_j, \, h, \, y)\neq P(S, \, h, \, y)$, then there exists $z\in\Z^n$ such that~$\left(h(y+z), \, h\right)\in E_j$.
 Using the change of variables $x=h(y+z)$, $r=h$,
 whose Jacobian is $h^n$, and summing over $z\in\Z^n$, we obtain
 \begin{equation*}
  \abs{\left\{(h, \, y)\in(0,+\infty)\times Q^n\colon h\geq\delta, \
  P(S_j, \, h, \, y)\neq P(S, \, h, \, y)\right\}}
  \leq \delta^{-n}\abs{E_j}.
 \end{equation*}
 On the other hand, the set where $0< h <\delta$ has measure~$\delta$.
 Therefore
 \[
  \limsup_{j\to+\infty}
  \abs{\left\{(h, \, y)\in(0,+\infty)\times Q^n\colon
  P(S_j,h,y)\neq P(S,h,y)\right\}}
  \leq \delta.
 \]
 Letting $\delta\downarrow 0$ concludes the proof.
\end{proof}

%
% \begin{goal} \label{goal:Ideformation}
%  {\BBB For all~$S\in\I_0(\Omega; \, \G)$ and all sequence of positive numbers~$(h_k)_{k\in\N}$ such that~$h_k\to 0$, there exists a (non-relabelled) subsequence such that
%  \[
%   \I_{\Omega}\!\left(S - P(S, \, h_k, \, y)\right) \to 0
%  \]
%  as~$k\to+\infty$, pointwise~a.e.~in~$y\in Q^n$.}
% %  {\BBB For all~$S\in \I_0(\Omega; \, \G)$ and almost all~$h\in (0, \, 1)$, there holds
% %  \begin{equation} \label{PI}
% %   \int_{Q^n} \I_{\Omega}\!\left(P(S, \, h, \, y)\right)
% %     \d y \leq C \I_{\Omega}(S)
% %  \end{equation}
% %  for some constant~$C$ that depends only on~$\Omega$, $\G$.
% %  Moreover, we have
% %  \begin{equation} \label{PIconv}
% %   \int_{Q^n} \I_{\Omega}\!\left(S - P(S, \, h, \, y)\right)
% %     \d y \to 0 \qquad \textrm{as } h\to 0
% %  \end{equation}
% %  for all~$S\in\I_0(\Omega; \, \G)$.}
% \end{goal}
%
% {\BBB The property~\eqref{PIconv} follows from~\eqref{PI} and~\eqref{Pflatconv} via a density argument, exactly as in the proof of Proposition~\ref{prop:deformation}. Therefore, it suffices to prove~\eqref{PI}.

\paragraph*{The relaxed flat norm.}

Recall that we have identified $\F_0(\Omega; \, \G)$ with a subgroup of~$\I_0(\Omega; \, \G)$.
We also extend~$\F_\Omega$ to a function~$\I_0(\Omega; \, \G)\to [0, \, +\infty]$ by setting~$\F_\Omega(S) := +\infty$ for all~$S\in\I_0(\Omega; \, \G)\setminus\F_0(\Omega; \, \G)$.
A natural question is whether the flat-norm~$\F_\Omega$ is lower semicontinuous with respect to~$\I_{\Omega}$.
As it turns out, the answer is positive when~$n = 1$ (see Proposition~\ref{prop:Frelax1} below), but \emph{negative} for~$n\geq 2$ (see Proposition~\ref{prop:Frelax2}).
However, since lower semicontinuity is important for our purposes, we are led to consider the lower semi-continuous envelope~$\overline{\F}_\Omega$ of~$\F_\Omega$, that is,
\begin{equation} \label{relaxedflat}
 \begin{split}
  \overline{\F}_\Omega(S) :=
  \inf\left\{\liminf_{j\to+\infty} \F_\Omega(S_j)\colon (S_j)_{j\in\N}\subseteq\F_0(\Omega; \, \G), \ \I_\Omega(S_j - S)\to 0 \ \textrm{ as } j\to+\infty \right\}
 \end{split}
\end{equation}
for all~$S\in\I_0(\Omega; \, \G)$.
It can be easily checked that~$\overline{\F}_\Omega$ satisfies all the definying properties of a norm on the group~$\{S\in\I_0(\Omega; \, \G)\colon\overline{\F}_\Omega(S) < +\infty\}$.
Therefore, we call~$\overline{\F}_\Omega$ the relaxed flat norm.
A diagonal argument shows that~$\overline{\F}_\Omega$ is lower semicontinuosu with respect to the~$\I_\Omega$-norm.
In fact, $\overline{\F}_\Omega$ is the largest lower semicontinuous function that satisfies~$\overline{\F}_\Omega\leq\F_\Omega$.
By density of polyhedral chains in~$\F_0(\Omega; \, \G)$, we can equivalently write
\begin{equation} \label{relaxedflatpoly}
 \begin{split}
  \overline{\F}_\Omega(S) =
  \inf\left\{\liminf_{j\to+\infty} \F_\Omega(S_j)\colon (S_j)_{j\in\N}\subseteq\P_0(\Omega; \, \G), \ \I_\Omega(S_j - S)\to 0 \ \textrm{ as } j\to+\infty \right\}
 \end{split}
\end{equation}
for all~$S\in\I_0(\Omega; \, \G)$.
Moreover, another diagonal argument shows that the infima in~\eqref{relaxedflat}, \eqref{relaxedflatpoly} are actually minima.

Recall that a chain~$S\in\I_0(\Omega; \, \G)$ has well-defined projections~$P(S, \, h, \, y)$ on a grid~$\GG(h,\ , y)$ for \emph{almost} all~$(h, \, y)\in(0, \, +\infty)\times Q^n$, thanks to Lemma~\ref{lemma:ICauchy}.
As it turns out, the relaxed flat norm of~$S$ is still enough to control the averaged flat norm of its projections.

\begin{lemma} \label{lemma:relaxeddeformation}
 For every $S\in\I_0(\Omega; \, \G)$ and for \emph{almost}
 every $h > 0$, there holds
 \begin{equation} \label{relaxeddeformation}
  \int_{Q^n} \F_\Omega(P(S, \, h, \, y)) \, \d y
  \leq \left(\sqrt{n} + \frac{n}{2}\right) \overline{\F}_\Omega(S).
 \end{equation}
\end{lemma}
\begin{proof}
 We can assume without loss of generality that~$\overline{\F}_\Omega(S) < +\infty$, for otherwise there is nothing to prove.
 Using the characterisation~\eqref{relaxedflatpoly}, we find a sequence of polyhedral chains~$(S_j)_{j\in\N}$ that satisfies~$\I_\Omega(S_j - S)\to 0$ and~$\overline{\F}_\Omega(S_j)\to\overline{\F}_\Omega(S)$ as~$j\to+\infty$.
 The estimate~\eqref{Pflat} in Proposition~\ref{prop:deformation} implies that
 \begin{equation} \label{reldef1}
  \int_{Q^n} \F_\Omega(P(S_j, \, h, \, y)) \, \d y
  \leq \left(\sqrt{n} + \frac{n}{2}\right) \F_\Omega(S_j)
 \end{equation}
 for all~$j$ and all~$h\in (0, \, 2R_*)$.
 The continuity estimate for the projection operator~$P$ (that is, Lemma~\ref{lemma:Ideformation}) states that the set
 \[
  E_j := \left\{(h, \, y)\in (0, \, +\infty)\times
   Q^n\colon P(S_j, \, h, \, y)\neq P(S, \, h, \, y) \right\}
 \]
 satisfies~$\abs{E_j}\to 0$ as~$j\to+\infty$.
 We extract a (non-relabelled) subsequence in such a way that~$\abs{E_j}\leq 2^{-j}$ and observe that~$\limsup_{j\to+\infty} E_j := \bigcap_{J\geq 1}\bigcup_{j\geq J} E_j$ is negligible, by Borel-Cantelli lemma.
 As a consequence, for almost every~$(h, \, y)$, the sequence~$(P(S_j, \, h, \, y))_{j\in\N}$ is eventually constant.
 Then, Fatou lemma and the estimate~\eqref{Pflat} in Proposition~\ref{prop:deformation} imply
 \begin{equation*}
  \begin{split}
  \int_{Q^n} \F_\Omega(P(S, \, h, \, y)) \, \d y
  &= \int_{Q^n} \lim_{j\to+\infty} \F_\Omega(P(S_j, \, h, \, y))  \, \d y \\
  &\leq \liminf_{j\to+\infty}
   \int_{Q^n} \F_\Omega(P(S_j, \, h, \, y))  \, \d y \\
  &\leq \left(\sqrt{n} + \frac{n}{2}\right) \lim_{j\to+\infty} \F_\Omega(S_j)
  = \left(\sqrt{n} + \frac{n}{2}\right)
   \lim_{j\to+\infty} \F_\Omega(S_j)
  \end{split}
 \end{equation*}
 for almost every~$h$.
\end{proof}

We point out a couple of consequences of Lemma~\ref{lemma:relaxeddeformation}.
The first a technical estimate, which will be useful later on.
We use the notation
\begin{equation} \label{Omega_r}
 \Omega_r :=
 \left\{x\in\Omega\colon (x, \, r)\in\QQ_\Omega \right\}
 = \left\{x\in\Omega\colon Q_r^n(x)\subseteq\Omega \right\}
\end{equation}
for~$r > 0$.

\begin{corollary} \label{cor:fixed-scale-index}
 Let~$S\in\I_0(\Omega; \, \G)$ satisfy~$\overline{\F}_\Omega(S)<+\infty$.
 Then, for almost every~$r\in(0, \, 1)$, we have
 \begin{equation} \label{fixed-scale-index}
  \int_{\Omega_r} \abs{\ind(S, \, Q_r^n(x))} \d x
  \leq \left(2\sqrt{n} + n\right)
   r^{n-1} \, \overline{\F}_\Omega(S).
 \end{equation}
\end{corollary}
\begin{proof}
 For almost every~$r\in (0, \, 1)$ and~$y\in Q^n$, the projection of~$S$ onto the grid~$\GG(r, \, y)$ is well-defined.
 Let~$K$ be a cube in~$\GG(r, \, y)$ and~$K_0$ its interior.
 By Remark~\ref{rk:localflat}, we have
 \[
  \F_{K_0}\!\left(\ind(S, \, K)\lb c_K\rb\right)
  = \frac{r}{2} \abs{\ind(S, \, K)}.
 \]
 Since the interiors of the cubes of~$\GG(r, \, y)$ are pairwise disjoint,  Lemma~\ref{lemma:localisation} and Remark~\ref{rk:localflat} imply
 \begin{equation} \label{fsc0}
  \frac{r}{2} \sum_{K\in\GG(r, \, y)\colon K\subseteq\Omega} \abs{\ind(S, \, K)}
  \leq \F_\Omega(P(S, \, r, \, y)).
 \end{equation}
 Since each cube of~$\GG(r, \, y)$ has the form~$K = Q^n_r(ry + rz)$ for some~$z\in\Z^n$, integrating both sides of~\eqref{fsc0} with respect to~$y$ and applying Lemma~\ref{lemma:relaxeddeformation} yields
 \begin{equation*} %\label{fsc1}
  \frac{r}{2} \sum_{z\in\Z^n}
  \int_{Q^n} \one_{\{Q^n_r(ry+rz)\subseteq\Omega\}}
  \abs{\ind(S, \, Q^n_r(ry+rz))} \, \d y
  \leq C_n\,\overline{\F}_\Omega(S),
 \end{equation*}
 where~$C_n := \sqrt{n} + n/2$.
 By making the change of variables~$x= ry + rz$ at the left-hand side, we obtain
 \[
  \frac{r^{1 - n}}{2}
  \sum_{z\in\Z^n} \int_{Q^n_r(rz)} \one_{\{Q^n_r(x)\subseteq\Omega\}}
  \abs{\ind(S, \, Q^n_r(x))} \, \d x
  \leq C_n \,\overline{\F}_\Omega(S),
 \]
 whence the desired estimate~\eqref{fixed-scale-index} follows.
\end{proof}

The second consequence of Lemma~\ref{lemma:relaxeddeformation} is the following estimate, implying that~$\overline{\F}_\Omega(S)$ is equivalent to~$\F_\Omega(S)$ \emph{if we assume a priori that} $S$ is flat.

\begin{corollary} \label{cor:relaxeddeformation}
 For every $S\in\F_0(\Omega; \, \G)$, there holds
 \[
  \overline{\F}_\Omega(S) \leq \F_\Omega(S)
  \leq \left(\sqrt{n} + \frac{n}{2}\right) \overline{\F}_\Omega(S).
 \]
\end{corollary}
\begin{proof}%[Proof of Corollary~\ref{cor:relaxeddeformation}]
 The inequality~$\overline{\F}_\Omega\leq\F_\Omega$ follows immediately from the definition~\eqref{relaxedflat} of~$\overline{\F}_\Omega$.
 Conversely, for a flat chain~$S$, the inequality~$\F_\Omega(S)\leq (\sqrt{n} + n/2) \overline{\F}_\Omega(S)$ follows from Lemma~\ref{lemma:relaxeddeformation} and property~\eqref{Pflatconv} in Proposition~\ref{prop:deformation}.
\end{proof}

The constant~$\sqrt{n} + n/2$ at the right-hand side is not sharp (see Proposition~\ref{prop:Frelax1} below).
Note that in Corollary~\ref{cor:relaxeddeformation}, we are \emph{assuming} that~$S$ is a flat chain to begin with.
For an arbitrary~$S\in\I_0(\Omega; \, \G)$, what remains to be proved or disproved is whether~$\overline{\F}_\Omega(S) < +\infty$ implies~$\F_\Omega(S) < +\infty$.
This is what we prove in case~$n = 1$, under an additional assumption on the coefficient group~$(\G, \, \abs{\,\cdot\,})$.

\begin{prop} \label{prop:Frelax1}
 Assume that~$n = 1$ and that all bounded subsets of~$\G$ are finite.
 Then, we have~$\overline{\F}_\Omega = \F_\Omega$.
%  In particular, $\F_\Omega$ is lower semicontinuous with respect to the~$\I_\Omega$-norm.
\end{prop}
\begin{proof}
 Let~$S\in\I_0(\Omega; \, \G)$ be a chain with~$\overline{\F}_\Omega(S) < +\infty$.
 Using~\eqref{relaxedflatpoly} and~\eqref{flat}, we find
 a sequence~$(S_j)_{j\in\N}$ of flat chains with~$\sup_{j\in\N} \I_\Omega(S_j) < +\infty$ that $\I_\Omega$-converges to~$S$,
 and sequences~$(A_j)_{j\in\N}\subseteq\P_0(\Omega; \, \G)$, $(B_j)_{j\in\N}\subseteq\P_1(\overline{\Omega}; \, \G)$ such that~$S_j = A_j + \partial B_j$ in~$\Omega$ and
 \begin{equation} \label{closed11}
  \lim_{j\to+\infty} \left(\M(A_j) + \M(B_j)\right) = \overline{\F}_\Omega(S).
 \end{equation}
 Combining~\eqref{closed11} with~\eqref{alpha_G}, we see that the cardinality of~$\spt A_j$ is bounded independently of~$j$, and so are all the coefficients of the points in~$A_j$.
 Since we have assumed that all bounded subsets of~$\G$ are finite, by applying the Bolzano-Weierstrass theorem we find a (non-relabelled) subsequence of the~$A_j$'s that $\F_\Omega$-converges to some~$A\in\P_0(\Omega; \, \G)$, with~$\M(A)\leq\liminf_{j\to+\infty}\M(A_j)$.
 (We can disregard any points of~$\spt A_j$ that accumulate on the boundary of~$\Omega$, since their contribution to the~$\F_\Omega$-norm will be negligible.)
 It remains to study the limit of~$\partial B_j$.

 Since we have assumed~$n = 1$, we can identify each chain~$B_j$ with a piecewise constant function~$\Omega\to\G$, which we will denote by the same symbol.
 Then, for all subinterval~$[a, \, b]\subseteq\Omega$ and all~$j$, we can write
 \begin{equation} \label{closed12}
  \ind(\partial B_j, \, [a, \, b]) = B_j(b) - B_j(a),
 \end{equation}
 thanks to~\eqref{Islice1} (and the fact that~$\ind(\cdot, \, [a, \, b])$ is a group homomorphism).
 By Remark~\ref{rk:indconvergesae}, we can extract a non-relabelled subsequence in such a way that, for almost every~$a$, $b$ with~$[a, \, b]\subseteq\Omega$, the difference~$B_j(b) - B_j(a)$ becomes eventually constant, for large values of~$j$.
 We will write~$B_*[a, \, b]$ for the value of this constant.
 The uniform bound on the mass of~$B_j$, given by~\eqref{closed11}, together with the fact that~$\M(A)\leq\liminf_{j\to+\infty}\M(A_j)$, implies
 \begin{equation} \label{closed13}
  \limsup_{j\to+\infty} \int_\Omega \abs{B_j(t)} \d t
  = \limsup_{j\to+\infty} \M(B_j)
  \leq \overline{\F}_\Omega(S) - \M(A).
 \end{equation}
 By Fatou lemma, we deduce
 \[
  \int_\Omega \liminf_{j\to+\infty} \abs{B_j(t)} \d t
  \leq \overline{\F}_\Omega(S) - \M(A).
 \]
 Therefore, for each connected component~$\Gamma\subseteq\Omega$ there is~$t_\Gamma\in\Gamma$ such that
 \[
  \liminf_{j\to+\infty} \abs{B_j(t_\Gamma)}
  \leq \frac{\overline{\F}_\Omega(S) - \M(A)}{\H^1(\Gamma)} .
 \]
 Using again the assumption that bounded subsets of~$\G$ are finite, and a diagonal argument, we can extract a (non-relabelled) subsequence along which $B_j(t_\Gamma)$ is constant, say~$B_j(t_\Gamma) = \sigma_\Gamma$.
 We define~$B\colon\Omega\to\G$ by setting~$B(t) := \sigma_\Gamma + B_*(t_\Gamma, \, t)$ for all~$t\in\Gamma$ and all component~$\Gamma$ of~$\Omega$.
 By construction, this function satisfies
 \begin{equation} \label{closed14}
  B(b) - B(a) = B_j(b) - B_j(a)
 \end{equation}
 for almost all~$a$, $b$ such that~$[a, \, b]\subseteq\Omega$ and all~$j$ large enough; more precisely, for almost all such~$a$, $b$ there is~$j_*(a, \, b)$ such that~\eqref{closed14} holds for all~$j\geq j_*(a, \, b)$.
 In particular, $B$ is the pointwise a.e.~limit of~$B_j$.
 Fatou lemma and~\eqref{closed13} imply that~$B$ is an integrable function.
 As a consequence, $B$ is a rectifiable chain --- that is, there exists a sequence of polyhedral chains~$\Omega\to\G$ that converges to~$B$ in mass, namely in $L^1(\Omega)$.
 Then, the boundary~$\partial B$ is flat and satisfies
 \[
  \F_\Omega(\partial B) \leq \M(B)
  \leq \liminf_{j\to+\infty} \M(B_j)
  \leq \overline{\F}_\Omega(S) - \M(A),
 \]
 because of~\eqref{closed13}.
 Since~$S = A + \partial B$ in~$\Omega$, we deduce that~$S$ is flat
 and $\F_\Omega(S) \leq \overline{\F}_\Omega(S)$.
\end{proof}

% {\BBB When~$n\geq 2$, finiteness of~$\overline{\F}_\Omega$ does not imply finiteness of~$\F_\Omega$, as we will see in Proposition~\ref{prop:Frelax2}.
% \textbf{What if we impose additinal conditions, though?}
% We will say that a chain~$S\in\I_0(\Omega; \, \G)$ satisfies the condition~(AC) if and only if the following condition holds:
% \begin{enumerate}[label=(AC), ref=AC]
%  \item \label{AC} for all~$\eps > 0$ there is~$\delta > 0$ such that, for all finite family of pairwise disjoint cubes~$Q^n_{r_i}(x_i)\subseteq\Omega$, with~$i\in\{1, \, \ldots, \, N\}$, if~$\sum_{i=1}^N r_i^n \leq \delta$ then
%  \[
%   \sum_{i=1}^N \int_0^{r_i} \abs{\ind(S, \, Q_t^n(x))} \d t \leq \eps.
%  \]
% \end{enumerate}
%
% \begin{question} \label{question:AC}
%  If~$S\in\I_0(\Omega; \, \G)$ satisfies~\eqref{AC} and is such that~$\overline{\F}_\Omega(S) < +\infty$, is it true that~$\F_\Omega(S) < +\infty$?
% \end{question}
% }

\paragraph{A counterexample to the lower semicontinuity
of the flat norm when~$n\geq 2$.}

In contrast with the case~$n = 1$, when~$n\geq 2$ the flat norm is not lower semicontinuous with respect to $\I_\Omega$-convergence.
The purpose of these next results is to provide a counterexample.
We take~$\G = \Z$ as a coefficient group for our chains.
In this case, elements of~$S\in\F_0(\Omega; \, \Z)$ can be naturally identified with distributions of order one, satisfying
\begin{equation} \label{countertestflat}
  \abs{\langle S, \, \varphi\rangle}
  \leq \max\left(\norm{\varphi}_{L^\infty(\Omega)}\!, \, \norm{\nabla\varphi}_{L^\infty(\Omega)}\right) \F_\Omega(S)
\end{equation}
for all~$\varphi\in C^\infty_{\mathrm{c}}(\Omega)$.
For polyhedral chains --- which are, by definition, discrete measures on~$\Omega$ --- the estimate~\eqref{countertestflat} follows by writing~$S = A + \partial B$, then estimating separately the action of~$A$ and~$\partial B$ on a test function~$\varphi$, using the fundamental theorem of calculus for the latter.
The estimate~\eqref{countertestflat} extends by density to arbitrary~$S\in\F_0(\Omega; \, \Z)$.
However, $\I_\Omega$-convergence does not always relate to weak$^*$ convergence in the sense of distributions, as we have seen in Example~\ref{ex:noncompatible}.
The following lemma provides a sufficient condition for that to happen.
We recall that~$\QQ_\Omega$ is set of pairs~$(x, \, r)$ such that~$Q_r^n(x)\subseteq\Omega$.

\begin{lemma} \label{lemma:Iandweak}
 Let~$(S_j)\subseteq\F_0(\Omega; \, \Z)$ be a sequence such that~$\I_\Omega(S_j-S)\to 0$ and~$S_j\rightharpoonup^* T$ in the sense of distributions as~$j\to+\infty$, for some~$S\in\F_0(\Omega; \, \Z)$ and~$T\in\mathscr{D}^\prime(\Omega)$.
 Assume moreover that there exists~$h\in L^1(\QQ_\Omega)$ such that
%  $\ind(S_j, \, Q^n_r(x))$ is eventually constant for almost every~$(x, \, r)\in\Omega\times (0, \, R_*)$ (that is, for a.e.~$(x, \, r)$ there is~$j_*(x, \, r)$ such that~$\ind(S_j, \, Q^n_r(x))$ is constant for~$j\geq j_*(x, \, r)$). Assume moreover that
 \begin{equation} \label{hp:counterdistr}
  \abs{\ind(S_j, \, Q^n_r(x))} \leq h(x, \, r)
 \end{equation}
 for all~$j\in\N$ and a.e.~$(x, \, r)\in\QQ_\Omega$.
 Then, $S = T$.
\end{lemma}
\begin{proof}
 Given~$\varphi\in C^\infty_{\mathrm{c}}(\Omega)$ and~$r > 0$ small enough that~$Q^n_r(x)\subseteq\Omega$ for all~$x\in\spt\varphi$, we define
 \[
  \varphi_r(x):=\frac{1}{r^n}\int_{Q_r^n(x)}\varphi(y)\,\d y.
 \]
 The function~$\varphi_r$ belongs to~$C^\infty_{\mathrm{c}}(\Omega)$ and converges to~$\varphi$ uniformly, along with its derivatives of all order, as~$r\to 0$.
 Moreover, for all~$R\in\F_0(\Omega; \, \Z)$ and almost all~$r > 0$ small enough, we have
 \begin{equation} \label{counteraverage}
  \langle R, \, \varphi_r\rangle
  = \frac{1}{r^n}\int_\Omega
  \varphi(x) \, \ind(R, \, Q_r^n(x))\,\d x.
 \end{equation}
 When~$R$ is polyhedral, Equation~\eqref{counteraverage} follows by applying Fubini theorem to the measure~$R\otimes\d x$; for general~$R$, it follows by a density argument.

 Per Remark~\ref{rk:indconvergesae}, the assumption~$\I_\Omega(S_j-S)\to 0$ as~$j\to+\infty$ allows us to extract a (non-relabelled) subsequence in such a way that $\ind(S_j, \, Q_r^n(x))\to\ind(S, \, Q_r^n(x))$ pointwise a.e.~in~$(x, \, r)\in\QQ_{\Omega}$ as~$j\to+\infty$.
 Therefore, using~\eqref{counteraverage}, the assumption~\eqref{hp:counterdistr} and dominated convergence, for almost every sufficiently small~$r$ we obtain
 \[
  \begin{split}
   \langle S, \, \varphi_r\rangle
   &=\frac{1}{r^n}\int_\Omega
    \varphi(x) \, \ind(S, \, Q_r^n(x))\,\d x \\
   &=\lim_{j\to+\infty}\frac{1}{r^n}\int_\Omega
    \varphi(x) \, \ind(S_j, \, Q_r^n(x))\,\d x\\
   &=\lim_{j\to+\infty}\langle S_j, \, \varphi_r\rangle
   = \langle T, \, \varphi_r\rangle .
  \end{split}
 \]
 Thanks to~\eqref{countertestflat}, we know that~$S_j$ and~$S$ are distributions of order at most one. So is~$T$, because it is the weak$^*$-limit of a sequence of distributions of order one. Since~$\varphi_r\to \varphi$ in~$C^1_{\mathrm{c}}(\Omega)$ as~$r\to 0$, we can now pass to the limit as~$r\to 0$ and conclude that~$\langle S, \, \varphi\rangle = \langle T, \, \varphi \rangle$ for all~$\varphi\in C^\infty_{\mathrm{c}}(\Omega)$.
\end{proof}

The following lemma states a well-known fact, which we prove anyway for the reader's convenience.

\begin{lemma} \label{lemma:deltaprime}
 There is no chain~$R\in\F_0(\R; \, \Z)$ such that
 \begin{equation} \label{deltaprimeaction}
  \langle R, \, \psi\rangle = \psi'(0)
  \qquad \textrm{for all } \psi\in C^\infty_{\mathrm{c}}(\R).
 \end{equation}
\end{lemma}
\begin{proof}
 We first observe that every~$R\in\F_0(\R; \, \Z)$ can be written in the form
 \begin{equation} \label{flat1Drepresentation}
  \langle R, \, \psi\rangle
  = \langle A, \, \psi\rangle+\int_J B(t) \, \psi'(t)\,\d t
  \qquad \textrm{for all } \psi\in C^\infty_{\mathrm{c}}(\R),
 \end{equation}
 for some~$A\in\P_0(\R; \, \Z)$ and~$B\in L^1(\R; \, \Z)$.
 For polyhedral chains, this is immediate, by writing~$R = A + \partial B$, where~$A$ is a polyhedral $0$-chain --- hence a finite sum of Dirac's deltas --- and~$B$ is a polyhedral $1$-chain, which can be identified with a piecewise constant function in~$L^1(\R; \, \Z)$.
 For general chains~$R\in\F_0(\R; \, \Z)$, \eqref{flat1Drepresentation} follows by a density argument.
%  Indeed, choose~$(R_m)_{m\in\N}\subseteq\P_0(J; \, \Z)$ converging to~$R$ such that
%  \[
%   \sum_{m=1}^{+\infty}\F_J(R_{m+1}-R_m)<+\infty.
%  \]
%  For every~$m$, write
%  \[
%   R_{m+1}-R_m=A_m+\partial B_m
%   \qquad\textrm{in }J,
%  \]
%  in such a way that
%  \[
%   \sum_{m=1}^{+\infty}\big(\M(A_m)+\M(B_m)\big)<+\infty.
%  \]
%  Since a nonzero~$0$-chain with integer coefficients has mass at least one, only finitely many of the~$A_m$'s are nonzero.
%  Thus,~$R_1+\sum_m A_m$ is a polyhedral~$0$-chain, which we denote by~$A$.
%  Identifying~$1$-chains in~$J$ with integer-valued piecewise constant functions, the partial sums of the~$B_m$'s are Cauchy in~$L^1(J)$ and converge to some~$b\in L^1(J)$.
%  Up to extracting a subsequence, the convergence is pointwise almost everywhere, and hence~$b$ is integer-valued almost everywhere.
%  Passing to the limit gives~\eqref{flat1Drepresentation}.
 Suppose now that there is~$R\in\F_0(\R; \, \Z)$ satisfying~\eqref{deltaprimeaction}.
 Then, Equation~\eqref{flat1Drepresentation} implies that we can write
 $A-B^\prime=-\delta^\prime_0$ for some~$A\in\P_0(\R; \, \Z)$ and~$B\in L^1(\R; \, \Z)$, where~$\delta_0$ is Dirac's delta at zero and the derivatives are taken in the sense of distributions.
 Since~$A$ is a finite atomic measure, there exists a piecewise constant function~$a\in L^1(\R)$ such that~$a^\prime=A$.
 Such function~$a$ satisfies~$(a - B)^\prime = -\delta_0^\prime$ and, hence, $a - B  = -\delta_0 + C$ for some constant~$C$.
 This is impossible, because it would imply that~$\delta_0$ is represented by an~$L^1$-function.
\end{proof}

We are now ready to present a counterexample to the equality~$\overline{\F}_\Omega = \F_\Omega$.
For simplicity, we assume that~$n = 2$ and~$\Omega\supseteq [0, \, 1]^2$.
However, the same construction can be generalised to any~$\Omega\subseteq\R^n$, by scaling and introducing dummy variables.
For every~$j\geq1$ and~$k\in\{0, \, \ldots, \, j-1\}$, let~$x_{j,k} := (k/j, \, k/j)$, $y_{j,k} := (k/j + 1/j, \, k/j)$, and
\begin{equation} \label{counterchains}
 S_j :=\sum_{k=0}^{j-1} \left(\lb y_{j,k}\rb - \lb x_{j,k}\rb\right).
\end{equation}
The chain~$S_j$ is illustrated in Figure~\ref{fig:counterexample}.

\begin{figure}%[h]
  \centering
  \includegraphics[height=0.3\textheight]{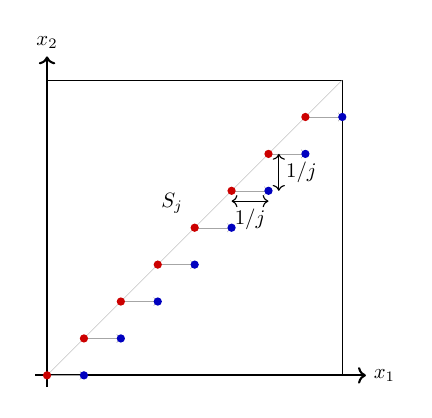}
  \caption{The chain~$S_j$ used in the proof of Proposition~\ref{prop:Frelax2}.
  This chain is the sum of~$j$ horizontal dipoles of length~$1/j$, arranged along the diagonal of the square~$[0, \, 1]^2$.
  Red points have multiplicity~$-1$, blue ones have multiplicity~$1$.
  As~$j\to\infty$, the dipoles become dense along the diagonal and their limit, with respect to the~$\I_\Omega$-norm, is not a flat chain.}
  \label{fig:counterexample}
 \end{figure}

\begin{prop} \label{prop:Frelax2}
 The sequence~$(S_j)_{j\in\N}\subseteq\P_0(\Omega; \, \Z)$
 is Cauchy with respect to the~$\I_\Omega$-norm.
 Its limit~$S\in\I_0(\Omega; \, \Z)$ satisfies
 \[
  \overline{\F}_\Omega(S) \leq 1, \qquad \F_\Omega(S) = +\infty.
 \]
\end{prop}
\begin{proof}
 We split the proof into several steps.

 \setcounter{step}{0}
 \begin{step}[$(S_j)_{j\in\N}$ is a Cauchy sequence in~$\I_0(\Omega; \, \Z)$]
  Let
  \[
   \Lambda_j:=\left\{0, \, \frac1j, \, \ldots, \, \frac{j-1}{j}\right\}.
  \]
  For each point~$x = (a, \, b)\in\Omega$ and~$r > 0$, a direct computation gives
  \begin{equation} \label{counterindexformula}
   \begin{split}
    \ind(S_j, \, Q_r^2(a, \, b))
    ={}&\#\left(\Lambda_j\cap
    \left[a-\frac{r}{2}-\frac1j, \, a+\frac{r}{2}-\frac1j\right]
    \cap\left[b-\frac{r}{2}, \, b+\frac{r}{2}\right]\right)\\
    &-\#\left(\Lambda_j\cap\left[a-\frac{r}{2}, \, a+\frac{r}{2}\right]\cap\left[b-\frac{r}{2}, \, b+\frac{r}{2}\right]\right).
   \end{split}
  \end{equation}
  We claim that the right-hand side of~\eqref{counterindexformula} is eventually constant as~$j\to+\infty$, for almost every~$(a, \, b, \, r)$.
  To see this, let
  \begin{equation} \label{counterLU}
   L:=\max\!\left\{0, \, a-\frac{r}{2}, \, b-\frac{r}{2}\right\},
   \qquad
   U:=\min\!\left\{1, \, a+\frac{r}{2}, \, b+\frac{r}{2}\right\}.
  \end{equation}
  We may disregard the set of triples for which one of the maximum or minimum above is attained by more than one term, or~$L=U$.
  This is a negligible set, contained in a finite union of affine hyperplanes.
  If~$L>U$, the two sets in~\eqref{counterindexformula} are empty for all sufficiently large~$j$.
  Suppose then that~$L<U$.
  For~$j$ sufficiently large, both intervals appearing in~\eqref{counterindexformula} are nonempty. Let
  \begin{align*}
   L_j &:= \max\left\{0,\,a-\frac{r}{2},\,b-\frac{r}{2}\right\} ,
   &\
   U_j &:= \min\left\{1-\frac1j,\,a+\frac{r}{2},\,b+\frac{r}{2}\right\}, \\
   \widetilde L_j &:= \max\left\{0,\,a-\frac{r}{2}-\frac1j,\,b-\frac{r}{2}\right\} \!,
   &\
   \widetilde U_j &:= \min\left\{1-\frac1j,\,a+\frac{r}{2}-\frac1j,\,b+\frac{r}{2}\right\} \!,
  \end{align*}
  so that~\eqref{counterindexformula} can be rewritten as
  \begin{equation} \label{counterindexformulabis}
   \ind(S_j, \, Q_r^2(a, \, b))
   = \#\big(\Lambda_j\cap[\widetilde L_j,\widetilde U_j]\big) -  \#\big(\Lambda_j\cap[L_j,U_j]\big)
  \end{equation}
  Recall that we have excluded the triples~$(a, \, b, \, r)$ for which the maximum or the minimum in~\eqref{counterLU} are attained by more than one term.
  Then, for all sufficiently large~$j$ we have
  \[
   \widetilde L_j=
   \begin{cases}
    L_j-\dfrac1j, &\text{if }L=a-\dfrac{r}{2},\\[4pt]
    L_j, &\text{if }L=0\text{ or }L=b-\dfrac{r}{2},
   \end{cases}
  \]
  and
  \[
   \widetilde U_j=
   \begin{cases}
    U_j-\dfrac1j, &\text{if }U=a+\dfrac{r}{2},\\[4pt]
    U_j, &\text{if }U=1\text{ or }U=b+\dfrac{r}{2}.
   \end{cases}
  \]
  For example, if~$L=a-r/2$, then~$a-r/2$ is strictly larger than both~$0$ and~$b-r/2$; hence, the same inequalities remain true after replacing~$a-r/2$ with~$a-r/2-1/j$, provided~$j$ is sufficiently large. The other cases are analogous.
  We may therefore find~$\eps_-$, $\eps_+\in\{0,1\}$, depending on~$(a, \, b, \, r)$ but not on~$j$, such that
  \begin{equation} \label{counterendpoints}
   \widetilde L_j=L_j-\frac{\eps_-}{j},
   \qquad
   \widetilde U_j=U_j-\frac{\eps_+}{j}.
  \end{equation}
%   for all sufficiently large~$j$.
%   Since both intervals are contained in~$[0,1-1/j]$, and
%   \[
%    \Lambda_j=\frac1j\Z\cap\left[0,1-\frac1j\right],
%   \]
%   we can simply count points of~$j^{-1}\Z$.
  We also recall that, for any nonempty interval~$[c, \, d]$,
  \begin{equation} \label{countercount}
   \#\left(\frac1j\Z\cap[c,d]\right)
   = \lfloor jd\rfloor-\lceil jc\rceil+1,
  \end{equation}
  where~$\lfloor x\rfloor$~denotes the largest integer~$k\leq x$ and~$\lceil x\rceil$ is the smallest integer~$h\geq x$.
  Now we can estimate the index~$\ind(S_j, \, Q^2_r(a, \, b))$
  using~\eqref{counterendpoints} and~\eqref{countercount}:
  \[
   \begin{split}
    \#\big(\Lambda_j\cap[\widetilde L_j,\widetilde U_j]\big)
    &=
    \left\lfloor jU_j-\eps_+\right\rfloor
    -
    \left\lceil jL_j-\eps_-\right\rceil+1\\
    &=
    \lfloor jU_j\rfloor-\lceil jL_j\rceil+1
    +\eps_- -\eps_+\\
    &=
    \#\big(\Lambda_j\cap[L_j,U_j]\big)
    +\eps_- -\eps_+.
   \end{split}
  \]
  We conclude that $\ind(S_j, \, Q_r^2(a, \, b))=\eps_- -\eps_+$ for all sufficiently large~$j$ ---
  in particular, $\ind(S_j, \, Q_r^2(a, \, b))$ is eventually constant for almost every~$(a, \, b, \, r)$.
  Then, Lemma~\ref{lemma:Iae} implies that~$(S_j)_{j\in\N}$
  is a Cauchy sequence in~$\I_0(\Omega; \, \Z)$.
 \end{step}

 \begin{step}[The indices of~$S_j$ are uniformly bounded]
  We claim that
  \begin{equation} \label{counterindexbound}
   \abs{\ind(S_j, \, Q_r^2(x))} \leq 4
  \end{equation}
  for all~$j$, $x\in\R^2$ and~$r>0$.
  Indeed, the symmetric difference of the intervals
  \[
   \left[a-\frac{r}{2}-\frac{1}{j}, \, a+\frac{r}{2}-\frac{1}{j}\right]
   \triangle \left[a-\frac{r}{2}, \, a+\frac{r}{2}\right]
  \]
  is contained in the union of two intervals of length~$1/j$.
  Each of them contains at most two points of~$j^{-1}\Z$, and hence~\eqref{counterindexbound} follows from~\eqref{counterindexformula}.
 \end{step}

 \begin{step}[The limit~$S$ is not flat]
  Each~$S_j$ can be written in the form~$S_j = \sum_{k=0}^{j-1} \partial\lb L_{j,k}\rb$, where~$L_{j,k}$ is the oriented line segment from~$(k/j, \, k/j)$ to~$(k/j + 1/j, \, k/j)$.
  Therefore, we have
  \[
   \F_\Omega(S_j) \leq \sum_{k=0}^{j-1}\M(\lb L_{j,k}\rb) = 1
  \]
  for all~$j\geq 1$ and, hence, $\overline{\F}_\Omega(S) \leq 1$.
  We claim that~$S$ is not flat.
  Indeed, suppose that~$S$ is flat.
  Then, the assumptions of Lemma~\ref{lemma:Iandweak} are satisfied, thanks to~\eqref{counterindexbound}.
  As a consequence, $S$ coincides with the weak$^*$ limit of~$(S_j)$ in the sense of distributions and can be computed.
  For all~$\varphi\in C^\infty_{\mathrm{c}}(\Omega)$ we have
  \begin{equation*} %\label{counterdistributionSj}
   \begin{split}
    \langle S_j, \, \varphi\rangle
    &=\sum_{k=0}^{j-1}
     \left(
      \varphi\left(\frac{k+1}{j}, \, \frac{k}{j}\right)
      -\varphi\left(\frac{k}{j}, \, \frac{k}{j}\right)
     \right)\\
    &=\sum_{k=0}^{j-1}\int_0^{1/j}
    \partial_1\varphi\left(\frac{k}{j}+t, \, \frac{k}{j}\right)\,\d t
     \longrightarrow
     \int_0^1\partial_1\varphi(t, \, t)\,\d t
   \end{split}
  \end{equation*}
  as~$j\to+\infty$, hence
  \begin{equation} \label{counterlimitdistribution}
   \langle S, \, \varphi\rangle
   =\int_0^1\partial_1\varphi(t, \, t)\,\d t
  \end{equation}
  for all~$\varphi\in C^\infty_{\mathrm{c}}(\Omega)$.
  Now, let $\Gamma:=\{(t, \, t)\colon 0\leq t \leq 1\}$.
  Choose~$\chi\in C^\infty_{\mathrm{c}}(\Omega)$ such that~$\chi=1$ in a neighbourhood of~$\Gamma$, and define
  \[
   f(x):=\chi(x)(x_1-x_2)
  \]
  for $x = (x_1, \, x_2)\in\R^2$.
  Then~$f(x) = x_1-x_2$ in a neighbourhood of~$\Gamma$ and~$f=0$ on~$\R^2\setminus\Omega$. %Let~$U:=(-1/2, \, 1/2)$.
  Since~$S$ belongs to~$\F_0(\Omega; \, \Z)$ and is compactly supported in~$\Gamma\subseteq\Omega$, its push-forward via~$f$ --- defined as~$\langle f_{\#}S, \, \psi\rangle = \langle S, \, \psi\circ f\rangle$ for all~$\psi\in C^\infty_{\mathrm{c}}(\R)$ --- belongs to~$\F_0(\R; \, \Z)$; see, e.g., \cite[4.1.14]{Federer} or~\cite[Sections 5--6]{Fleming}.
%  For a polyhedral chain~$R=\sum_i m_i\lb x_i\rb\in\P_0(\Omega; \, \Z)$, set
%  \[
%   f_\#R:=\sum_{i:\, f(x_i)\in J}m_i\lb f(x_i)\rb.
%  \]
%  We claim that
%  \begin{equation} \label{counterpushforward}
%   \F_J(f_\#R)
%   \leq\max\{1, \, \operatorname{Lip}(f)\}\F_\Omega(R).
%  \end{equation}
%  Indeed, write~$R=A+\partial B$ in~$\Omega$, with
%  \[
%   B=\sum_{\nu=1}^N m_\nu\lb[u_\nu, \, v_\nu]\rb,
%  \]
%  and set
%  \[
%   C:=\sum_{\nu=1}^N m_\nu\lb[f(u_\nu), \, f(v_\nu)]\rb,
%  \]
%  where degenerate segments are understood as zero.
%  Since~$f=1$ on~$\R^2\setminus\Omega$ and~$1\notin J$, we have
%  \[
%   f_\#R=f_\#A+\partial C
%   \qquad\textrm{in }J.
%  \]
%  Moreover,
%  \[
%   \M(f_\#A)\leq\M(A),
%   \qquad
%   \M(C)\leq\operatorname{Lip}(f)\M(B),
%  \]
%  and~\eqref{counterpushforward} follows by taking the infimum.
%  Thus,~$f_\#$ extends continuously to a homomorphism
%  \[
%   f_\#\colon\F_0(\Omega; \, \Z)\to\F_0(J; \, \Z).
%  \]
  However, for all~$\psi\in C^\infty_{\mathrm{c}}(\R)$,  Equation~\eqref{counterlimitdistribution} gives
  \[
   \begin{split}
    \langle f_\#S, \, \psi\rangle
    &=\langle S, \, \psi\circ f\rangle
    =\int_0^1\partial_1(\psi\circ f)(t, \, t)\,\d t
    =\psi'(0),
   \end{split}
  \]
  because~$f(t, \, t)=0$ and~$\partial_1f(t, \, t)=1$ for all~$t\in(0, \, 1)$.
  This contradicts Lemma~\ref{lemma:deltaprime}.
  Therefore,~$S\notin\F_0(\Omega; \, \Z)$.
  \qedhere
 \end{step}
\end{proof}

\paragraph{Strict inclusion of flat chains in~$\I_0(\Omega; \, \G)$.}

In this paragraph, we prove the following result.

\begin{prop} \label{prop:flat-vs-I}
 The inclusion~$\mathbb F_0(\Omega;\G)\subseteq\I_0(\Omega; \, \G)$
 is an equality if and only if~$n=1$ \emph{and} $\G$ is bounded.
 If either $n\geq 2$ or $\G$ is unbounded, then there
 exists $S\in\I_0(\Omega;\G)$ such that~$\overline{\F}_\Omega(S)=+\infty$.
\end{prop}

The proof of Proposition~\ref{prop:flat-vs-I} relies on the following observation:
for any open set~$U\subseteq\Omega$ and any polyhedral chain~$S\in\P_0(\Omega; \, \G)$, we have
\begin{equation} \label{restr-I}
 \I_U(S\mres U) \leq \I_\Omega(S),
\end{equation}
because every cube contained in~$U$ is also contained in~$\Omega$, and its intersection index with~$S\mres U$ coincides with its intersection index with~$S$.
Therefore, the restriction induces a group homomorphism~$\I_0(\Omega; \, \G)\to\I_0(U; \, \G)$, which we still denote as~$\cdot\,\mres U$.

\begin{lemma} \label{lemma:localisationrel}
 Given pairwise disjoint, open polyhedral sets~$U_1$,\ldots, $U_N$ contained in~$\Omega$ and~$S\in\I_0(\Omega; \, \G)$, we have
 \begin{equation*} %\label{relaxed-flat-localisation}
  \sum_{i=1}^N \overline{\F}_{U_i}(S\mres U_i)
  \leq \overline{\F}_\Omega(S) .
 \end{equation*}
\end{lemma}
\begin{proof}
 Let~$(S_j)_{j\in\N}\subseteq\P_0(\Omega; \, \G)$ satisfy~$\I_\Omega(S_j-S)\to 0$ as~$j\to+\infty$.
 The inequality~\eqref{restr-I} implies $\I_{U_i}(S_j\mres U_i-S\mres U_i)\to 0$
 as~$j\to+\infty$ for every $i$.
 Then, Lemma~\ref{lemma:localisation} gives
 \[
 \begin{split}
  \liminf_{j\to+\infty}\F_\Omega(S_j)
  \geq \liminf_{j\to+\infty} \sum_{i=1}^N\F_{U_i}(S_j\mres U_i)
  \geq \sum_{i=1}^N \overline{\F}_{U_i}(S\mres U_i).
 \end{split}
 \]
 Taking the infimum over all such sequences proves the lemma.
\end{proof}

\begin{proof}[Proof of Proposition~\ref{prop:flat-vs-I}]
 We distinguish three cases.

 \medskip
 \noindent
 \emph{Case 1: $n=1$ and $\G$ is bounded.}
%  Set
%  \[
%   M_\G:=\sup_{\sigma\in\G}|\sigma|<+\infty
%  \]
%  and
 Let~$S\in\I_0(\Omega; \, \G)$, and let~$(S_j)_{j\in\N}\subseteq\P_0(\Omega;\G)$ be a sequence such that $\I_\Omega(S_j-S)\to 0$ as~$j\to+\infty$.
 By Remark~\ref{rk:indconvergesae}, upon extracting a subsequence, we may assume that~$\ind(S_j, \, [a, \, b])$ is eventually constant for almost every pair~$a<b$ such that $[a, \, b]\subseteq\Omega$.
 Let~$\Gamma$ be a connected component of~$\Omega$.
 By Fubini's theorem, we can choose~$t_\Gamma\in\Gamma$ such that, for almost every~$t\in\Gamma$, the intersection index of~$S_j$ on the interval with endpoints~$t_\Gamma$ and~$t$ is eventually constant.
 Since~$n=1$, we may identify polyhedral $1$-chains with
 piecewise constant $\G$-valued functions.
 For every~$j$, choose such a function~$B_j\colon \Omega\to\G$ so that
 $\partial B_j=S_j$ in~$\Omega$ and normalise it by requiring
 $B_j(t_\Gamma)=0$ on every connected component~$\Gamma$ of~$\Omega$.
 By~\eqref{Islice1}, for almost every~$t\in\Gamma$ we have
 \[
  B_j(t)=
  \begin{cases}
   \ind(S_j, \, [t_\Gamma, \, t])
      &\textrm{if } t>t_\Gamma,\\[1mm]
   -\ind(S_j,\, [t, \, t_\Gamma])
      &\textrm{if } t<t_\Gamma.
  \end{cases}
 \]
 Therefore, $B_j(t)$ is eventually constant for almost every~$t\in\Omega$.
 Moreover, $B_j$ is uniformly bounded, because~$\G$ is bounded.
 Therefore, by definition of the flat norm and dominated convergence, we deduce
 \[
  \F_\Omega(S_j - S_k)
  \leq \M(B_j - B_k)
  = \int_\Omega \abs{B_j(t) - B_k(t)} \d t
  \to 0 \qquad\textrm{as } j, \, k\to+\infty.
 \]
 Thus $(S_j)_j$ is a Cauchy sequence with respect to the flat norm, and it $\F_\Omega$-limit must coincide with its~$\I_\Omega$-limit because of~\eqref{IF}.
 As a consequence, $S\in\F_0(\Omega; \, \G)$.

 \medskip
 \noindent
 \emph{Case 2: $n=1$ and $\G$ is unbounded.}
 Choose an open interval~$\Gamma$ with~$\overline{\Gamma}\subseteq\Omega$.
 Since~$\G$ is unbounded, we can find a sequence
 $(\sigma_j)_{j\in\N}\subseteq\G$ such that
 \[
  \abs{\sigma_j}>1 \quad \textrm{for all } j\in\N,
  \qquad
  4\sum_{j=1}^{+\infty}\frac{1}{\abs{\sigma_j}}
  <\H^1(\Gamma).
 \]
 We can also choose points~$a_j\in\Gamma$ so that the intervals
 $U_j := (a_j - 2\abs{\sigma_j}^{-1}, \, a_j + 2\abs{\sigma_j}^{-1})$
 have pairwise disjoint closures and are contained in~$\Gamma$.
 Now, we consider the dipoles
 \[
  D_j:= \sigma_j \lb a_j+\abs{\sigma_j}^{-1}\rb -
  \sigma_j \lb a_j\rb.
 \]
%  Since the two points defining~$D_j$ are at distance
%  $\abs{\sigma_j}^{-1}$ and stay uniformly away from~$\partial\Omega$
%  relative to this distance,
 We estimated the~$\I_\Omega$-norm of each~$D_j$ from above and below in  Example~\ref{ex:Idipole}; we have
 \[
  \sum_{j=1}^{+\infty}\I_\Omega(D_j)
  \leq \sum_{j=1}^{+\infty}\frac{C_\Omega}{\abs{\sigma_j}}
  < +\infty,
 \]
 where~$C_\Omega$ is a constant depending on~$\Omega$ only. As a consequence, the series~$S := \sum_{j=1}^{+\infty} D_j$ converges with respect to the~$\I_\Omega$-norm.
 For every~$j$, we have
 \[
  \F_{U_j}(D_j) = 1, \qquad
  \overline{\F}_{U_j}(D_j)
  \geq \frac{2}{3}\F_{U_j}(D_j) = \frac{2}{3}
 \]
 by Remark~\ref{rk:localflat} and Corollary~\ref{cor:relaxeddeformation}.
 Thus, Lemma~\ref{lemma:localisationrel} gives
 \[
  \overline{\F}_\Omega(S)
  \geq \sum_{j=1}^N \overline{\F}_{U_j}(S\mres U_j)
  = \sum_{j=1}^N \overline{\F}_{U_j}(D_j)
  \geq \frac{2N}{3}
 \]
 for any integer~$N\geq 1$.
 Since~$N$ is arbitrary, we must have~$\overline{\F}_\Omega(S) = +\infty$.

 \begin{figure}%[h]
  \centering
  \includegraphics[height=0.25\textheight]{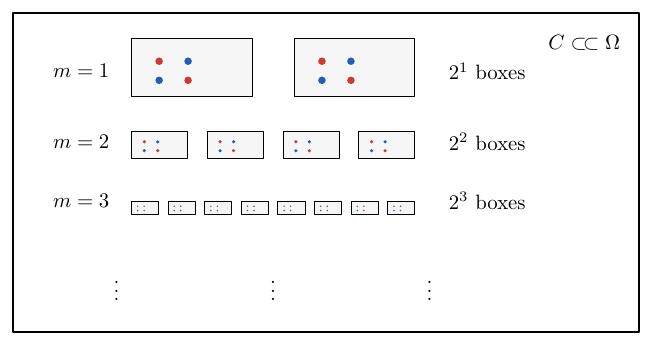}
  \caption{The chain~$S$ constructed in Case~3 of Proposition~\ref{prop:flat-vs-I}.
  The rectangles~$U_{m,k}$ are shown in gray; red points have multiplicity~$-1$, blue ones have multiplicity~$1$.}
  \label{fig:manyquadrupoles}
 \end{figure}

 \medskip
 \noindent
 \emph{Case 3: $n\geq2$.}
 Fix~$\sigma\in\G\setminus\{0\}$ and choose an open cube
 $C$ with~$\overline C\subseteq\Omega$.
 For~$\eps\in (0, \, 1)$ sufficiently small, we can find pairwise disjoint open rectangles
 \[
  U_{m,k}\csubset C, \qquad
  m \geq 1, \quad k\in\{1, \, \ldots, \, 2^m\},
 \]
 of the form
 \[
  U_{m,k} = a_{m,k} + 2^{-m}\eps \left((-3,4)^2\times(-3,3)^{n-2}\right).
 \]
 Indeed, for every~$m$ the rectangles~$U_{m,1}$, \ldots, $U_{m,2^m}$ may be arranged in one row; the length of each row is bounded from above by a constant multiple of~$\eps$,
 while the sum of the heights of all rows is bounded by a constant
 multiple of~$\eps\sum_{m=1}^{+\infty}2^{-m} = \eps$.
 Thus, the whole family~$U_{m,k}$ fits inside~$C$ if~$\eps$ is small enough.
%  We also assume~$\eps<1$.
 Inside each~$U_{m,k}$, we consider the quadrupole
 \[
 \begin{split}
  Q_{m,k}:={}
  &\sigma \lb a_{m,k} + 2^{-m}\eps(e_1+e_2)\rb
   - \sigma\lb a_{m,k} + 2^{-m}\eps \, e_1\rb\\
  &- \sigma \lb a_{m,k} + 2^{-m}\eps \, e_2\rb
   + \sigma \lb a_{m,k}\rb,
 \end{split}
 \]
 where~$\sigma\in\G\setminus\{0\}$ is given.
 The construction is illustrated in Figure~\ref{fig:manyquadrupoles}.
 The upper-bound computation in Example~\ref{ex:Iquadrupole} gives
 \[
 \begin{split}
  \sum_{m=1}^{+\infty}\sum_{k=1}^{2^m}
  \I_\Omega(Q_{m,k})
  &\leq C_{\Omega, n} \, \eps^2
  \sum_{m=1}^{+\infty}2^{-m}
  < +\infty,
 \end{split}
 \]
 where~$C_{\Omega, n}$ depends only on~$\Omega$, $n$.
 Therefore the series
 \[
  S := \sum_{m=1}^{+\infty} \sum_{k=1}^{2^m}Q_{m,k}
 \]
 converges in~$\I_0(\Omega; \, \G)$.
 As before,
%  For every~$m,k$, continuity of restriction gives
%  \[
%   S\mres U_{m,k}=Q_{m,k}.
%  \]
%  Moreover,
%  \[
%   \dist_\infty(a_{m,k},\partial U_{m,k})
%   =3\eps2^{-m}.
%  \]
%  Hence
 Remark~\ref{rk:localflat} and Corollary~\ref{cor:relaxeddeformation} give
 \[
  \F_{U_{m,k}}(Q_{m,k})
  = 2^{1 - m}\abs{\sigma} \eps ,
  \qquad
  \overline{\F}_{U_{m,k}}(Q_{m,k})
%   \geq \frac{1}{\sqrt n+n/2} \F_{U_{m,k}}(Q_{m,k})
  \geq \frac{2^{1 - m}\abs{\sigma} \eps}{\sqrt n+n/2},
 \]
 so Lemma~\ref{lemma:localisationrel} implies
 \[
 \begin{split}
  \overline{\F}_\Omega(S)
  &\geq \sum_{m=1}^M\sum_{k=1}^{2^m}
   \overline{\F}_{U_{m,k}}(S\mres U_{m,k})
  \geq \frac{2M\abs{\sigma} \eps}{\sqrt n+n/2}
 \end{split}
 \]
 and, since~$M\geq 1$ is arbitrary, we conclude that~$\overline{\F}_\Omega(S)=+\infty$.
\end{proof}

\paragraph{A flatness criterion.}

The following result provides a sufficient condition for a chain~$S\in\I_0(\Omega; \, \G)$ to be flat, which will be helpful in proving Proposition~\ref{prop:Orliczflat}.
We recall that~$\Omega_r$ is defined by~\eqref{Omega_r}.

\begin{prop} \label{prop:Dini-flatness}
 Let~$S\in\I_0(\Omega; \, \G)$.
 Suppose that the coefficient group~$\G$ is countable and that there exists~$r_0>0$ such that
 \begin{equation} \label{hp:Dini-index}
  \int_0^{r_0}\frac{1}{r^n}
   \int_{\Omega_r}\abs{\ind(S, \, Q_r^n(x))} \d x\,\d r
  <+\infty.
 \end{equation}
 Then~$S\in\F_0(\Omega; \, \G)$.
\end{prop}

According to Corollary~\ref{cor:fixed-scale-index}, any flat chain (or more genrally, any chain~$S$ with~$\overline{\F}_\Omega(S) < +\infty$) satisfies
\begin{equation*}
 \frac{1}{r^n} \int_{\Omega_r}\abs{\ind(S, \, Q_r^n(x))} \d x
  \leq \frac{C_n}{r} \, \overline{\F}_\Omega(S)
  \leq \frac{C_n}{r} \, \F_\Omega(S)
\end{equation*}
for almost every~$r\in (0, \, 1)$.
Thus,~\eqref{hp:Dini-index} requires a stronger decay of the indices of~$S$ as~$r\downarrow 0$.
The proof is based on polyhedral approximations of~$S$ associated with nested dyadic grids.

Given~$h>0$, we set~$h_k := 2^{-k}h$ for all integers~$k\geq 0$.
Given~$\xi\in Q^n$, let~$y_k(\xi)\in Q^n$ be the representative modulo~$\Z^n$ of
$2^k\xi + (1/2, \, \ldots, \, 1/2)$.
Then, the cubes of~$\GG(h_k, \, y_k(\xi))$ can equivalently be written as
\[
 h\xi + h_k z + [0, \, h_k]^n, \qquad z\in\Z^n.
\]
In particular, the grids~$\GG(h_k, \, y_k(\xi))$ are nested --- that is, each cube of~$\GG(h_k, \, y_k(\xi))$ is a union of cubes of~$\GG(h_{k+1}, \, y_{k+1}(\xi))$.
An arbitrary cube~$Q_r^n(x)\subseteq\Omega$ with~$r > 2h_k$ can be approximated by a union of cubes of the grid~$\GG(h_k, \, y_k(\xi))$.
More precisely, suppose that~$r > 2h_k$ and set
\begin{equation} \label{ell_kihxi}
 \ell_{k,i}^{h,\xi}(x, \, r)
 :=h\xi_i+h_k\left\lceil\frac{x_i-r/2-h\xi_i}{h_k}\right\rceil,
 \qquad
 N_k^h(r):=\left\lfloor\frac{r}{h_k}\right\rfloor-1
\end{equation}
for~$i\in\{1, \, \ldots, \, n\}$.
Here, $\lceil t\rceil$ denotes the smallest integer greater than or equal to a given~$t$, and~$\lfloor t\rfloor$ denotes the largest integer less than or equal to~$t$.
We define
\begin{equation} \label{inner-grid-cube}
 C_k^{h,\xi}(x, \, r)
 := \prod_{i=1}^n \left[\ell_{k,i}^{h,\xi}(x, \, r), \,
 \ell_{k,i}^{h,\xi}(x, \, r) + N_k^h(r)h_k\right].
\end{equation}
Thus, $C_k^{h,\xi}(x, \, r)$ is a union of cells of~$\GG(h_k, \, y_k(\xi))$, is itself a cube, and is contained in~$Q_r^n(x)$.
(Note that~$C^{h,\xi}_k(x, \, r)$ is not the union of all the cells of~$\GG(h_k, \, y_k(\xi))$ contained in~$Q_r^n(x)$, which is not a cube in general).
Writing
\begin{equation} \label{Phi-grid-cube}
 C_k^{h,\xi}(x, \, r) = Q^n_{r_k}(x_k),
 \qquad
 \Phi_k(h, \, \xi; \, x, \, r) := (x_k, \, r_k),
\end{equation}
a direct computation gives
\begin{equation} \label{inner-grid-cube-close}
 \abs{x_k-x}\leq \sqrt n\,h_k,
 \qquad
 0<r-r_k<2h_k.
\end{equation}
Moreover,~$Q_r^n(x)\setminus C_k^{h,\xi}(x, \, r)$ is contained in the~$2h_k$-neighbourhood of~$\partial Q_r^n(x)$.

\begin{lemma} \label{lemma:Phi-compression}
 There exists a universal constant~$C>0$ such that, for any~$r_0 > 0$, any measurable function $g\colon\QQ_\Omega\to [0, \, +\infty)$ and any~$k\geq 0$, we have
 \begin{equation} \label{Phi-compression}
  \int_{r_0/2}^{r_0}\!\int_{Q^n}\!\int_{\QQ_\Omega\cap\{r>2h_k\}}
  g(\Phi_k(h, \, \xi; \, x, \, r))\,\d x\,\d r\,\d\xi\,\frac{\d h}{h}
  \leq C\int_{\QQ_\Omega}g(x, \, r)\,\d x\,\d r.
 \end{equation}
\end{lemma}
\begin{proof}
 We identify~$g$ with a function~$\R^n\times (0, \, +\infty)\to[0, \, +\infty)$,
 by setting~$g := 0$ out of~$\QQ_\Omega$.
 Set~$s := h_k = 2^{-k}h$ and~$a_k:=2^{-k}r_0$,
 so that~$s\in(a_k,2a_k)$ and~$\d h/h=\d s/s$.
 Let~$\mathbf 1 := (1, \, 1, \, \ldots, \, 1)$.
 We recall that
 \[
  \Phi_k(h, \, \xi; \, x, \, r)
  = \left(\ell_k^{h,\xi}(x, \, r) + \frac{Ns}{2}\mathbf 1, \,
   Ns\right),
 \]
 where~$\ell_{k}^{h,\xi} = (\ell_{k,1}^{h,\xi}, \, \ldots, \, \ell_{k,n}^{h,\xi})$
 and each~$\ell_{k,i}^{h,\xi}$ is defined in~\eqref{ell_kihxi},
 and~$N := N_k^h(r) = \lfloor r/s\rfloor - 1$.
 Moreover, for each $(s, \, x, \, r)$, the function~$\xi\mapsto\ell_k^{h,\xi}(x, \, r)$ is $2^{-k}\Z^n$-periodic.
 By subdividing~$Q^n$ into $2^{kn}$ subcubes of size~$2^{-k}$, and making a linear change of variable in each of them,
%  (exactly as in the proof of of~\eqref{preservesLebesgue}),
 we obtain
 \[
  \int_{Q^n} g(\Phi_k(h, \, \xi; \, x, \, r)) \, \d \xi
  = \int_{Q^n} g\!\left(\ell(\eta) + \frac{Ns}{2}\mathbf 1, \,
   Ns\right) \d\eta,
 \]
 where
 \[
  \ell_i(\eta) = s\eta_i+s\left\lceil\frac{x_i - r/2}{s} - \eta_i\right\rceil,
  \qquad i\in\{1, \, \ldots, \, n\} .
 \]
 For each given~$(s, \, x, \, r)$, the function~$\eta_i\mapsto \ell_i(\eta)$ is $1$-periodic and, on each interval of the form~$[(x_i - r/2)/s + m, \, (x_i - r/2)/s + m + 1)$ for~$m\in\Z$, coincides with a dilation of a factor~$s$ followed by a translation.
 Therefore, a change of variable~$\ell = \ell(\eta)$ with~$\eta\in \prod_{i=1}^n[(x_i - r/2)/s, \, (x_i - r/2)/s + 1)$ gives
 \[
  \int_{Q^n} g(\Phi_k(h, \, \xi; \, x, \, r)) \, \d \xi
  = s^{-n}
   \int_{x-\frac{r}{2}\mathbf{1}+[0,s)^n}
   g\!\left(\ell + \frac{Ns}{2}\mathbf 1, \, Ns\right) \d\ell,
 \]
 Next, we integrate over~$x\in\Omega_r$ and apply Fubini's theorem to exchenge the order of integration.
 Since for every given~$\ell\in\R^n$, the set of~$x$ such that
 $\ell\in x-(r/2)\mathbf 1+[0,s)^n$ has measure at most~$s^n$, we obtain
 \begin{equation} \label{strangeintegrals1}
  \begin{split}
   \int_{\Omega_r} \int_{Q^n} g(\Phi_k(h, \, \xi; \, x, \, r)) \, \d \xi \, \d x
   \leq \int_{\R^n} g\!\left(\ell + \frac{Ns}{2}\mathbf 1, \, Ns\right) \d\ell
   = \int_{\R^n} g\!\left(u, \, Ns\right) \d u.
  \end{split}
 \end{equation}
 Next, we integrate over~$(s, \, r)$.
 Denoting~$G(\rho):=\int_{\R^n}g(u,\rho)\,\d u$, we obtain
 \[
  \begin{split}
   \int_{r_0/2}^{r_0}\!\int_{Q^n}\!\int_{\QQ_\Omega\cap\{r>2h_k\}}
    g(\Phi_k(h, \, \xi; \, x, \, r))\,\d x\,\d r\,\d\xi\,\frac{\d h}{h}
   &\leq \int_{a_k/2}^{a_k} \int_{2s}^{+\infty}
    G\!\left(Ns\right)\d r \, \frac{\d s}{s},
  \end{split}
 \]
 where, we recall, $N=\lfloor r/s\rfloor-1$.
 For each integer~$m\geq 1$, the equality~$N = m$ holds if and only if~$(m+1)s\leq r<(m+2)s$.
 Therefore, by partitioning the interval~$r\in[0, \, +\infty)$ according to the possible values of~$N$, we obtain
 \[
  \begin{split}
   \int_{r_0/2}^{r_0}\!\int_{Q^n}\!\int_{\QQ_\Omega\cap\{r>2h_k\}}
    g(\Phi_k(h, \, \xi; \, x, \, r))\,\d x\,\d r\,\d\xi\,\frac{\d h}{h}
   &\leq \sum_{m=1}^{+\infty}
    \int_{a_k/2}^{a_k} G(ms) \, \d s \\
   &=\sum_{m=1}^{+\infty}\frac{1}{m} \int_{ma_k/2}^{ma_k} G(\rho)\,\d\rho
  \end{split}
 \]
 Finally, we apply Fubini theorem at the right-hand side.
 We obtain
 \[
  \begin{split}
   \int_{r_0/2}^{r_0}\!\int_{Q^n}\!\int_{\QQ_\Omega\cap\{r>2h_k\}}
    g(\Phi_k(h, \, \xi; \, x, \, r))\,\d x\,\d r\,\d\xi\,\frac{\d h}{h}
   &\leq \int_{0}^{+\infty} G(\rho) H(\rho) \,\d\rho,
  \end{split}
 \]
 where
 \[
  H(\rho) :=  \sum_{m=1}^{+\infty}\frac{1}{m} \one_{[ma_k/2, \, ma_k]}(\rho)
  = \sum_{\substack{m\geq 1\\ \rho/a_k \leq m \leq 2\rho/a_k}}
     \frac{1}{m}.
 \]
 It is easily seen that the function~$H(\rho)$ is bounded uniformly with respect to the parameter~$k$, e.g.~by estimating the series from above with an integral.
 The lemma follows.
\end{proof}

Let~$S\in\I_0(\Omega; \, \G)$ be given.
We assume that the projections~$P(S, \, h_k, \, y)$ are well-defined for all~$k\geq 1$ and almost every~$y\in Q^n$; this condition is satisfied for almost every choice of~$h > 0$.
Then, we define
\begin{equation} \label{Dinidyadic}
 P_k^{h,\xi} := P(S, \, h_k, \, y_k(\xi))
\end{equation}
for all~$k\geq 0$.

\begin{lemma} \label{lemma:dyadic-projections}
 Let~$S\in\I_0(\Omega; \, \G)$.
 For almost every~$h>0$ and every~$k\geq 0$, there holds
 \begin{equation} \label{dyadic-mass-average}
  \int_{Q^n} \M(P_k^{h,\xi}) \,\d\xi
  = h_k^{-n} \int_{\Omega_{h_k}} \abs{\ind(S, \, Q_{h_k}^n(x))} \d x.
 \end{equation}
 Moreover, for almost every~$(h,\xi)\in(0,+\infty)\times Q^n$ and every~$k\geq 0$, there holds
 \begin{equation} \label{dyadic-flat-increment}
  \F_\Omega(P_{k+1}^{h,\xi}-P_k^{h,\xi})
  \leq C_n h_k \, \M(P_{k+1}^{h,\xi}),
 \end{equation}
 for some constant~$C_n$ depending only on~$n$.
\end{lemma}
\begin{proof}
 For every~$k$, the map~$\xi\mapsto y_k(\xi)$ preserves Lebesgue measure on the torus~$\R^n/\Z^n$: for every integrable~$\Z^n$-periodic function~$f$, there holds
 \begin{equation} \label{preservesLebesgue}
  \int_{Q^n} f(y_k(\xi))\,\d\xi
  = \int_{Q^n} f(y)\,\d y.
 \end{equation}
 Equation~\eqref{preservesLebesgue} follows by subdividing~$Q^n$ into~$2^{kn}$ cubes of size~$2^{-k}$ and making a linear change of variables, i.e.~$y=2^k\xi - z$ for some constant translation vector~$z$, on each of them.
 By applying~\eqref{preservesLebesgue} and~\eqref{deformation}, we obtain
 \[
  \begin{split}
   \int_{Q^n} \M(P_k^{h,\xi}) \, \d\xi
   &= \int_{Q^n} \M(P(S,h_k,y)) \, \d y \\
   &= \int_{Q^n} \sum_{z\in\Z^n}
    \one_{\{Q^n_{h_k}(h_ky+h_kz)\subseteq\Omega\}}
    \abs{\ind\!\left(S, \, Q^n_{h_k}(h_ky+h_kz)\right)} \,\d y .
  \end{split}
 \]
 Since the integrand is nonnegative, we can exchange the series and the integral.
 The change of variables~$x=h_ky+h_kz$ then gives
 \[
  \begin{split}
   \int_{Q^n}\M(P_k^{h,\xi})\,\d\xi
   &= h_k^{-n}\sum_{z\in\Z^n} \int_{h_kQ^n+h_kz}
    \one_{\Omega_{h_k}}(x) \abs{\ind(S, \, Q^n_{h_k}(x))} \d x \\
   &= h_k^{-n} \int_{\Omega_{h_k}} \abs{\ind(S, \, Q^n_{h_k}(x))} \d x,
  \end{split}
 \]
 because the cubes~$h_kQ^n+h_kz$ form a partition of~$\R^n$ up to their boundaries.
 This proves~\eqref{dyadic-mass-average}.

%  Up to discarding a negligible set of~$(h, \, \xi)$'s, all the relevant indices are well-defined and,
 Since the grid~$\GG(h_{k+1}, \, y_{k+1}(\xi))$ is a dyadic subdivision of~$\GG(h_k, \, y_k(\xi))$ and the intersection indices are additive under subdivision of a cube (see Remark~\ref{rk:index-subdivision}), for almost every~$(h, \, \xi)$ we have
 \[
  P_k^{h,\xi}=P(P_{k+1}^{h,\xi}, \, h_k, \, y_k(\xi)).
 \]
 Then, reasoning exactly as in the proof of~\eqref{Pflatmassconv}, we find a polyhedral $1$-chain~$B_k^{h,\xi}$ such that
 \[
  \partial B_k^{h,\xi} = P_{k+1}^{h,\xi}-P_k^{h,\xi}
  \qquad\textrm{in }\Omega,
  \qquad
  \M(B_k^{h,\xi})\leq C_n h_k\M(P_{k+1}^{h,\xi}).
 \]
 Therefore,~\eqref{dyadic-flat-increment} follows.
\end{proof}

Lemma~\ref{lemma:dyadic-projections}, combined with the assumption~\ref{hp:Dini-index}, will imply that~$(P^{h,\xi}_k)_{k\in\N}$ is a Cauchy sequence in~$\F_0(\Omega; \, \G)$, for suitable choices of~$(h, \, \xi)$.
However, in order to identify the limit of the~$P^{h,\xi}_k$'s and show that it coincides with~$S$, we need to compare the intersection indices of~$P^{h,\xi}_k$ on arbitrary cubes.
We will do that by making use of the approximations~$C_k^{h,\xi}$ defined in~\eqref{inner-grid-cube}.
More precisely, for every~$Q_r^n(x)\subseteq\Omega$ and almost every~$(h, \, \xi)$, we define
\begin{equation} \label{inner-grid-index}
 J_k^{h,\xi}(x, \, r)
 :=\sum_{\substack{K\in\GG(h_k, \, y_k(\xi))\\ K\subseteq C_k^{h,\xi}(x, \, r)}} \ind(S, \, K).
\end{equation}
First, in Lemma~\ref{lemma:grid-boundary-error}, we will compare~$\ind(P_k^{h,\xi}, \, Q^n_r(x))$ with~$J^{h,\xi}_k(x, \, r)$; then, in Lemma~\ref{lemma:random-cubical-reconstruction}, we will compare~$J^{h,\xi}_k(x, \, r)$ with~$\ind(S, \, Q_r^n(x))$.

\begin{lemma} \label{lemma:grid-boundary-error}
 There exists a constant~$C_{\Omega,n}>0$, depending only on~$\Omega$ and~$n$, such that
 \begin{equation} \label{grid-boundary-error}
  \abs{\left\{(x, \, r)\in\QQ_\Omega\colon 2h_k<r,\
  \ind(P_k^{h,\xi}, \, Q_r^n(x))\neq J_k^{h,\xi}(x, \, r)\right\}}
  \leq \frac{C_{\Omega,n}}{\alpha_\G}\,h_k\,\M(P_k^{h,\xi})
 \end{equation}
 for almost every~$(h, \, \xi)$ and every~$k\geq 0$.
\end{lemma}
\begin{proof}
 If~$\ind(P_k^{h,\xi},Q_r^n(x))\neq J_k^{h,\xi}(x, \, r)$, then there exists a cell~$K$ contributing a nonzero coefficient to~$P_k^{h,\xi}$ whose centre belongs to~$Q_r^n(x)\setminus C_k^{h,\xi}(x, \, r)$.
 By~\eqref{inner-grid-cube}, the centre of such a cell lies at distance at most~$2h_k$ from one of the faces of~$Q_r^n(x)$.
 For a given~$c\in\Omega$, the set of~$(x, \, r)\in\QQ_\Omega$ for which~$c\in Q_r^n(x)$ and~$c$ lies within distance~$2h_k$ of~$\partial Q_r^n(x)$ has measure at most~$C_{\Omega,n}h_k$.
 Indeed, for a given value of~$r$, the set of centres~$x$ satisfying the condition above is a boundary layer of thickness~$2h_k$ inside~$Q_r^n(c)$, whose measure is bounded by~$C_n r^{n-1}h_k$; integrating in~$r$ proves the claim.
 Summing over the nonzero atoms of~$P_k^{h,\xi}$ and using~\eqref{alpha_G} gives~\eqref{grid-boundary-error}.
\end{proof}

\begin{lemma} \label{lemma:random-cubical-reconstruction}
 Assume that the coefficient group~$\G$ is countable.
 Let~$S\in\I_0(\Omega; \, \G)$ and~$r_0 > 0$ be given.
 Then, there exists a sequence~$k_j\to+\infty$ such that, for almost every~$(h, \, \xi)\in(r_0/2, \, r_0)\times Q^n$ and almost every~$(x, \, r)\in\QQ_\Omega$, there holds
 \begin{equation} \label{random-cubical-reconstruction}
  J_{k_j}^{h,\xi}(x, \, r) = \ind(S, \, Q_r^n(x))
 \end{equation}
 for all sufficiently large~$j$ --- that is, for all~$j$ larger than some~$j_*(h, \, \xi; \, x, \, r)$.
\end{lemma}
\begin{proof}
 Let
 \[
  f(x, \, r):=\ind(S, \, Q_r^n(x)), \qquad (x, \, r)\in\QQ_\Omega.
 \]
 The map~$f$ is measurable: if~$S$ is polyhedral, then $S$ is a coubntable sum of indicator functions of polyherdal sets; for general~$S$, $f$ is the pointwise almost everywhere limit of a sequence of measurable functions (see Lemma~\ref{lemma:ICauchy} and Remark~\ref{rk:indconvergesae}).
 Remark~\ref{rk:index-subdivision} implies that the intersection index~$\ind(S, \, Q_r^n(x))$ is additive under a subdivision of a cube into~$2^{kn}$ congruent cubes, for every integer~$k\geq1$, except for a negligible set of~$(x, \, r)$'s.
 The estimate~\eqref{Phi-compression} shows that the inverse image under~$\Phi_k$ of this negligible set is negligible.
 Hence, for every~$k$, the map~$J_k^{h,\xi}$ defined in~\eqref{random-cubical-reconstruction} satisfies
 \begin{equation} \label{J-f-Phi}
  J_k^{h,\xi}(x, \, r) = f(\Phi_k(h, \, \xi; \, x, \, r))
 \end{equation}
 for almost every~$(h, \, \xi; \, x, \, r)$ such that~$h\in(r_0/2, \, r_0)$ and~$2h_k<r$.

 Now, the lemma will follow if we prove that
 \begin{equation*} %\label{random-index-conv-measure}
  \int_{r_0/2}^{r_0}\!\int_{Q^n}
  \abs{\left\{(x, \, r)\in\QQ_\Omega\colon
  J_k^{h,\xi}(x, \, r) \neq f(x, \, r)\right\}}
  \d\xi\,\frac{\d h}{h} \to 0
 \end{equation*}
 as~$k\to +\infty$. Actually, in view of~\eqref{J-f-Phi}, it will be enough to prove that
 \begin{equation} \label{random-index-conv-measure}
  \int_{r_0/2}^{r_0}\!\int_{Q^n}
  \abs{\left\{(x, \, r)\in\QQ_\Omega\colon 2h_k<r,\
  f(\Phi_k(h, \, \xi; \, x, \, r))\neq f(x, \, r)\right\}}
  \d\xi\,\frac{\d h}{h} \to 0
 \end{equation}
 as~$k\to +\infty$, because the total measure of the set where~$0 < r \leq 2h_k$ tends to zero anyway as~$k\to+\infty$.
 Thus, let us focus on the proof of~\eqref{random-index-conv-measure}.
 For every~$\eps > 0$ be given, we claim that there exists a compact set~$K\subseteq\QQ_\Omega$ such that~$\abs{\QQ_\Omega\setminus K}<\varepsilon$ and~$f\mres K$ is continuous.
 Indeed, since~$\G$ is countable, for all~$\eps > 0$ there exists a finite subset~$\G_0\subseteq\G$ such that~$\abs{f^{-1}(\G\setminus\G_0)}\leq\eps/2$.
 Moreover, for each~$\sigma\in\G_0$ there is a compact set~$K_\sigma\subseteq f^{-1}(\sigma)$ such that~$\abs{f^{-1}(\sigma)\setminus K_\sigma}\leq\eps/(\#\G_0)$.
 The claim follows by taking~$K:= \cup_{\sigma\in\G_0} K_\sigma$.
 Consequently, there exists~$\delta>0$ such that for all~$z\in K$, $z'\in K$,
 \begin{equation*} %\label{falmostconstant}
  \abs{z-z'} \leq \delta \qquad\Longrightarrow\qquad f(z)=f(z').
 \end{equation*}
 By~\eqref{inner-grid-cube-close}, we know that~$\abs{\Phi_k(h, \, \xi; \, x, \, r)-(x, \, r)}\leq Ch_k$, for some universal constant~$C$.
 Therefore, for all~$k$ large enough that~$2^{1-k}Cr_0\leq \delta$, we have the inclusion
 \[
  \left\{(x, \, r)\in\QQ_\Omega\colon
  f(\Phi_k(h, \, \xi; \, x, \, r))\neq f(x, \, r)\right\}
  \subseteq
  (\QQ_\Omega\setminus K)\cup\Phi_k(h, \, \xi; \, \cdot)^{-1}(\QQ_\Omega\setminus K)
 \]
 up to a negligible set.
 Now, %$\abs{\QQ_\Omega\setminus K}\leq\eps$ by construction, and
 Lemma~\ref{lemma:Phi-compression} implies
 \[
  \begin{split}
   \int_{r_0/2}^{r_0} \int_{Q^n}
    \abs{\left\{(x, \, r)\in\QQ_\Omega\colon 2h_k<r,\
    \Phi_k(h, \, \xi; \, x, \, r))\in
    \QQ_\Omega\setminus K\right\}} \d\xi\,\frac{\d h}{h}
   \leq C \abs{\QQ_\Omega\setminus K}
   \leq C \eps.
  \end{split}
 \]
 Since~$\varepsilon$ is arbitrary,~\eqref{random-index-conv-measure} follows.
\end{proof}

Finally, we are ready to prove Proposition~\ref{prop:Dini-flatness}.

\begin{proof}[Proof of Proposition~\ref{prop:Dini-flatness}]
 By applying Lemma~\ref{lemma:dyadic-projections}, Tonelli's theorem and the change of variables~$r=2^{-k}h$, we have
 \begin{equation} \label{Dini-mass-product}
  \begin{split}
   \int_{r_0/2}^{r_0}\int_{Q^n} \sum_{k=0}^{+\infty}
    h_k \, \M(P_k^{h,\xi})\,\d\xi\,\frac{\d h}{h}
   &\leq \sum_{k=0}^{+\infty} \int_{r_0/2}^{r_0} h_k^{1-n} \int_{\Omega_{h_k}} \abs{\ind(S, \, Q_{h_k}^n(x))} \d x \, \frac{\d h}{h} \\
   &= \sum_{k=0}^{+\infty} \int_{2^{-k-1}r_0}^{2^{-k}r_0} r^{1-n} \int_{\Omega_r} \abs{\ind(S, \, Q_{r}^n(x))} \d x \, \frac{\d r}{r} \\
   &= \int_0^{r_0}\frac{1}{r^n}
   \int_{\Omega_r}\abs{\ind(S, \, Q_r^n(x))}\,\d x\,\d r
   < +\infty.
  \end{split}
 \end{equation}
 The right-hand side is finite, because of the assumption~\eqref{hp:Dini-index}.
 Therefore, for almost every~$(h,\xi)\in(r_0/2, \, r_0)\times Q^n$, we have
 \begin{equation} \label{dyadic-mass-summable}
  \sum_{k=0}^{+\infty}h_k \, \M(P_k^{h,\xi})<+\infty.
 \end{equation}
 By~\eqref{dyadic-flat-increment}, for every such~$(h, \, \xi)$ the sequence~$(P_k^{h,\xi})_{k\in\N}$ is Cauchy with respect to the~$\F_\Omega$-norm.
 Hence, there exists~$T^{h,\xi}\in\F_0(\Omega; \, \G)$ such that
 $\F_\Omega(P_k^{h,\xi} - T^{h,\xi})\to 0$ as $k\to+\infty$.

 We claim that~$T^{h,\xi} = S$.
 Let
 \[
  E_k^{h,\xi}:=\left\{(x, \, r)\in\QQ_\Omega\colon 2h_k<r,\
  \ind(P_k^{h,\xi},Q_r^n(x))\neq J_k^{h,\xi}(x, \, r)\right\}.
 \]
 Lemma~\ref{lemma:grid-boundary-error} and~\eqref{Dini-mass-product} imply
 \[
  \int_{r_0/2}^{r_0}\int_{Q^n}
  \sum_{k=0}^{+\infty} \abs{E_k^{h,\xi}}\,\d\xi\,\frac{\d h}{h}<+\infty.
 \]
 As a consequence, Borel--Cantelli lemma implies that~$\limsup_{k\to +\infty} E_k^{h,\xi} := \bigcap_{K\geq 1}\bigcup_{k\geq K} E_k^{h,\xi}$ has measure zero for almost every~$(h, \, \xi)\in (r_0/2, \, r_0)\times Q^n$.
 In other terms, for almost every~$(h, \, \xi)\in (r_0/2, \, r_0)\times Q^n$ and almost every~$(x, \, r)\in\QQ_\Omega$, we have
 \begin{equation} \label{P-J-eventual}
  \ind(P_k^{h,\xi}, \, Q_r^n(x)) = J_k^{h,\xi}(x, \, r)
 \end{equation}
 for all sufficiently large~$k$ --- that is, for all~$k$ larger than some~$k_*(h, \, \xi; \, x, \, r)$.
 On the other hand, by Lemma~\ref{lemma:random-cubical-reconstruction}, we can extract a (non-relabelled) subsequence in such a way that
 \begin{equation} \label{S-J-eventual}
  \ind(S, \, Q_r^n(x)) = J_k^{h,\xi}(x, \, r)
 \end{equation}
 for almost every~$(h, \, \xi; \, x, \, r)$ and sufficiently large~$k$. Combining~\eqref{P-J-eventual} and~\eqref{S-J-eventual} with Remark~\ref{rk:indconvergesae}, we obtain
 \begin{equation*}
  \ind(S, \, Q_r^n(x)) = \ind(T^{h,\xi}, \, Q_r^n(x))
 \end{equation*}
 for almost every~$(x, \, r)\in\QQ_\Omega$.
 As a consequence, $S = T^{h,\xi}$ and~$S$ is flat.
\end{proof}

\section{Manifold-valued maps and their homotopy classes}
\label{sect:N}

\subsection{Sobolev spaces of manifold-valued maps}
\label{sect:Sobolev}

Let~$p \geq 1$ be \emph{an integer} and 
let~$\Omega\subseteq\R^{p+1}$ be a bounded, Lipschitz domain.
In order to simplify some statements, we will
adopt a rather \emph{non-standard terminology}
and say that a sequence
$(u_j)_{j\in\N}\subseteq W^{1,p}(\Omega, \, \R^m)$
converges weakly in~$W^{1,p}(\Omega)$ to some limit
map~$u\in W^{1,p}(\Omega, \, \R^m)$ if and only if
\begin{equation} \label{weakconv}
 u_j\to u \textrm{ strongly in } L^p(\Omega) \quad 
 \textrm{ and } \quad \sup_{j\in\N}\norm{\nabla u_j}_{L^p(\Omega)} < +\infty.
\end{equation}
This definition agrees with the usual one when~$p > 1$,
but for~$p = 1$ corresponds instead to~weak$^*$ convergence in~BV.
However, this definition allows us to treat the cases~$p > 1$
and~$p = 1$ at the same time.

Let~$\NN\subseteq\R^m$ be a smooth, compact, 
connected submanifold (without boundary) of some Euclidean space.
We define 
% the set of~$\NN$-valued $W^{1,p}$-maps by taking
% advantage of the Euclidean embedding~$\NN\subseteq\R^m$, that is,
\[
 W^{1,p}(\Omega, \, \NN)
 := \left\{u\in W^{1,p}(\Omega, \, \R^m)\colon u(x)\in\NN
 \ \textrm{for a.e. } x\in\Omega\right\} \! .
\]
We also write~$C^\infty(\overline{\Omega}, \, \NN)$
for the set of maps~$\varphi\colon\overline{\Omega}\to\NN$
that admit a smooth extension defined on an open neighbourhood
of~$\overline{\Omega}$.
% We write~$D_p$ for the $p$-Dirichlet functional, i.e.
% \begin{equation} \label{Dirichlet}
%  D_p(u) = D_p(u, \, \Omega)
%  := \int_{\Omega} \abs{\nabla u}^p \, \d x
% \end{equation}
% if~$u\in W^{1,p}(\Omega, \, \NN)$. %and~$D_p(u) := +\infty$ otherwise.
We consider the relaxed~$p$-Dirichlet functional, given by
\begin{equation} \label{relaxed}
 \begin{split}
  \overline{D}_p(u) 
%    = \overline{D}_p(u, \, \Omega)
   := \inf\!\big\{ &\liminf_{j\to+\infty}
   \int_{\Omega} \abs{\nabla\varphi_j}^p \d x \colon \\
   &(\varphi_j)_{j\in\N} \subseteq 
   C^\infty(\overline{\Omega}, \, \NN),
   \ \varphi_j\to u \textrm{ strongly in } L^p(\Omega) \big\}
 \end{split}
\end{equation}
for all~$u\in W^{1,p}(\Omega, \, \NN)$.
(As usual, we define~$\inf\varnothing = +\infty$.)
By definition, the functional~$\overline{D}_p$
is sequentially lower semicontinuous with respect to
weak~$W^{1,p}(\Omega)$-convergence. We write
\begin{equation} \label{Hweak}
 \Hw^{1,p}(\Omega, \, \NN)
 := \left\{u\in W^{1,p}(\Omega, \, \NN)\colon
  \overline{D}_p(u) < +\infty \right\} \! .
\end{equation}
The set~$\Hw^{1,p}(\Omega, \, \NN)$ is the sequential
closure of~$C^\infty(\overline{\Omega}, \, \NN)$
with respect to weak~$W^{1,p}$-convergence. We also write
$\Hs^{1,p}(\Omega, \, \NN)$ for the closure
of~$C^\infty(\overline{\Omega}, \, \NN)$ with respect
to the \emph{strong} $W^{1,p}$-topology.

% \begin{goal} \label{goal:relaxed}
%  {\BBB The function~$\overline{D}_p\colon \Hw^{1,p}(\Omega, \, \NN)\to\R$
%  is continuous with respect to the \emph{strong} $W^{1,p}$-topology.}
% \end{goal}

% Following the terminology of~\cite{ABO1, ABO2},
We will say that a map~$u\colon\overline{\Omega}\to\R^m$
% has \emph{nice singularities} 
belongs to~$R^{1,p}(\Omega, \, \NN)$ if and only if there exist
a finite set~$\Sigma = \Sigma(u)$ and a constant~$C$ such
that~$u$ is locally Lipschitz 
in~$\overline{\Omega}\setminus\Sigma$ and
\begin{equation} \label{nicesing}
 \abs{\nabla u(x)} \leq \frac{C}{\dist(x, \, \Sigma)}
 \qquad \textrm{for a.e. } x\in\Omega\setminus \Sigma.
\end{equation}
In the terminology of~\cite{ABO1,ABO2},
maps in~$R^{1,p}(\Omega, \, \NN)$ have ``nice singularities''.
Any $\NN$-valued Sobolev map can be strongly approximated
by maps with nice singularities.

\begin{theorem} \label{th:nicedensity}
 The set~$R^{1,p}(\Omega, \, \NN)$ is dense 
 in~$W^{1,p}(\Omega, \, \NN)$, with respect to the strong topology.
\end{theorem}

Theorem~\ref{th:nicedensity} was proved by Bethuel
in~\cite[Theorem~2]{Bethuel-Density}.
% (The statement in~\cite{Bethuel-Density} does not mention
% explicitly that the approximating maps satisfy gradient bounds
% as in~\eqref{nicesing}, but this follows from the proof).
Generalisations of this result to higher-order or fractional Sobolev spaces
are given in~\cite{BrezisMironescu2015, BousquetPonceVanSchaftingen-II, Detaille2026}.
In Section~\ref{sect:RS} below, we will revisit the arguments of~\cite{BousquetPonceVanSchaftingen-II} and prove a variant of Theorem~\ref{th:nicedensity} that will be more suitable for our purposes.

\subsection{Free homotopy classes of maps~$\SS^{p}\to\NN$}
\label{sect:homotopy}

From now on, $p\geq 1$ will denote \emph{an integer}.
Given continuous maps~$h_1$, $h_2\colon\SS^{p}\to\NN$,
we will say that~$h_1$, $h_2$ are freely homotopic
(or just homotopic, for short) to one another
if and only if there exists a continuous map
$H\colon \overline{B}^{p+1}_2\setminus B^{p+1}\to\NN$ such that
\[
 H(x) = h_1(x), \quad H(2x) = h_2(x)
 \qquad \textrm{for any } x\in \partial B^{p+1}.
\]
Since the annulus~$\overline{B}^{p+1}_2\setminus B^{p+1}$
is homeomorphic to the cylinder~$[0, \, 1]\times\SS^{p}$,
this definition of free homotopy is equivalent
to the usual one, in which the map~$H$ is defined
on~$[0, \, 1]\times\SS^{p}$. However, free homotopy
differs substantially from, say, based homotopy,
which is used to define the homotopy groups,
in that the curves~$h_1$, $h_2$ are not required to pass
through a given base point in~$\NN$.
Free homotopy is an equivalence relation
on~$C(\SS^{p}, \, \NN)$,
and we denote the set of equivalence classes as~$\GN$.
The set~$\GN$ can be described in terms of more standard
objects in algebraic topology, namely homotopy groups;
see Remark~\ref{rk:conjugacy}.

Let~$\mathscr{P}(\GN)$ denote the power set of~$\GN$.
There is a natural binary \emph{multi-valued}
operation in~$\GN$, i.e. a map
\[
 +\colon \GN\times\GN\to\mathscr{P}(\GN),
\]
which is defined as follows.
Let~$D\subseteq\R^{p+1}$ be a closed ball with two holes,
say for instance~$D := D_0 \setminus(D_1 \cup D_2)$, where
\begin{equation} \label{domain+}
 D_0 := \overline{B}^{p+1}_2,
 \qquad D_1 := B^{p+1}_{1/2}(-1, \, 0, \, \ldots, \, 0),
 \qquad D_2 := B^{p+1}_{1/2}( 1, \, 0, \, \ldots, \, 0).
\end{equation}
We identify all the boundary components of~$D$,
i.e.~$\partial D_0$, $\partial D_1$, $\partial D_2$,
with the unit circle~$\SS^{p}$, up to translation and dilations.
% (Note that translations and dilations preserve the orientation.)
In particular, we give all the boundary components the same
orientation, the same we give to~$\SS^{p}$.
Given free homotopy classes~$\sigma_1$ and~$\sigma_2$ in~$\GN$,
we define~$\sigma_1 + \sigma_2$ as the set of homotopy
classes~$\sigma\in\GN$ for which
there exists a continuous map~$G\colon D\to\NN$ such that
$G$ restricted to~$\partial D_1$, $\partial D_2$, $\partial D_0$
respectively belongs to the homotopy
classes~$\sigma_1$, $\sigma_2$, $\sigma$.
We extend~$+$ to an $n$-ary operation,
for any integer~$n\geq 2$, by setting e.g.
\[
 (\sigma_1 + \sigma_2) + \sigma_3
 := \bigcup_{\sigma\in \sigma_1 + \sigma_2} \left(\sigma + \sigma_3\right)
\]
and so on.

\begin{lemma} \label{lemma:polygroup}
 The operation~$+$ on~$\GN$ satisfies the following properties:
 \begin{enumerate}[label=(\arabic*)]
  \setcounter{enumi}{-1}
  \item $\sigma_1 + \sigma_2$ is nonempty for all~$\sigma_1$, $\sigma_2\in\GN$;
  \item for all~$\sigma_1$, $\sigma_2$ and~$\sigma_3\in\GN$, there holds
  $(\sigma_1 + \sigma_2) + \sigma_3 = \sigma_1 + (\sigma_2 + \sigma_3)$;
  \item there exists~$0\in\GN$ such that~$\sigma + 0 = \{\sigma\}$
  for all~$\sigma\in\GN$;
  \item for all~$\sigma\in\GN$ there exists~$-\sigma\in\GN$
  such that~$\sigma + (-\sigma) \ni 0$;
  \item for all~$\sigma_1$ and~$\sigma_2\in\GN$, there holds
  $\sigma_1 + \sigma_2 = \sigma_2 + \sigma_1$.
 \end{enumerate}
\end{lemma}
\begin{proof}[Sketch of the proof]
 The proof of Lemma~\ref{lemma:polygroup}
 follows from elementary arguments in topology.
 We sketch the arguments, but omit the details
 for the sake of brevity.
 Property~(0) depends on the assumption
 that~$\NN$ is (path-)connected, which allows us to extend
 any continuous map~$g\colon\partial D_1\cup\partial D_2\to\NN$
 to a continuous map~$G\colon D\to\NN$. 
 For instance, we consider a small ``tube''~$\Sigma$,
 homeomorphic to~$[0, \, 1]\times\overline{B}^{p-1}$,
 joining a small geodesic disk~$B_\eps(x_1)\subseteq D_1$ 
 with a geodesic disk in~$B_\eps(x_2) \subseteq D_2$. 
 We define~$G$ inside~$\Sigma$ by travelling along geodesic paths
 in~$\NN$ that join~$g_{|B_\eps(x_1)}$ with~$g_{|B_\eps(x_2)}$.
 Then, we extend $G$ to the whole of~$D$ by observing 
 that~$D$ retracts onto~$\Sigma\cup \partial D_1\cup\partial D_2$,
 that is, there is a continuous
 map~$D\to \Sigma\cup \partial D_1\cup\partial D_2$ which is equal
 to the identity on~$\Sigma\cup \partial D_1\cup\partial D_2$.
 The class~$0\in\GN$ is the homotopy class of constant maps,
 and~$-\sigma$ is the homotopy class that contains
 the continuous map~$\overline{h}\colon\SS^{p}\to\NN$
 obtained from an arbitrary~$h\in\sigma$ by reversing the orientation.
 Property~(4) follows from the fact that there exists
 a diffeomorphism~$D\to D$ that coincides with the identity
 on~$\partial D_0$ but exchanges~$\partial D_1$ with~$\partial D_2$.
 Such a diffeomorphism can be constructed, for instance, by
 rotating each sphere~$\partial B^{p+1}_r$ 
 about a given axis~$e\in\SS^{k}$ by an angle~$\theta(r)$,
 which increases continuously from~$0$ to~$\pi$
 as~$r$ decreases from~$2$ to~$3/2$.
\end{proof}

Lemma~\ref{lemma:polygroup} shows that~$\GN$
is a (commutative) polygroup, that is,
it enjoys algebraic properties that are similar to 
those of groups, except that the operation
is multi-valued. Note that we may have
$\sigma + (-\sigma) \supsetneq \{0\}$,
as shown in Remark~\ref{rk:multivalued}.
In view of Lemma~\ref{lemma:polygroup}, we can adopt
some notation that is commonly used for the operation of group,
such as dropping the parenthesis in $n$-ary sums and
writing~$\sigma_1 - \sigma_2$ instead of~$\sigma_1 + (-\sigma_2)$.

\begin{remark} \label{rk:orientation}
 If~$X$ is a topological manifold 
 and~$\phi\colon\SS^p\to X$ an homeomorphism,
 it makes sense to define the homotopy class 
 of a map~$u\colon X\to\NN$ as the homotopy class 
 of~$u\circ\phi\colon\SS^p\to\NN$. 
 However, these classes \emph{could} depend on
 the choice of~$\phi$. To avoid this ambiguity,
 whenever~$X$ is a smooth manifold 
 (or a piecewise linear one, such as the
 boundary of a convex polyhedron) we fix an
 orientation of~$X$ and choose~$\phi$
 as an \emph{orientation-preserving} diffeomorphism.
 Choosing another orientation-preserving 
 diffeomorphism will not change the homotopy
 class~$\sigma$ of~$u$, because orientation-preserving
 diffeomorphisms~$X\simeq\SS^p\to\SS^p$
 are homotopic to one another. However, changing 
 the orientation of~$X$ will change the homotopy 
 class of~$u$ from~$\sigma$ to~$-\sigma$. 
 In a similar way, it makes sense to define the
 homotopy class of a map~$u\colon\R^p\to\NN$ that
 is constant in a neighbourhood of infinity,
 by composition with the inverse stereographic
 projection~$\R^p\to\SS^p\setminus\{(1, \,  0, \, \ldots, \, 0)\}$,
 which is orientation-preserving.
\end{remark}

In the next lemma, we show that ``changing the base point''
does not affect free homotopy classes.
We will denote~$I := [0, \, 1]$.
More precisely, let~$z_0\in\NN$, $z_1\in\NN$
be given points in~$\NN$, $\gamma\colon I\to\NN$
be a continuous path from~$z_0$ to~$z_1$,
and let~$v_0\colon\R^p\to\NN$ be a continuous map
such that~$v_0 = z_0$ in~$\R^p\setminus B^p$.
We define
\begin{equation} \label{actionofpi1}
 v_1(x) :=
 \begin{cases}
  v_0(x) &\textrm{if } \abs{x} \leq 1, \\
  \gamma(\abs{x} - 1) &\textrm{if } 1 < \abs{x} \leq 2, \\
  z_1 &\textrm{if } \abs{z} > 2.
 \end{cases}
\end{equation}

\begin{lemma} \label{lemma:actionofpi1}
 The maps~$v_0$ and~$v_1$ belong to the same
 homotopy class in~$\GN$.
\end{lemma}
\begin{proof}
 Let~$H\colon \R^p\times [0, \, 1]\to\NN$ be given by
 \[
  H(x, \, t) :=
  \begin{cases}
   v_0(x) &\textrm{if } \abs{x} \leq 1, \\
   \gamma(\abs{x} - 1) &\textrm{if } 1 < \abs{x} \leq 1 + t, \\
   \gamma(t)           &\textrm{if } \abs{z} > 1 + t.
  \end{cases}
 \]
 The map~$H$ is continuous, with~$H(\cdot, \,  0) = v_0$
 and~$H(\cdot, \, 1) = v_1$.
 Upon applying the stereographic projection,
 $H$ induces a homotopy between~$v_0$ and~$v_1$
 in an annulus $\overline{B}_2^{p+1}\setminus B_1^{p+1}$.
\end{proof}

\begin{remark} \label{rk:conjugacy}
 In case~$z_1 = z_0$, Equation~\eqref{actionofpi1}
 defines an action of~$\pi_1(\NN, \, z_0)$
 on~$\pi_p(\NN, \, z_0)$; more precisely, the based homotopy class of a loop~$\gamma$ in~$\pi_1(\NN, \, z_0)$ acts on the based homotopy class of~$v_0$ in~$\pi_p(\NN, \, z_0)$ and gives the based homotopy class of~$v_1$ in~$\pi_1(\NN, \, z_0)$.
 Lemma~\ref{lemma:actionofpi1} shows that each orbit of
 this action is contained in an element of~$\Gamma_p(\NN)$.
 In fact,
%  Conversely, each element of~$\Gamma_p(\NN)$ is contained
%  in an orbit of this action and
 there is a natural bijection between elements of~$\GN$
 and orbits of the action of the fundamental
 group~$\pi_1(\NN, \, z_0)$ on~$\pi_{p}(\NN, \, z_0)$
 (see e.g.~\cite[Sections~VI and~VII.D]{Mermin} for a self-contained explanation and~\cite[Proposition~4A.2 p.~422]{Hatcher} for a rigorous statement and proof).
%  A proof of this claim in case~$p=1$ is given,
%  e.g., in~\cite[]{pirla};
%  the proof in case~$p>1$ is similar.
%  By identifying elements of~$\GN$
%  with orbits in~$\pi_{p}(\NN, \, z_0)$
%  under the action of~$\pi_1(\NN, \, z_0)$,
 By taking advantage of this bijection
 and identifying elements of~$\GN$ with
 subsets of~$\pi_{p}(\NN, \, z_0)$,
 we can express the operation~$+$ on~$\GN$ as
 \[
  \sigma_1 + \sigma_2
  = \left\{(\textrm{orbit of } s_1 \cdot s_2)\colon
  s_1\in\sigma_1, \ s_2\in\sigma_2 \right\} \! ,
 \]
 where~$\cdot$ denotes the operation on~$\pi_p(\NN, \, z_0)$.
\end{remark}

\begin{remark} \label{rk:multivalued}
 Suppose that~$p = 1$ and~$\pi_1(\NN)$
 is not Abelian.
 Since~$\pi_1(\NN, \, z_0)$ acts on itself by conjugation,
 $\sigma + (-\sigma)$ contains all the
 conjugacy classes in~$\pi_1(\NN, \, z_0)$
 with elements of the form~$s g s^{-1} g^{-1}$,
 for~$s\in\sigma$ and~$g\in\pi_1(\NN, \, z_0)$.
 If~$s$ does not commute with all the other elements
 of~$\pi_1(\NN, \, z_0)$ then~$\sigma + (-\sigma)$ contains
 more than one homotopy class.
\end{remark}

For later convenience, we state a few
consequences of the definition of~$+$.
Let~$I := [0, \, 1]$ and let~$R^{p}\subseteq \partial I^{p+2}$
be the $p$-skeleton of the~$(p+2)$-cube~$I^{p+2}$.
Let~$u\colon R^p\to\NN$ be a continuous map.
For each boundary $(p+1)$-face~$K$ of~$I^{p+2}$
(with the orientation induced by~$I^{p+2}$),
we let~$\sigma(u, \, \partial K)\in\GN$ be
the free homotopy class of~$u$ restricted to~$\partial K$.
This is well-defined, consistently with the definition 
given above, because~$\partial K$ is homeomorphic 
to the sphere~$\SS^p$.

\begin{lemma} \label{lemma:homotopysum-cubes}
 For each boundary $(p+1)$-face~$K_0$ of~$I^{p+2}$,
 and each~$u\in C(R^p, \, \NN)$, we have
 \begin{equation} \label{homotopysum}
  \sigma(u, \, \partial K_0) \in 
  \sum_{K\neq K_0} \sigma(u, \, \partial K),
 \end{equation}
 where the sum is taken over the $(p + 1)$-faces~$K$
 of~$I^{p+2}$ with~$K\neq K_0$.
\end{lemma}
\begin{proof}
 Let~$p_0$ be a  point in the interior of~$K_0$.
 Since~$\partial T^{p+2}$ is homeomorphic to~$\SS^{p+1}$,
 the stereographic projection induces an 
 homeomorphism~$\partial I^{p+2}\setminus\{p_0\}\to\R^{p+1}$
 that maps~$\partial K_0$ into a ($p$-dimensional) sphere.
 The image~$\widetilde{R}^p$ of~$R$ under that homeomorphism
 is homotopically equivalent to --- in fact, a deformation retract of ---
 a disk~$D_p$ with~$2^{p+2} - 2$ holes,
 one for each boundary face~$K$ of~$I^{p+2}$ other than~$K_0$.
 By composition with the homeomorphism and the retraction,
 the map~$u$ on~$R^p$ induces a continuous map~$v\colon D_p\to\NN$,
 such that the homotopy class of~$v$ on each boundary component
 of~$D_p$ is the same as the homotopy class of~$u$ 
 on the boundary of the associated~$K$. 
 Then, \eqref{homotopysum} follows from the definition of 
 the operation~$+$ on~$\GN$.
\end{proof}

% In the next lemma, we consider continuous maps~$\R^p\to\NN$
% that are constant in a neighbourhood of infinity.
% Any such map has a well defined homotopy class
% in~$\GN$, by composition with the
% stereographic projection that maps~$\R^p$ to~$\SS^p$ minus a point.
Let~$z_0\in\NN$ be a point, let~$B_1$, $B_2$
be closed disjoint balls in~$\R^p$, and let~$v_1$, $v_2$
be continuous maps~$\R^p\to\SS^p$ such that~$v_1 = z_0$ in~$\R^p\setminus B_1$, $v_2 = z_0$ in~$\R^p\setminus B_2$.
We denote the homotopy classes of~$v_1$, $v_2$
by~$\sigma_1$, $\sigma_2$, respectively.
Let~$v\colon\R^p\to\NN$ be defined as~$v := v_1$ in~$B_1$,
$v := v_2$ in~$B_2$, and~$v := z_0$ elsewhere.

\begin{lemma} \label{lemma:homotopysum}
 The homotopy class of~$v$ belongs to~$\sigma_1 + \sigma_2$.
\end{lemma}
\begin{proof}
 Upon translating and rescaling the balls~$B_1$, $B_2$
 (which does not affect the homotopy classes) 
 and composing with the stereographic projections,
 we can identify~$v$ with a map defined on~$\partial B^{p+1}_2$
 that is equal to~$z_0$ everywhere, except 
 for two small geodesic disks~$B_1^\eps$, $B_2^\eps$
 of radius~$\eps < 1/100$
 centred at the points~$(-2, \, 0, \, \ldots, \, 2)$
 and~$(2, \, 0, \, \ldots, \, 0)$.
 Let~$D := \overline{B}^{p+1}_2\setminus (D_1\cup D_1)$,
 with~$D_1$, $D_2$ as in~\eqref{domain+}.
 We define a continuous map~$G\colon D\to\NN$ as follows:
 \[
  G(x) := 
  \begin{cases}
   v\!\left(\dfrac{2x}{\abs{x}}\right) 
    &\textrm{if } x\in D, \abs{x}\geq 1, \\
   z_0 &\textrm{if } x\in D, \abs{x} < 1.
  \end{cases}
 \]
 The map~$G$ is constant everywhere except
 for two small ``tubes'', of thickness proportional
 to~$\eps$, which connect the disks~$B^1_\eps$, $B^2_\eps$
 with~$\partial D_1$, $\partial D_2$.
 The homotopy class of~$G$ restricted to~$\partial D_1$,
 $\partial D_2$ is the same as the homotopy class
 of~$v$ restricted to~$D_1^\eps$, $D_2^\eps$,
 respectively. Therefore, the lemma follows from
 the definition of~$+$.
\end{proof}

\subsubsection{Homotopy classes in critical Sobolev spaces}
\label{sect:homotopy_W1p}

% Given a finite atlas~$\{(U_i, \, \varphi_i)\}_{i=1}^{N}$ for~$M$,
% where each~$\varphi_i\colon U_i\subseteqM\to \varphi_i(U_i)\subseteq\R^p$
% is a coordinate chart, we define~$W^{1,p}(M, \, \R^m)$
% as the set of measurable maps~$u\colonM\to\R^m$
% such that~$u\circ\varphi_i^{-1}\in W^{1,p}(\varphi_i(U_i), \, \R^m)$
% for all indices~$i$ and~$u(x)\in\NN$ for~$\H^p$-almost all~$x\inM$.
% (It can be checked that this definition is independent of
% the choice of the atlas; see e.g.~{\BBB\textbf{REF}} for details).
The notion of homotopy class extends to Sobolev maps
in~$W^{1,p}(\SS^p, \, \NN)$, although those are not
necessarily continuous. (In fact, Brezis and Nirenberg~\cite{BN1}
showed that homotopy classes are well-defined
in an even larger class, i.e. the space~$\mathrm{VMO}(\SS^p, \, \NN)$,
but we will not need such generality here.) 
The construction is based on the tubular neighbourhood theorem,
which implies the existence of a nearest-point projection onto~$\NN$
defined in a neighbourhood of~$\NN$ of uniform thickness.
More precisely, there exists~$\delta_* = \delta_*(\NN) > 0$
with the following properties: for all~$y\in\R^m$
with~$\dist(y, \, \NN)\leq \delta_*$,
there exists a unique~$\pi_{\NN}(y)\in\NN$ such that
\begin{equation} \label{projN}
 \abs{y - \pi_\NN(y)} < \abs{y - z} 
 \qquad \textrm{for all } z\in\NN, \, y\neq z,
\end{equation}
and the map~$y\mapsto\pi_\NN(y)$ is smooth.
% (See e.g.~{\BBB\textbf{Simon}}).

The projection~$\pi_\NN$ allows us to define homotopy
classes of maps in~$W^{1,p}(\SS^p, \, \NN)$, as follows.
For all~$u\in W^{1,p}(\SS^p, \, \NN)$, all~$x\in\SS^{p}$
and all~$\eps > 0$, let
\begin{equation} \label{average}
 \overline{u}_\eps(x) :=
 \fint_{B_\eps(x)\cap\SS^p} u(y) \, \d\H^p(y),
\end{equation}
where the integral average is taken component-wise.
The map~$x\in\SS^p\mapsto \overline{u}_\eps(x)$ is continuous,
by Lebesgue's dominated convergence theorem.
While~$\overline{u}_\eps(x)$ is not necessarily an element of~$\NN$,
its distance from~$\NN$ can be estimated using the Poincar\'e inequality.
More precisely, since~$u(y)\in\NN$ for a.e.~$y\in\SS^p$, we have
\begin{equation} \label{Poincare-closetoN}
 \begin{split}
  \dist(\overline{u}_\eps(x), \, \NN)^p
  &\leq \fint_{B_\eps(x)\cap\SS^p} \abs{u(y) - \overline{u}_\eps(x)}^p\d\H^p
  \leq C\int_{B_\eps(x)\cap\SS^p} \abs{\nabla u(y)}^p\d\H^p.
 \end{split}
\end{equation}
The constant~$C$ in front of the right-hand side
can be chosen independently of ($x$ and)~$\eps$,
by scaling reasons. As a consequence, 
there exists~$\eps_*(u) > 0$ such that
$\dist(\overline{u}_\eps(x), \, \NN)^p\leq\delta_*$
for all~$x\in\SS^p$ and all~$\eps \in (0, \, \eps_*(u)]$,
so that the projection~$u_\eps := \pi_{\NN}\circ\overline{u}_\eps$
is well-defined and continuous.
For all~$\eps$, $\eps^\prime\in (0, \, \eps_*(u)]$,
the maps~$u_\eps$, $u_{\eps^\prime}$ are homotopic to one another;
a homotopy (in the sense previously defined) is given by
$H_\eps(r\omega) := u_{(r-1)\eps + (2-r)\eps^\prime}(\omega)$
for all~$r\in [1, \, 2]$, $\omega\in\SS^p$.
We define the homotopy class of~$u\in W^{1,p}(\SS^p\, \, \NN)$
as the homotopy class of~$u_\eps$, for arbitrary~$\eps\in (0, \, \eps_*(u)]$.
The theory of homotopy classes for Sobolev 
and, more generally, VMO maps
is carried out in~\cite{BN1, BN2}. 
Here, we recall only two properties.

\begin{prop}[{\cite{White86}, \cite[Theorem~1]{BN1}}]
\label{prop:smallenergy}
 For any~$u\in W^{1,p}(\SS^p, \, \NN)$, there
 exists a positive value~$\eta = \eta(u, \, p, \, \NN)$
 such that any map~$v\in W^{1,p}(\SS^p, \, \NN)$ satisfying
 \[
  \norm{\nabla u - \nabla v}_{L^p(\SS^p)}
   + \norm{u - v}_{L^p(\SS^p)} \leq \eta
 \]
 belongs to the same homotopy class as~$u$.
 In particular, there exists~$\alpha_p(\NN) > 0$
 such that any map~$u\in W^{1,p}(\SS^p, \, \NN)$
 with~$\norm{\nabla u}_{L^p(\SS^p)}^p \leq \alpha_p(\NN)$
 belongs to the trivial homotopy class.
\end{prop}

\begin{prop}[{\cite[Proposition p.~267]{SchoenUhlenbeck2}}]
\label{prop:densitysmooth}
 The set~$C^\infty(\SS^p, \, \NN)$ is dense
 in~$W^{1,p}(\SS^p, \, \NN)$, with respect to the strong topology. 
\end{prop}

\subsubsection{Norms of homotopy classes}
\label{sect:norms}

We now define a function --- or,
in case~$p = 1$, a family of functions ---
on~$\GN$ that satisfies ``norm-like'' properties.
Let~$z_0\in\NN$ be a given basepoint,
and let~$N_p := (1, \, 0, \, \ldots, \, 0)\in\SS^p$
be the North pole of the sphere.
For all~$\sigma\in\GN$, we define
\begin{equation} \label{norm_p}
 \abs{\sigma}_{z_0} := \inf\left\{
  \int_{\SS^{p}} \abs{\nabla_\top v}^p \, \d\H^p \colon
  v\in (W^{1,p}\cap C)(\SS^p, \, \NN), \ v \in \sigma, \ v(N_p) = z_0 \right\} \! .
\end{equation}
Here~$\nabla_\top$ denotes the Riemannian gradient on~$\SS^{p}$,
applied component-wise to $v = (v^1, \, \ldots, \, v^m)$.
We do \emph{not} claim that the infimum at the right-hand
side of~\eqref{norm_p} is a minimum.
However, we do claim that the infimum is finite, that is,
the set at the right-hand side of~\eqref{norm_p} is not empty;
this is a consequence of Lemma~\ref{lemma:actionofpi1} and the assumption that~$\NN$ is path-connected.
As we will see momentarily, when~$p\geq 2$ the function~$\abs{\,\cdot\,}_{z_0}$ actually turns out to be independent of~$z_0$.
By contrast, when~$p = 1$, different choices of the basepoint~$z_0$ may lead to different functions, although all of them are comparable to one another, as made precise by Lemma~\ref{lemma:norm_p} below.

Before we proceed to the proof of these claims, we write~\eqref{norm_p} in equivalent ways.
% First, since smooth maps are dense in~$W^{1,p}(\SS^p, \, \NN)$
% (Proposition~\ref{prop:densitysmooth}),
% we can restrict the infimum on the right-hand side
% to smooth maps~$v$ in the homotopy class~$\sigma$.
By a density argument, there is no loss of generality in
restricting the infimum in~\eqref{norm_p}
to maps~$v$ that are equal to~$z_0$ in a neighbourhood of~$N_p$.
Then, since the $p$-Dirichlet energy is conformally invariant,
up to composition with the stereographic
projection~$\SS^p\setminus\{N_p\}\to\R^p$ we can write
\begin{equation} \label{norm_plane}
 \begin{split}
  \abs{\sigma}_{z_0} = \inf\bigg\{
   \int_{\R^p} \abs{\nabla v}^p \d x
   \colon &v\in C(\R^p, \, \NN), \
   \nabla v \in L^p(\R^p, \, \R^{m\times p}),
   \ v \in \sigma, \\[-3mm]
   &v = z_0 \textrm{ in a neighbourhood of infinity} \bigg\} .
 \end{split}
\end{equation}
By rescaling, we also have
\begin{equation} \label{norm_ball}
 \begin{split}
  \abs{\sigma}_{z_0} = \inf\bigg\{
   \int_{B^p} \abs{\nabla v}^p \d x
   \colon &v\in (W^{1,p}\cap C)(B^p, \, \NN),
   \ v \in \sigma, \ v = z_0 \textrm{ on } \partial B^p \bigg\} .
 \end{split}
\end{equation}
These reformulations will be useful in proving
the following result.

\begin{lemma} \label{lemma:norm_p}
 The function~$\abs{\,\cdot\,}_{z_0}\colon\GN\to\R$ satisfies
 the following properties:
 \begin{enumerate}[label=(\roman*), ref=\roman*]
  \item $\abs{-\sigma}_{z_0} = \abs{\sigma}_{z_0}$
  for all~$\sigma\in\GN$;
  \item for all~$\sigma_1\in\GN$ and~$\sigma_2\in\GN$,
  there exists~$\sigma\in\sigma_1 + \sigma_2$
  such that~$\abs{\sigma}_{z_0} \leq \abs{\sigma_1}_{z_0} + \abs{\sigma_2}_{z_0}$;
  \item \label{item:alpha_p}
  there exists a constant~$\alpha_p(\NN) > 0$,
  depending only on~$p$ and~$\NN$, such that
  \begin{equation*}
   \abs{\sigma}_{z_0} \geq \alpha_p(\NN) \qquad
   \textrm{for all } \sigma\in\GN, \ \sigma\neq 0, \ z_0\in\NN;
  \end{equation*}
  \item \label{item:equivalentnorms}
  if~$p = 1$, there exists a constant~$C>0$,
  depending only on~$\NN$, such that
  \begin{equation*}
   \abs{\sigma}_{z_1} \leq C \abs{\sigma}_{z_0}
  \end{equation*}
  for all~$z_0\in\NN$, $z_1\in\NN$, and all~$\sigma\in\GN$;
  \item \label{item:samenorms}
  if~$p \geq 2$, then
  \begin{equation*}
   \abs{\sigma}_{z_1} = \abs{\sigma}_{z_0}
  \end{equation*}
  for all~$z_0\in\NN$, $z_1\in\NN$, and all~$\sigma\in\GN$.
 \end{enumerate}
\end{lemma}

\begin{proof}%[Proof of Proposition~\ref{goal:norm_p}]
 We prove each statement separately.

 \medskip
 \noindent
 \textit{Proof of~(i).}
 This property follows immediately from the definition~\eqref{norm_p}
 of~$\abs{\,\cdot\,}_{z_0}$.

 \medskip
 \noindent
 \textit{Proof of~(ii).}
 We use the characterisation~\eqref{norm_ball}
 of~$\abs{\,\cdot\,}_{z_0}$.
 Let~$v_1\in (W^{1,p}\cap C)(B^p, \, \NN)$,
 $v_2(W^{1,p}\cap C)(B^p, \, \NN)$ be maps, with trace
 on~$\partial B^n$ equal to~$z_0$
 belonging to the homotopy classes~$\sigma_1$, $\sigma_2$, respectively.
 Let~$\overline{B}_1(x_1)$, $\overline{B}_1(x_2)$ be disjoint,
 closed balls in~$\R^p$ of radius~$1$. Let~$v\colon\R^p\to\NN$ be defined
 as~$v(x) := v_1(x - x_1)$ for~$x \in B_1(x_1)$,
 $v(x) := v_2(x - x_2)$ for~$x\in B_1(x_2)$, and~$v(x) := z_0$ otherwise.
 Lemma~\ref{lemma:homotopysum} shows that the
 homotopy class~$\sigma$ of~$v$ belongs to~$\sigma_1 + \sigma_2$,
 so~\eqref{norm_plane} implies
 \[
  \abs{\sigma}_{z_0}
  \leq \int_{\R^p} \abs{\nabla v}^p \d x
  = \int_{\R^p} \abs{\nabla v_1}^p \d x
   + \int_{\R^p} \abs{\nabla v_2}^p \d x.
 \]
 Statement~(ii) follows by taking
 the infimum over all admissible~$v_1\in\sigma_1$,
 $v_2\in\sigma_2$.

 \medskip
 \noindent
 \textit{Proof of~(iii).} This is an immediate consequence
 of Proposition~\ref{prop:smallenergy}.

 \medskip
 \noindent
 \textit{Proof of~(iv).}
 Assume~$p = 1$.
 We use again the formulation~\eqref{norm_ball} of the problem.
 Let~$z_0\in\NN$, $z_1\in\NN$ be given points in~$\NN$
 and~$\sigma\in\pN$. We can assume~$\sigma\neq 0$,
 otherwise there is nothing to prove.
 Let~$\gamma\colon [0, \, 1]\to\NN$ a minimising geodetic
 (parametrised by multiples of arclength) from~$z_0$ to~$z_1$.
 Let~$v_0\in (W^{1,p}\cap C)(B^p, \, \NN)$
 be a map equal to~$z_0$ on~$\partial B^n$
 and belonging to the homotopy class~$\sigma$.
 We define~$v_1$ as in~\eqref{actionofpi1}.
 The map~$v_1$, too, belongs to the homotopy class~$\sigma$,
 because of Lemma~\ref{lemma:actionofpi1}. Then, we have
 \[
  \begin{split}
   \abs{\sigma}_{z_1}
   \leq \int_{\R} \abs{\nabla v_0} \d x
    + 2 \int_0^1 \abs{\gamma^\prime(t)} \d t
   = \int_{\R} \abs{\nabla v_0} \d x
    + 2 \dist_{\NN}(z_0, \, z_1),
  \end{split}
 \]
 where~$\dist_{\NN}(z_0, \, z_1)$ denotes the geodesic distance
 between~$z_0$ and~$z_1$
%  i.e.~the length of a minimising geodesic~$\gamma$ joining them.
%  (The last equality follows from the fact
%  that~$\abs{\gamma^\prime}$ is constant.)
 By taking the infimum over all admissible~$v_0$
 in~$\sigma$, we obtain
 \begin{equation} \label{almostequivalentnorms}
  \begin{split}
   \abs{\sigma}_{z_1}
   \leq \abs{\sigma}_{z_0} + 2\,\mathrm{diam}\,\NN,
  \end{split}
 \end{equation}
 where~$\mathrm{diam}\,\NN$ is the geodesic
 diameter of~$\NN$. Combining~\eqref{almostequivalentnorms}
 with~\eqref{item:alpha_p}, we deduce
 that the inequality~\eqref{item:equivalentnorms} holds
 with a constant~$C = 1 + \frac{2\,\mathrm{diam}\,\NN}{\alpha_1(\NN)}$.

 \medskip
 \noindent
 \textit{Proof of~(v).}
 Suppose now that~$p\geq 2$.
 We take~$z_0$, $z_1$, $\gamma$ and~$v_0$
 exactly as in the proof of Statement~(iv).
 This time, however, we take a parameter~$R > 1$
 and define~$v_1$ as
 \begin{equation*}
  v_1(x)
  = \begin{cases}
     v_0(x) &\textrm{if } \abs{x} < 1, \\[2mm]
     \gamma\!\left(\dfrac{\log\rho}{\log R}\right)
      &\textrm{if } 1 \leq \abs{x} < R, \\[2mm]
     z_1 &\textrm{if } \abs{x} \geq R.
    \end{cases}
 \end{equation*}
 Again,~$v_1$ belongs to the homotopy class~$\sigma$,
 because of Lemma~\ref{lemma:actionofpi1} and we have
 \[
  \begin{split}
   \abs{\sigma}_{z_1}
   \leq \int_{\R^p} \abs{\nabla v_0}^p \d x
    + \H^{p-1}(\SS^{p-1}) \int_1^R \abs{\gamma^\prime\!\left(\frac{\log r}{\log R}\right)}^p
    \frac{\d r}{r \, (\log R)^p}.
  \end{split}
 \]
 The quantity~$\abs{\gamma^\prime}$ is constant,
 because~$\gamma$ is a geodesic.
 Then, we can evaluate the right-hand side explicitely, as
 \[
  \begin{split}
   \abs{\sigma}_{z_1}
   \leq \int_{\R^p} \abs{\nabla v_0}^p \d x
    + \frac{\H^{p-1}(\SS^{p-1})
    \dist_{\NN}(z_0, \, z_1)^p}{(\log R)^{p-1}} .
  \end{split}
 \]
 By taking the infimum over all admissible~$v_0$
 in~$\sigma$, then letting~$R\to+\infty$, we deduce
 \begin{equation*}
  \begin{split}
   \abs{\sigma}_{z_1} \leq \abs{\sigma}_{z_0} \! .
  \end{split}
 \end{equation*}
 The opposite inequality holds by the very same proof.
\end{proof}

Since spheres and boundaries of cubes are homeomorphic,
homotopy classes are also defined for
maps~$\partial [0, \, h]^{p+1}\to\NN$.
In fact, we have the following property:

\begin{lemma} \label{lemma:cubebounds}
 There exists a constant~$C$, depending only
 on~$p$ and~$\NN$, that satisfies
 \begin{equation} \label{norm_lb}
  \abs{\sigma}_{z_0}
  \leq C \int_{\partial [0, \, h]^{p+1}} \abs{\nabla v}^p \, \d\H^p
 \end{equation}
 for all~$\sigma\in\GN$, all~$z_0\in\NN$, all~$h > 0$, and all
 $v\in W^{1,p}(\partial [0, \, h]^{p+1}, \, \NN)$ in
 the homotopy class~$\sigma$.
\end{lemma}
Note that in Lemma~\ref{lemma:cubebounds},
% the case~$p = 1$ is allowed and
we are \emph{not} requiring~$v$ to take the value~$z_0$ anywhere.
\begin{proof}[Proof of Lemma~\ref{lemma:cubebounds}]
 The right-hand side of~\eqref{norm_lb} is invariant by rescaling,
 so we can assume with no loss of generality that~$h = 1$.
 We can further suppose that~$v$ is continuous, by
 a density argument based on Proposition~\ref{prop:densitysmooth}.
 Let~$\phi\colon\SS^p\to\partial [0, \, 1]^{p+1}$
 be an invertible, Lipschitz map with Lipschitz inverse
 and let~$x_0 := \phi(N_p)$. By using~$v\circ\phi$
 as a test map in~\eqref{norm_p}, we obtain
 \begin{equation} \label{norm_lb_unitcube}
  C \int_{\partial [0, \, 1]^{p+1}} \abs{\nabla v}^p \, \d\H^p
  \geq \int_{\partial\SS^p} \abs{\nabla (v\circ\phi)}^p \, \d\H^p
  \geq \abs{\sigma}_{v(x_0)} \! .
 \end{equation}
 The desired inequality follows from~\eqref{norm_lb_unitcube}
 and Statements~\eqref{item:samenorms}, \eqref{item:equivalentnorms} in Lemma~\ref{lemma:norm_p}.
\end{proof}

\section{The singular set}
\label{sect:Sgrid}

\subsection{The assumption on the target manifold}
\label{sect:assumptions}

Let~$\NN$ be a smooth, compact, connected submanifold
without boundary of some Euclidean space~$\R^m$,
and let~$p\geq 1$ be an integer.
Henceforth, we will always assume that
\begin{enumerate}[label=(H), ref=H]
 \item \label{hp:single-valued}
 for all $\sigma_1\in\GN$, $\sigma_2\in\GN$, the set
 $\sigma_1 + \sigma_2$ contains no more than one element.
\end{enumerate}
Then, Lemma~\ref{lemma:polygroup} implies
that~$\sigma_1 + \sigma_2$
contains exactly one element for all~$\sigma_1$, $\sigma_2\in\GN$,
and moreover, the single-valued operation~$+$ induces
an Abelian group structure on~$\GN$.
We will abuse of notations slightly
and write, e.g.~$\sigma_1 + \sigma_2 = \sigma$
instead of~$\sigma_1 + \sigma_2 = \{\sigma\}$, and so on.
In fact, under assumption~\eqref{hp:single-valued} we have
\[
 \pN = \GN,
\]
where~$\pN$ is the~$p$-th homotopy group of~$\NN$.
(This equality, or rather a group isomorphism~$\pN\to\GN$,
follows from Remark~\ref{rk:conjugacy}; see also
Lemma~\ref{lemma:homotopysum}).
In practice, we will continue working with the group~$\GN$
as defined above, but write~$\pN$ instead of~$\GN$
to keep the notation consistent with the standard one.
Assumption~\eqref{hp:single-valued} allows us to consider more general target manifolds than those considered in~\cite{CO1, CO2}, where~$\NN$ was assumed to be $(p-1)$-connected.
For instance, all simply connected manifolds
satisfy~\eqref{hp:single-valued} for all~$p\geq 1$,
because of Remark~\ref{rk:conjugacy}. Moreover, all Lie groups
(and mroe generally, all topological groups and all H-spaces)
satisfy~\eqref{hp:single-valued} for all~$p\geq 1$; see e.g.~\cite[Chapter~7, \oldS~3, Theorems~5 and~7 on pp.~382--384]{Spanier} or~\cite[Example~4A.3 on p.~422]{Hatcher}.
In case~$p=1$, condition~\eqref{hp:single-valued} is equivalent to the fundamental group of~$\NN$ being Abelian (see Remarks~\ref{rk:conjugacy} and~\ref{rk:multivalued}).
% \begin{remark}
%  Contrary to~\cite{CO1, CO2}, we do \emph{not}
%  assume that~$\NN$ is $(p - 1)$-connected.
%  For instance, the sphere~$\SS^2$ satisfies~\eqref{hp:single-valued}.
%  and~\eqref{hp:last} with~$p = 3$, because~$\pi_3(\SS^2)\simeq\Z$
%  and the norm satisfies
%  \[
%   \abs{d}_{z_0} \leq C \abs{d}^{3/4}
%  \]
%  for all~$d\in\Z$ (see \cite{Riviere}).
% \end{remark}
%
We point out a consequence of the assumption~\eqref{hp:single-valued}.
% that will be important to our analysis.

\begin{prop} \label{prop:compactnorm}
 Let~$\NN\subseteq\R^m$ be a smooth, compact, connected submanifold
 without boundary, and let~$p\geq 1$ be an integer.
 Suppose that the condition~\eqref{hp:single-valued}
 is satisfied. Then, for all~$\Lambda > 0$, there are only finitely
 many homotopy classes~$\sigma\in\pi_p(\NN)$ that contain
 maps~$v\in (W^{1,p}\cap C)(\SS^p, \, \NN)$ with
 \begin{equation} \label{cmptnorm0}
  \int_{\SS^p} \abs{\nabla v}^p \d\H^p \leq \Lambda.
 \end{equation}
\end{prop}

In other words, using the notation introduced in~\eqref{norm_p},
bounded subsets of~$\pi_p(\NN)$ with respect to the norm~$\abs{\,\cdot\,}_{z_0}$ are finite.

\begin{proof}[Proof of Proposition~\ref{prop:compactnorm}]
 This result is a consequence of~\cite[Theorem~4]{DuzaarKuwert}.
 Alternatively, one can apply more recents results in~\cite{VanSchaftingen2020}. More precisely,
 let~$v\in (W^{1,p}\cap C)(\SS^p, \, \NN)$ satisfy~\eqref{cmptnorm0}.
 Since~$v$ takes values in the compact manifold~$\NN$,
 the Gagliardo-Nirenberg interpolation inequality~\cite{BrezisMironescu2018} implies
 \begin{equation} \label{cmptnorm1}
  [v]_{W^{s,p/s}(\SS^p)}^p :=
  \iint_{\SS^p\times\SS^p}
   \frac{\abs{v(x) - v(y)}^p}{\abs{x - y}^{2p}} \, \d x \, \d y
  \leq C\norm{\nabla v}^s_{L^p(\SS^p)}
  \norm{v}_{L^\infty(\SS^p)}^{1-s}
%   + C\norm{v}_{L^\infty(\SS^p)}
  \leq C.
 \end{equation}
 for all~$s\in (0, \, 1)$ and some constant~$C$ that depends only on~$\Lambda$, $\NN$, $s$ and~$p$.
 (Here the distance~$\abs{v(x) - v(y)}$ is measured
 with respect to the ambient Euclidean space~$\R^m\supseteq\NN$.)
 Given~$\eps > 0$, let~$U_\eps(v)$ be the set of all pairs~$(x, \, y)\in\SS^p\times\SS^p$ such that~$\dist_{\NN}(v(x), \, v(y)) > \eps$, where~$\dist_{\NN}$ is the geodesic distance on~$\NN$.
 Since the geodesic distance and the Euclidean one are equivalent on a compact manifold --- that is, there exists a constant~$C > 0$ such that $\abs{z - w} \leq \dist_{\NN}(z, \, w) \leq C \abs{z - w}$ for all~$z\in\NN$, $w\in\NN$ ---,
 the estimate~\eqref{cmptnorm1} implies
 \[
  \iint_{U_\eps(v)}
   \frac{\d x \, \d y}{\abs{x - y}^{2p}}
  \leq \frac{C}{\eps^p}.
 \]
 Then, Theorem~1.4 in~\cite{VanSchaftingen2020}
 states that there is only a finite number
 of possibilities for the homotopy class of~$v$.
 The topological assumption in~\cite{VanSchaftingen2020}
 is even weaker than our condition~\eqref{hp:single-valued};
 indeed, the result in~\cite{VanSchaftingen2020} assumes
 that all the orbits for the action of~$\pi_1(\NN)$
 on~$\pi_p(\NN)$ are finite, while condition~\eqref{hp:single-valued} implies that all such orbits are singletons
 (see Remark~\ref{rk:conjugacy}).
 Therefore, the proposition follows.
\end{proof}

Under the assumption~\eqref{hp:single-valued}, Lemma~\ref{lemma:norm_p} implies that~$\abs{\,\cdot\,}_{z_0}$ is a norm on~$\pN$, for any~$z_0\in\NN$.
Moreover, different choices of the basepoint lead to the same norm, if~$p\geq 2$, or to equivalent norms, if~$p = 1$.
In what follows, we choose the basepoint~$z_0\in\NN$ once and for all and write~$\abs{\,\cdot\,}$ instead of~$\abs{\,\cdot\,}_{z_0}$.
We consider chains with coefficients in the normed Abelian group~$(\pN, \, \abs{\,\cdot\,})$.

\subsection{Approximating the singular set on grids}

Let~$\Omega\subseteq\R^{p+1}$ be a bounded, Lipschitz domain.
We recall the definition of the topological singular
set~$\Snice(u)$ of a map~$u\in R^{1,p}(\Omega, \, \NN)$
with nice singularities.
Let~$\Sigma = \Sigma(u)\subseteq\Omega$
be a finite set such that~$u$ is locally Lipschitz
in~$\overline{\Omega}\setminus\Sigma$.
Let~$\rho > 0$ be a small enough radius,
such that the closed balls~$B_\rho(x)$, with~$x\in\Sigma$,
are pairwise disjoint and contained in~$\Omega$.
Let~$\sigma(u, \, x)\in\GN = \pN$ denote the
(free) homotopy class of~$u$ restricted
to~$\partial B_\rho(x)$. By definition of free homotopy,
$\sigma(u, \, x)$ is independent of the choice of~$\rho$.
We define
\begin{equation} \label{Stop_nice}
 \Snice(u) := \sum_{x\in\Sigma} \sigma(u, \, x)
   \llbracket x\rrbracket \in \P_0(\Omega; \, \pN).
\end{equation}
Recalling the definition of the operation on~$\GN = \pN$
and the definition of the intersection index~$\ind$ in~\eqref{intersection},
we obtain the following property: if~$V$ is a Lipschitz domain,
homeomorphic to a disk, such that~$\overline{V}\subseteq\Omega$
and~$\spt\Snice(u)\cap \partial V =\varnothing$, then
\begin{equation} \label{SI}
 \ind(\Snice(u), \, V) = (\textrm{homotopy class
 of the restriction } u_{|\partial V}).
\end{equation}
Our goal is to show that~$\Snice$ extends to
arbitrary maps in~$W^{1,p}(\Omega, \, \NN)$.
As a first step, we define the singular set of~$u$
restricted to a \emph{grid}.
We will denote by~$C$ several different constants that only
depend on~$\Omega$, $\NN$, $p$ and may change from line to line,
and write~$A\lesssim B$ as a synonym of~$A \leq C B$.

% First of all, we recall that there exists a
% (nonlinear, but continuous) extension operator
% $W^{1,p}(\Omega, \, \R^n) \to W^{1,p}(\Omega^\prime, \, \NN)$,
% where~$\Omega^\prime\supseteq\overline{\Omega}$ is a larger
% (but still bounded) domain, such that
% \begin{equation} \label{extension}
%  \int_{\Omega^\prime} \abs{\nabla u}^p \, \d x
%  \lesssim \int_{\Omega} \abs{\nabla u}^p \, \d x
% \end{equation}
% for all~$u\in W^{1,p}(\Omega, \, \NN)$.
% Here and in what follows, we denote with the same symbol
% the map~$u\colon\Omega\to\NN$ and its extension
% to~$\Omega^\prime$. Such an operator can be defined
% for instance, by composition with a reflection across
% the boundary~$\partial\Omega$, since~$\Omega$ is bounded
% and has Lipschitz boundary (see e.g.~\cite[Proposition~8.1]{ABO2}).
% Next, let~$h_* := \frac{1}{2\sqrt{n}} \dist(\Omega, \, \partial\Omega^\prime) > 0$.
As in Section~\ref{sect:grids}, we work with grids~$\GG(h, \, y)$ of size~$h > 0$.
% so that any cell of the grids that intersects~$\Omega$ is contained in~$\Omega^\prime$.
Let~$u\in W^{1,p}(\Omega, \, \NN)$, $h > 0$ and~$y\in Q^{p+1}$ be given.
% By abuse of notation, we identify~$u$
% with its image in~$W^{1,p}(\Omega^\prime, \, \NN)$
% under the extension operator defined above.
For almost every choice of~$y$, the restriction of~$u$
to the~$p$-skeleton~$R^p(h, \, y)\cap\Omega$ is well-defined
and belongs to~$W^{1,p}(R^p(h, \, y), \, \NN)$.
Therefore, given any $(p+1)$-cell~$K\in\GG(h, \, y)$ such that~$K\subseteq\Omega$,
$u_{|\partial K}$ has a well-defined (free)
homotopy class~$\sigma(u, \, \partial K)\in\GN = \pN$,
as defined in Section~\ref{sect:homotopy_W1p}.
Let~$c_K$ be the centre of~$K$. We define
\begin{equation} \label{Sgrid}
 \Sgrid(u, \, h, \, y) 
 := \sum_{K\in\GG^{p+1}(h, \, y)\colon K\subseteq\Omega}
  \sigma(u, \, \partial K)
  \llbracket c_K \rrbracket.
\end{equation}
In case~$u\in R^{1,p}(\Omega, \, \NN)$,
the homotopy class~$\sigma(u, \, \partial K)$
is the sum of the classes~$\sigma(u, \, x)$
among all the singular points~$x$ of~$u$ contained in~$K$.
This claim follows from the very definition of the
operation in~$\GN$. Therefore, recalling the
definition~\eqref{deformation} of the 
approximation~$P(S, \, h, \, y)$ for a polyhedral 
chain~$S$, we immediately obtain
\begin{equation} \label{Sgridnice}
 \Sgrid(u, \, h, \, y) = P(\Snice(u), \, h, \, y)
\end{equation}
for all~$u\in R^{1,p}(\Omega, \, \NN)$, all~$h > 0$
and almost all~$y\in Q^{p+1}$. In light of~\ref{Pflatconv}
in Proposition~\ref{prop:deformation} and~\eqref{Sgridnice},
we deduce
\begin{equation} \label{Sgridnice-conv}
 \begin{split}
  \int_{Q^{p+1}} \F_\Omega(\Snice(u) - \Sgrid(u, \, h, \, y)) \, \d y
%   \lesssim h \, \M(\Snice(u))
  \to 0
  \qquad \textrm{as } h\to 0.
 \end{split}
\end{equation}
% In particular, when~$u$ has finitely many singularities,
% $\Sgrid(u, \, h, \, y)$ is indeed an approximation of~$\Snice(u)$.
%
% Our next task is to find bounds on~$\Sgrid(u, \, h, \, y)$.
% The following lemma does not provide a uniform bound
% in terms of~$h$, but will still be useful later on.
%
% \begin{lemma} \label{lemma:massS}
%  For any~$u\in W^{1,p}(\Omega, \, \NN)$, we have
%  \begin{equation*} %\label{massS}
%   \begin{split}
%    \int_{Q^{p+1}} \M(\Sgrid(u, \, h, \, y)) \, \d y
%    \lesssim \frac{1}{h} \int_{\Omega} \abs{\nabla u}^p \d x.
%   \end{split}
%  \end{equation*}
% \end{lemma}
% \begin{proof}
%  The mass of~$\Sgrid(u, \, h, \, y)$ can
%  be bounded from above using the estimate~\eqref{norm_lb}:
%  \begin{equation*} %\label{mass_Sgrid}
%   \begin{split}
%    \M(\Sgrid(u, \, h, \, y))
%    = \sum_{K} \abs{\sigma(u, \, \partial K)}_p
%    \lesssim \int_{R^p(h, \, y, \, \Omega)} \abs{\nabla u}^p \, \d\H^p.
%   \end{split}
%  \end{equation*}
%  Combining this inequality with Lemma~\ref{lemma:Fubini}
%  and~\eqref{extension}, the lemma follows.
% \end{proof}
Therefore, bounds on~$\Sgrid(u, \, h, \, y)$
that are uniform in~$h$ lead to bounds on~$\Snice(u)$.

\begin{lemma} \label{lemma:Sbdd}
 For all~$u\in\Hw^{1,p}(\Omega, \, \NN)$ and all~$h > 0$, we have
 \begin{equation} \label{Sgridbdd}
  \int_{Q^{p+1}} \F_\Omega(\Sgrid(u, \, h, \, y)) \, \d y
  \lesssim \overline{D}_p(u).
 \end{equation}
 In addition, if~$u\in R^{1,p}(\Omega, \, \NN)$, then
 \begin{equation} \label{Snicebdd}
  \F_\Omega(\Snice(u)) \lesssim \overline{D}_p(u).
 \end{equation}
\end{lemma}
The proof of this lemma is based on ideas in~\cite[Lemma~1]{Luckhaus-PartialReg}.

\begin{figure}[htbp]
 \centering

 \hspace{25mm}
 \begin{subfigure}[b]{0.7\textwidth}
  \centering
  \includegraphics[width=\textwidth]{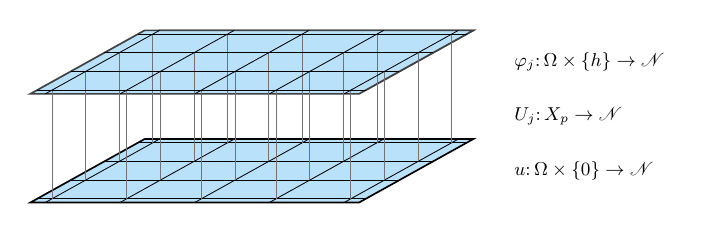}
%         \caption{}
  \label{fig:Sbdd1}
 \end{subfigure}

 \vspace{3mm}
 \hspace{25mm}
 \begin{subfigure}[b]{0.7\textwidth}
  \centering
  \includegraphics[width=\textwidth]{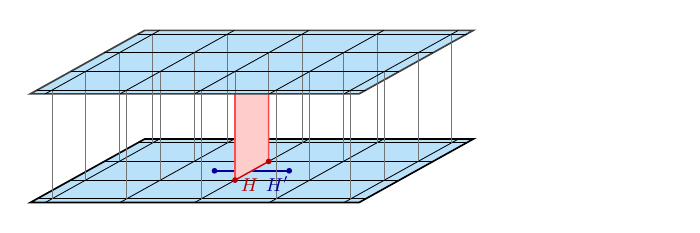}
%         \caption{}
  \label{fig:Sbdd2}
 \end{subfigure}

 \caption{Proof of Lemma~\ref{lemma:Sbdd}.
  Top: the grids on~$\Omega$ and on~$\Omega\times[0, \, h]$ (see Step~2);
  bottom: the geometry of~$T(U_j, \, h, \, y)$ (see Step~3).
 }
 \label{fig:Sbdd}
\end{figure}

\begin{proof}[Proof of Lemma~\ref{lemma:Sbdd}]
 We focus on the proof of~\eqref{Sgridbdd}; once~\eqref{Sgridbdd} is established, the estimate~\eqref{Snicebdd} will follow, thanks to~\eqref{Sgridnice-conv}.
 Let~$u\in\Hw^{1,p}(\Omega, \, \NN)$.
 A diagonal argument, applied to the definition of the relaxed energy~\eqref{relaxed}, shows that there exists a sequence~$(\varphi_j)_{j\in\N}$
 in~$C^\infty(\overline{\Omega}, \, \NN)$
 that converges to~$u$ strongly in~$L^{p}(\Omega)$ and satisfies
 \[
  \lim_{j\to+\infty} \int_\Omega\abs{\nabla\varphi_j}^p\d x
  = \overline{D}_p(u) < +\infty.
 \]
 We split the rest of the proof into several steps.

 \medskip
 \setcounter{step}{0}
 \begin{step}
  Given~$h > 0$, we consider grids~$\GG(h, \, y)$.
  Fatou lemma and Lemma~\ref{lemma:Fubini} imply
  \begin{align*}
   h^{p - k + 1}\int_{Q^{p+1}} \left(\liminf_{j\to+\infty}
    \int_{R_{k}(h, \, y)\cap\Omega} \abs{\nabla\varphi_j}^p \d\H^{k}\right) \d y
   &\leq \overline{D}_p(u), \\
   h^{p - k + 1}\int_{Q^{p+1}} \left(\liminf_{j\to+\infty}
    \int_{R_{k}(h, \, y)\cap\Omega} \abs{u - \varphi_j}^p \d\H^{k}\right) \d y
   &\leq \liminf_{j\to+\infty} \int_{\Omega} \abs{u - \varphi_j}^p \d x = 0
  \end{align*}
  for all~$k\in\{0, \, 1, \, \ldots, p\}$.
  Therefore, upon extracting a (non-relabelled) subsequence~$j\to+\infty$, we can fix once and for all a value~$y\in Q^{p+1}$ that satisfies
  \begin{gather}
   h^{p - k + 1}
    \int_{R_{k}(h, \, y)\cap\Omega} \abs{\nabla u}^p \d\H^{k}
    \lesssim \int_{\Omega} \abs{\nabla u}^p \d x, \label{Sbddavu} \\
   h^{p - k + 1} \int_{R_k(h, \, y)\cap\Omega}
    \abs{\nabla\varphi_j}^p \d\H^{k}
    \lesssim \overline{D}_p(u)
   \qquad \textrm{for all } j\in\N, \label{Sbddavphij} \\
    h^{p - k + 1} \int_{R_k(h, \, y)\cap\Omega}
    \abs{u - \varphi_j}^p \d\H^{k}
    \to 0 \qquad \textrm{as } j\to+\infty \label{SbddavLp}
  \end{gather}
  for all~$k\in\{0, \, 1, \, \ldots,  p\}$.
  (The inequality~\eqref{Sbddavu}, too, follows by Lemma~\ref{lemma:Fubini}.)
  We write~$\widetilde{R}_k$ for the union of the~$k$-cells
  of~$\GG(h, \, y)$ that are contained in~$\Omega$.
%   Since~$\varphi_j\to u$ weakly in~$W^{1,p}(\Omega)$,
%   by applying Fubini theorem (and possibly extracting a non-relabelled subsequence) we can make sure that the traces of~$\varphi_j$ on the~$(p-1)$-skeleton~$\widetilde{R}_{p-1} := \widetilde{R}_{p-1}(h, \, y)$ converge weakly in~$W^{1,p}(\widetilde{R}_{p-1})$ to that of~$u$,
%   except for a negligible set of~$y\in Q^{p+1}$.
  By applying the Gagliardo-Nirenberg
  interpolation inequality on each~$(p-1)$-cell, we obtain
  \begin{equation*}
   \begin{split}
    \norm{u - \varphi_j}_{L^\infty(\widetilde{R}_{p-1})}
    \lesssim \norm{\nabla u - \nabla\varphi_j}_{L^p(\widetilde{R}_{p-1})}^{1 - 1/p}
     \norm{u - \varphi_j}_{L^p(\widetilde{R}_{p-1})}^{1/p}
     + h^{1/p - 1} \norm{u - \varphi_j}_{L^p(\widetilde{R}_{p-1})} \! .
   \end{split}
  \end{equation*}
%   where~$\widetilde{R}_{p-1} := \widetilde{R}_{p-1}(h, \, y)$ is
%   the~$(p-1)$-skeleton of~$\GG(h, \, y)$.
  Then, thanks to~\eqref{Sbddavu}, \eqref{Sbddavphij}, \eqref{SbddavLp},
  for all~$j$ large enough we have
  \begin{equation} \label{Stopbdd-p-1sk}
   \norm{u - \varphi_j}_{L^\infty(\widetilde{R}_{p-1})} \leq \delta_*.
  \end{equation}
  Here~$\delta_* > 0$ is small enough that
  the nearest-point projection onto~$\NN$
  (see~\eqref{projN}) is well-defined
  in a~$\delta_*$-neighbourhood of~$\NN$.
 \end{step}

 \medskip
 \begin{step}
  Next, we extend~$\GG(h, \, y)$
  to a higher-dimensional grid, of the form
  \begin{equation} \label{Stop-E}
   \EE := \left\{K\times[0, \, h]
   \colon K\in\GG(h, \, y), \ K\subseteq\Omega \right\}
  \end{equation}
  (see Figure~\ref{fig:Sbdd}).
  Let~$\EE_{p}$ be the collection of~$p$-dimensional
  cells of~$\EE$, and
  let $X_p := \cup_{K\in\EE_p} K$ be the~$p$-skeleton of~$\EE$.
  We define a map~$U_j\colon X_p\to\NN$ as follows.
  On the~$p$-cells of~$X_p$ contained
  in~$\Omega\times\{0\}$, we define~$U_j := u$.
  On $X_p\cap\left(\Omega\times\{h\}\right)$,
  we define~$U_j := \varphi_j$.
%   where~$j = j(h, \, y)$ has been chosen above.
  Finally, on the cells of the form~$K\times [0, \, h]$
  for~$K\in\GG_{p-1}$, we define
  \[
   U_j(x, \, t) := \pi_{\NN}\!\left(\left(1 - \frac{t}{h}\right) u(x)
    + \frac{t}{h}\varphi_j(x)\right)\!,
   \qquad \textrm{for } (x, \, t)\in K\times [0, \, h].
  \]
  This map is well-defined, thanks to~\eqref{Stopbdd-p-1sk}.
  Since the projection~$\pi_{\NN}$ is Lipschitz-continuous
  on the~$\delta_*$-neighbourhood of~$\NN$, we have
  \begin{equation*}
   \begin{split}
    \int_{X_p} \abs{\nabla U_j}^p \d\H^p
    &\leq \int_{\widetilde{R}_p} \left(\abs{\nabla u}^p
     + \abs{\nabla\varphi_j}^p\right) \d\H^p \\
    &\hspace{1cm} + Ch \int_{\widetilde{R}_{p-1}} \left(\abs{\nabla u}^p
     + \abs{\nabla\varphi_j}^p
     + \frac{\abs{u - \varphi_j}^p}{h^p}\right) \d\H^{p-1}
   \end{split}
  \end{equation*}
  and hence, thanks to~\eqref{Sbddavu}, \eqref{Sbddavphij}, \eqref{SbddavLp},
  \begin{equation} \label{Stopbdd-Uj}
   \begin{split}
    h\int_{X_p} \abs{\nabla U_j}^p \d\H^p
    &\leq \int_{\Omega} \abs{\nabla u}^p \d x
    + \overline{D}_p(u) + \mathrm{o}_{j\to+\infty}(1).
   \end{split}
  \end{equation}
  Note the extra factor of~$h$ in front of the right-hand side of~\eqref{Stopbdd-Uj}.
 \end{step}

 \medskip
 \begin{step}
  Now, we use the map~$U_j$ to a polyhedral $1$-chain whose boundary is~$\Sgrid(u, \, h, \, y)$.
  For each $p$-cell~$H\in\GG_p(h, \, y)$ contained in~$\Omega$,
  let~$H^\prime$ be its dual~$1$-cell, as in Figure~\ref{fig:Sbdd}.
%   and let~$a_H$, $b_H$ be the endpoints of~$H^\prime$.
  We consider orientations of~$H$ and~$H^\prime$ in such a way that,
  for any positive basis~$(\tau_1, \, \ldots, \, \tau_{p})$
  of the $p$-plane containing~$H$ and~$\tau_{H^\prime}$ the unit tangent vector to~$H^\prime$ consistent with its orientation,
  $(\tau_1, \, \ldots, \, \tau_{p}, \tau_H)$
  is a positive basis of~$\R^{p+1}$.
  Let~$\sigma(U_j, \, H)$ be the (free) homotopy class
  of~$U_j$ restricted to~$\partial(H\times [0, \, h])$,
  where the latter is given the orientation induced
  by~$H\times [0, \, h]$. We define
  \begin{equation*}
   T(U_j, \, h, \, y)
   := \sum_{H\in \GG_p(h, \, y)\colon H\subseteq\Omega}
   \sigma(U_j, \, H) \lb H^\prime_0\rb
  \end{equation*}
  and claim that
  \begin{equation} \label{Stopbdd-dipole}
   \Sgrid(u, \, h, \, y)
%    = \sum_{H\in \GG_p(h, \, y)\colon H\subseteq\Omega}
%    \left(\sigma(U_j, \, H) \lb b_H\rb - \sigma(U_j, \, H) \lb a_H\rb\right) \! .
   = \partial T(U_j, \, h, \, y) \qquad \textrm{in } \Omega.
  \end{equation}
  First, we observe that the right-hand side of~\eqref{Stopbdd-dipole} does not depend on the choice of an orientation for each~$H$.
  Indeed, changing the orientation of~$H$
  switches the orientation of~$H^\prime$, but also
  changes~$\sigma(U_j, \, H)$ into~$-\sigma(U_j, \, H)$,
  because~$\sigma(U_j, \, H)$ depends on the choice
  of an orientation-preserving
  homeomorphism~$\partial(H\times [0, \, h])\to\SS^p$
  (see Remark~\ref{rk:orientation}).
  As a consequence, the product $\sigma(U_j, \, H)\lb H^\prime\rb$ remains unaffected.
  Next, let~$K$ be a ($(p+1)$-dimensional) cube
  in~$\GG(h, \, y)$ contained in~$\Omega$, and let~$c_K$ be its centre.
  The~$1$-cells~$H^\prime\in\GG_1(h, \, y)$
  that are incident to~$c_K$ are in one-to-one
  correspondence with the boundary faces~$H$ of~$K$.
  The boundary~$(p+1)$-faces of~$K\times [0, \, h]$
  are exactly those of the form~$H\times [0, \, h]$,
  $K\times\{0\}$, and~$K\times\{h\}$.
  The homotopy class of~$U_j$ restricted
  to~$\partial(K\times\{0\})$ is the homotopy class
  of~$u$ on~$\partial K$, $\sigma(u, \, \partial K)$.
  The homotopy class of~$U_j$ restricted
  to~$\partial(K\times\{h\})$ is trivial, because~$\varphi_j$
  is smooth. Therefore, Lemma~\ref{lemma:homotopysum-cubes}
  implies
  \begin{equation} \label{Stopbdd-sigma}
   \sigma(u, \, \partial K)
   = \sum_{H^\prime \ni c_K} \sigma(U_j, \, H),
  \end{equation}
  where the sum is taken over all~$1$-cells~$H^\prime$
  of the dual grid that are incident to~$c_K$.
  Since~\eqref{Stopbdd-sigma} holds for any~$K$,
  the claim~\eqref{Stopbdd-dipole} follows.
 \end{step}

 \begin{step}
  Using the definition~\eqref{flat} of~$\F_\Omega$, \eqref{Stopbdd-dipole},
  and Lemma~\ref{lemma:cubebounds}, we obtain
  \[
   \begin{split}
    \F_{\Omega}(\Sgrid(u, \, h, \, y))
    \leq \M(T(U_j, \, h, \, y))
    = h \sum_{H\in\GG_{p}\colon H\subseteq\Omega} \abs{\sigma(U_j, \, H)}
    \leq C h \int_{X_p} \abs{\nabla U_j}^p \d\H^p.
   \end{split}
  \]
  Then, from~\eqref{Stopbdd-Uj}, we deduce
  \begin{equation} \label{daje}
   \begin{split}
    \F_{\Omega}(\Sgrid(u, \, h, \, y))
    &\lesssim \int_{\Omega} \abs{\nabla u}^p \d x
    + \overline{D}_p(u) + \mathrm{o}_{j\to+\infty}(1).
   \end{split}
  \end{equation}
  Since
  \[
   \int_{\Omega} \abs{\nabla u}^p \d x \leq \overline{D}_p(u),
  \]
  by lower semicontinuity of~$\overline{D}_p(u)$,
  taking the limit as~$j\to+\infty$ at the right-hand side of~\eqref{daje},
  proves~\eqref{Sgridbdd}.
  \qedhere
 \end{step}
\end{proof}

\section{Approximation by maps with nice singularities
and bounded relaxed energy}
\label{sect:RDbar}

The goal of this section is to prove Proposition~\ref{prop:RDbar}.
This will be achieved by combining two propositions:

\begin{prop} \label{prop:RS}
 There exists~$h_* > 0$ such that the following statement holds:
 For all~$u\in W^{1,p}(\Omega, \, \NN)$, there exists a family of maps~$(u_h)_{0 < h\leq h_*}\subseteq R^{1,p}(\Omega, \, \NN)$ that converges strongly in~$W^{1,p}(\Omega)$ to~$u$ as~$h\to 0$ and satisfies
 \[
  \F_\Omega(\Snice(u_h))
  \leq C \int_{\Omega} \abs{\nabla u}^p \, \d x
  + C\int_{Q^{p+1}} \F_\Omega(\Sgrid(u, \, h, \, y)) \, \d y,
 \]
 for all~$h\in (0, \, h_*]$ and some constant~$C$ that depends
 only on~$\Omega$, $\NN$, $p$.
\end{prop}

\begin{prop} \label{prop:dipole}
 For all~$u\in R^{1,p}(\Omega, \, \NN)$, there holds
 \[
  \overline{D}_p(u)
  \leq \int_\Omega \abs{\nabla u}^p \d x + C \, \F_\Omega(\Snice(u)),
 \]
 for some constant~$C$ that depends only on~$\Omega$, $\NN$, $p$.
\end{prop}

Proposition~\ref{prop:RDbar} is an immediate consequence
of Propositions~\ref{prop:RS} and~\ref{prop:dipole},
which will be proved in Subsections~\ref{sect:RS}
and~\ref{sect:dipole} respectively, combined with the estimate~\eqref{Sgridbdd} from the proof of Lemma~\ref{lemma:Sbdd}.

\subsection{Proof of Proposition~\ref{prop:RS}}
\label{sect:RS}

Proposition~\ref{prop:RS} is a variant of~\cite[Theorem~2]{Bethuel-Density}
and can be proved by reasoning along the lines
of~\cite[Theorem~3]{BousquetPonceVanSchaftingen-II}.
First of all, we recall that there exists a
nonlinear, but continuous extension operator
$W^{1,p}(\Omega, \, \NN) \to W^{1,p}(\Omega^\prime, \, \NN)$,
where~$\Omega^\prime\supseteq\overline{\Omega}$ is a larger
(but still bounded) domain, such that
\begin{equation} \label{extension}
 \int_{\Omega^\prime} \abs{\nabla u}^p \, \d x
 \lesssim \int_{\Omega} \abs{\nabla u}^p \, \d x
\end{equation}
for all~$u\in W^{1,p}(\Omega, \, \NN)$.
Here and in what follows, we denote with the same symbol
the map~$u\colon\Omega\to\NN$ and its extension
to~$\Omega^\prime$. Such an operator can be defined,
for instance, by composition with a reflection across
the boundary~$\partial\Omega$, since~$\Omega$ is bounded
and has Lipschitz boundary (see e.g.~\cite[Proposition~8.1]{ABO2}).
Next, let~$h_* := \frac{1}{2\sqrt{n}} \dist(\Omega, \, \partial\Omega^\prime) > 0$.
For~$0 < h \leq h_*$ and~$y\in Q^{p+1}$, we will denote by~$\GG(h, \, y, \, \Omega)$ the set of ($n$-dimensional)
cubes~$k\in\GG(h, \, y, \, \Omega)$ such
that~$K\cap\Omega\neq\varnothing$.
Our choice of~$h_*$ guarantees that every cube in~$\GG(h, \, y, \, \Omega)$ is (compactly) contained in~$\Omega^\prime$.
We denote by~$\GG_j(h, \, y, \, \Omega)$
the collection of~$j$-cells of~$\GG(h, \, y, \, \Omega)$
and by~$R_j(h, \, y, \, \Omega)$ its~$j$-skeleton, that is~$R_j(h, \, y, \, \Omega) := \bigcup_{K\in\GG_j(h, \, y, \, \Omega)} K$.

The first tool in the proof of Proposition~\ref{prop:RS} is the following result,
which is a variant of~\cite[Proposition~2.1]{BousquetPonceVanSchaftingen-II}.
Let~$u\in W^{1,p}(\Omega, \, \NN)$ and~$h\in (0, \, h_*]$ be given.
By Lemma~\ref{lemma:Fubini}, an averaging argument, and the estimate~\eqref{extension},
we can select~$y = y(u, \, h)\in Q^{p+1}$ in such a way that
% for each~$j\in\{0, \, 1, \, \ldots, \, p\}$,
the trace of~$u$ on~$R_j(h, \, y)\cap\Omega^\prime$ belongs
to~$W^{1,p}(R_j(h, \, y)\cap\Omega^\prime)$ and satisfies
\begin{equation} \label{hp:goodgrid}
 \int_{R^j(h, \, y)\cap\Omega^\prime} \abs{\nabla u}^p \d\H^j
  \leq C h^{j - n} \int_\Omega \abs{\nabla u}^p \d x
  \qquad \textrm{for all } j\in\{0, \, 1, \, \ldots, p\},
\end{equation}
where the constant~$C$ depends only on~$\Omega$ and~$p$.
Let~$\mathscr{C}\subseteq\GG(h, \, y, \, \Omega)$
be a given collection of cubes of the grid~$\GG(h, \, y)$,
all of them having non-empty intersection with~$\Omega$.
We denote by~$R_{\mathscr{C}}$ and~$R_{\partial\mathscr{C}}$,
respectively, the $(p+1)$-dimensional skeleton of~$\mathscr{C}$
and the $p$-dimensional one, that is,
\begin{equation} \label{RC}
 R_{\mathscr{C}} := \bigcup_{K\in\mathscr{C}} K, \qquad
 R_{\partial\mathscr{C}} := \bigcup_{K\in\mathscr{C}} \partial K.
\end{equation}
For all~$t > 0$ and all~$X\in\R^{p+1}$, we define
\begin{equation} \label{U_t}
 U_t(X) := \left\{x\in\Omega^\prime \colon
 \dist(x, \, X) < t\right\} \! .
\end{equation}
In particular, we have~$U_t(R_{\partial\mathscr{C}})\subseteq\Omega^\prime$
if~$h\leq h_0$ and~$t < h_0$.

% In the following result, we consider maps of the form~$u\circ\Phi$,
% where~$u\in W^{1,p}(\Omega, \, \NN)$ and~$\Phi\colon\R^{p+1}\to\R^{p+1}$
% is smooth. Such maps need not be well-defined pointwise,
% for~$u$ is defined pointwise almost everywhere only.
% Instead, $u\circ\Phi$ is defined by approximation:
% given a sequence of smooth maps~$(u_j)_{j\in\N}\subseteq C^\infty(\overline{\Omega^\prime}, \, \R^m)$ that converge
% to~$u$ strongly in~$W^{1,p}(\Omega^\prime)$, under suitabnle assumptions
% on~$\Phi$ it is possible to prove that~$u_j\circ\Phi\to v$
% strongly in~$W^{1,p}(\Omega^\prime)$ as~$j\to+\infty$,
% and the limit is independent of the choice of the approximating sequence
% (see~\cite[Section~2]{BousquetPonceVanSchaftingen-II}).
% We define~$u\circ\Phi := v$.

\begin{prop}%[{\cite[Proposition~2.1 and Addendum~2]
%{BousquetPonceVanSchaftingen-II}}]
\label{prop:opening}
 Let $0 < h \leq h_0$, $y\in\R^{p+1}$, $0 < \theta < 1/2$
 and~$\mathscr{C}\subseteq\GG(h, \, y, \, \Omega)$.
 Let~$u\in W^{1, p}(\Omega, \, \NN)$ be a map
 (which we extend to an element of~$W^{1,p}(\Omega^\prime, \, \NN)$
 satisfying~\eqref{extension}). If the condition~\eqref{hp:goodgrid}
 is satisfied, then there exists a map
 $u^{\mathrm{op}}_h\in W^{1, p}(\Omega^\prime, \, \NN)$
 that satisfies the following properties:
 \begin{enumerate}[label=(\roman*)]
  \item for every cube $K\in\mathscr{C}$ and every boundary
  $j$-face~$H\subseteq\partial K$, $u^{\mathrm{op}}_h$ is constant
  on the $(p - j + 1)$-dimensional cubes of edge length~$2\theta h$
  that are orthogonal to~$H$ and are centred at some point of~$H$;

  \item $u^{\mathrm{op}}_h = u$
  in~$\Omega^\prime\setminus U_{2\theta h}(R_{\partial\mathscr{C}})$;

  \item for every cube $K\in\mathscr{C}$ and every boundary
  $p$-face~$H\subseteq\partial K$, we have
  \begin{equation*}
  \norm{\nabla u^{\mathrm{op}}_h}_{L^p(U_{2\theta h}(H))}
  \leq C \norm{\nabla u}_{L^p(U_{2\theta h}(H))}
  \end{equation*}
  for some constant $C > 0$ depending only on~$p$ and~$\theta$;

  \item we have
  \[
   \begin{split}
    \sup_{a+Q_s^{p+1}\subseteq U_{\theta h}(R_{\partial\mathscr{C}})}
    \frac{1}{s^{2p +2}} \int_{a+Q_s^{p+1}} \int_{a+Q_s^{p+1}}
    \abs{u^{\mathrm{op}}_h(x) - u^{\mathrm{op}}_h(y)} \d x \, \d y \to 0
   \end{split}
  \]
  as~$s\to 0$;

  \item we have
  \[
   \F_{\Omega^\prime}\!\left(\Sgrid(u, \, h, \, y)
    - \Sgrid(u^{\mathrm{op}}_h, \, h, \, y)\right)
   \leq C \, \int_{\Omega} \abs{\nabla u}^p \d x
  \]
  for some constant $C > 0$ depending only on~$p$ and~$\theta$.
\end{enumerate}
\end{prop}

The proof of Proposition~\ref{prop:opening} is based on the
following result.

\begin{prop}[{\cite[Proposition~2.2]{BousquetPonceVanSchaftingen-II}}]
\label{prop:openingcube}
 Let $j \in \{0, \, 1, \, \ldots, p\}$, $h > 0$,
 $0 < \theta_1 < \theta_2$ and $A \subseteq \R^j$ be an open set.
 For every $u\in W^{1, p}(A \times Q_{\theta_2 h}^{p - j + 1}, \, \R^m)$,
 there exists a smooth map $\zeta\colon\R^{p - j + 1}\to\R^{p - j + 1}$
 that satisfies the following properties:
 \begin{enumerate}[label=(\roman*)]
  \item $\zeta$ is constant in $Q_{\theta_1h}^{p - j + 1}$;
  \item $\zeta(x) = x$ for all~$x\notin Q_{\theta_2 h}^{p - j + 1}$
  and $\zeta(x)\in Q_{\theta_2h}^{p - j + 1}$ for all~$x\in Q_{\theta_2 h}^{p - j + 1}$;
  \item letting $\phi\colon\R^{p+1}\to\R^{p+1}$ be
  defined as~$\phi(x) = (x', \, \zeta(x''))$,
  $x = (x', x'') \in \R^j\times\R^{p - j + 1}$,
  the map $u\circ\phi$ is well-defined, belongs
  to~$W^{1, p}(A \times Q_{\theta_2h}^{p - j + 1}, \, \R^m)$,
  and satisfies
  \begin{equation*}
   \norm{\nabla(u\circ\phi)}_{L^p({A \times Q^{p-j+1}_{\theta_2 h}})}
   \leq C \norm{\nabla u}_{L^p({A \times Q^{p-j+1}_{\theta_2 h}})},
  \end{equation*}
  for some constant $C > 0$ depending only on~$p$, $\theta_1$ and $\theta_2$.
 \end{enumerate}
\end{prop}

\begin{proof}[Proof of Proposition~\ref{prop:opening}]
 The result follows along the lines of~\cite[Proposition~2.2]{BousquetPonceVanSchaftingen-II}. First, we choose a
 finite sequence $(\theta_i)_{0 \leq i \leq p}$ such that
 \[
  \theta = \theta_p < \ldots < \theta_i < \ldots < \theta_0 < \theta_{-1} = 2\theta.
 \]
%  Let~$u\in W^{1,p}(\Omega, \, \R^m)$.
 For~$i\in\{1, \, \ldots, \, p\}$, let~$\mathscr{C}^i$ be
 the set of all~$i$-dimensional boundary faces of
 the cubes in~$\mathscr{C}$.
 Let~$R_{\mathscr{C}^i} := \bigcup_{H\in\mathscr{C}^i} H$
 be the $i$-skeleton of~$\mathscr{C}$.
 For all~$i \in \{0, \, \ldots, \, p\}$, there exists a map
 $\Phi^i\colon\R^{p+1}\to \R^{p+1}$
 (depending on the given~$u\in W^{1,p}(\Omega, \, \NN)$)
 such that
 \begin{enumerate}[label=(\alph*)]
  \item for every~$H\in\mathscr{C}^j$ with~$0 \leq j \leq i$,
  the map~$\phi^{i}$ is constant on the $p - j + 1$ dimensional
  cubes of edge length $2\theta_i h$ which are orthogonal to~$H$
  and centred at a point of~$H$;

  \item $\phi^i(x) = x$ if
  $\dist(x, \, R_{\mathscr{C}^i}) > \theta_{i-1} h$ and
  $\phi^i(U_{\theta_{i-1} h}(R_{\mathscr{C}^i}))
  \subseteq U_{\theta_{i-1} h}(R_{\mathscr{C}^i})$;

  \item $u\circ \phi^i$ is well-defined,
  belongs to~$W^{1,p}(\Omega^\prime, \, \R^m)$,
  and for every for every $H\in\mathscr{C}^i$ we have
  \begin{equation*}
   \norm{\nabla(u \circ \phi^i)}_{L^p(U_{\theta_{i-1} h}(H))}
   \leq C \norm{\nabla u}_{L^p(U_{\theta_{i-1} h}(H))},
  \end{equation*}
  for some constant $C > 0$ depending only on~$p$ and $\theta$.
 \end{enumerate}
 The construction proceeds inductively.
 To define~$\phi^0$, we apply Proposition~\ref{prop:openingcube}
 to the map $u$ around each vertex~$H\in\mathcal{C}^0$,
 choosing~$A := \R^0 = \{0\}$. This allows us to
 define a map~$\phi_0$ such that~$\phi_0$ is constant on
 each ball~$B_{\theta_0 h}(H)$, with~$H\in\mathscr{C}^0$,
 and agrees with the identity on~$\R^{p+1}\setminus\bigcup_{H\in\mathscr{C}^0} B_{2\theta h}(H)$.
 Once we have constructed~$\phi^{i-1}$,
 we apply again Proposition~\ref{prop:openingcube},
 this time to the map~$u\circ\phi^{i-1}$, choosing as~$A$
 the interior of~$H\in\mathscr{C}^i$.
 As a consequence, for each cube~$H\in\mathscr{C}^i$
 we define a map~$\varphi_H$, which is constant on
 the $(p-i+1)$-dimensional cubes of edge length~$2\theta_i h$
 that are orthogonal to~$H$ and centred at a point of~$H$,
 and agrees with the identity out of~$U_{\theta_{i-1} h}(H)$.
 Then, we define $\phi^i\colon\R^{p+1}\to\R^{p+1}$ as
 \begin{equation} \label{opening-i}
  \phi^i(x)=
  \begin{cases}
   \phi^{i-1}(\varphi_{H}(x))
    &\textrm{if } x\in U_{\theta_{i-1} h}(H) \
    \textrm{ for some } H\in\mathscr{C}^i, \\
   \phi^{i-1}(x) &\mathrm{otherwise.}
  \end{cases}
 \end{equation}
 The proof that~$\phi^i$ is well-defined and
 satisfies~(a)--(c) is contained
 in~\cite{BousquetPonceVanSchaftingen-II}.

 We define~$u_h^{\mathrm{op}} := u\circ\phi^p$.
 Properties~(i)--(iii) in the statement follows by~(a)--(c)
 above, while property~(iv) is proved in~\cite[Addendum~2 to Proposition~2.1]{BousquetPonceVanSchaftingen-II}.
 It only remains to prove~(v). Let~$K\in\mathscr{C}$
 and let~$H\in\mathscr{C}^p$ be one of the boundary faces of~$K$.
%  Let~$K^\prime$ be the cube of~$\GG(h, \, y)$
%  such that~$H = \partial K \cap \partial K^\prime$.
 Up to a translation and a rotation, suppose that~$H$
 is contained in the coordinate hyperplane~$\{x_1 = 0\}$,
 while~$K$ is contained in~$\{x_1 \leq 0\}$.
 Writing~$x = (x_1, \, x'')$ for the variable in~$\R^{p+1}$,
 and keeping Proposition~\ref{prop:openingcube}
 and~\eqref{opening-i} into account, we find~$a\in (-\theta h, \, \theta h)$
 such that
 \[
  u_h^{\mathrm{op}}(x) = (u\circ\phi^{p-1})(a, \, x'')
  \qquad \textrm{for } x = (0, \, x'')\in H.
 \]
 We distinguish now three cases, depending on the sign of~$a$.
 If~$a > 0$, le us set~$H^\prime := [0, \, a]\times H$
 and~$\eps(H) := 1$; if~$a < 0$, we set
 $H^\prime := [-a, \, 0]\times H$ and~$\eps(H) := -1$;
 if~$a = 0$, we set~$\eps(H) := 0$.
 We write~$\sigma(w, \, \partial L)$ for the homotopy class
 of a map~$w\in W^{1,p}(\partial L, \, \NN)$ on the boundary of
 a cuboid~$L$ (in the sense defined in
 Section~\ref{sect:homotopy_W1p}).
 We claim that
 \begin{equation} \label{openinghomotopy}
  \sigma(u^{\mathrm{op}}_h, \, \partial K)
  = \sigma(u, \, \partial K) +
   \sum_{H\subseteq\partial K} \eps(H)
   \, \sigma(u, \, \partial H^\prime)
 \end{equation}
 (on the understanding that~$\eps(H) \,
 \sigma(u, \, \partial H^\prime) = 0$ if~$\eps(H) = 0$,
 even if~$\sigma(u, \, \partial H^\prime)$ is undefined).
 Here the sum is taken over all~$p$-dimensional boundary
 faces~$H$ of~$K$. The claim~\eqref{openinghomotopy}
 follows from the definition of the operation in~$\pN = \GN$,
 by noting that~$\partial K \cup \bigcup_{H\subseteq\partial K} \partial H^\prime$ is a deformation retract of
 (hence, homotopically equivalent to) a disk with several holes
 and that~$u^{\mathrm{op}}_h$ is continuous in a neighbourhood
 of~$R_{\mathscr{C}^{p-1}}$, by property~(a) above and
 Sobolev embedding. With~$H\in\GG_p(h, \, y, \, \Omega)$,
 $K\subseteq\{x_1 = 0\}$ and~$K\in\GG(h, \, y, \, \Omega)$,
 $K\subseteq\{x_1 \leq 0\}$ as above,
 let~$K^\prime$ be the cube of~$\GG(h, \, y)$
 such that~$\partial K\cap\partial K^\prime = H$,
 let~$c_K$, $c_{K^\prime}$ be the centres of~$K$,
 $K^\prime$ respectively, and let~$D_H = \lb c_K\rb
 - \lb c_{K^\prime}\rb\in\F_0(\Omega^\prime; \, \Z)$.
 From the definition~\eqref{Sgrid} of~$\Sgrid$
 and~\eqref{openinghomotopy}, we obtain
 \begin{equation*}
  \Sgrid(u^{\mathrm{op}}_h, \, h, \, y)
  = \Sgrid(u, \, h, \, y)
   + \sum_{H\in\GG_{p}(h, \, y, \, \Omega)}
   \eps(H) \sigma(u, \, \partial H^\prime) D_H
  \qquad \textrm{in } \Omega^\prime.
 \end{equation*}
 Then, using Lemma~\ref{lemma:cubebounds}, we can estimate
 from above the~$\F_{\Omega^\prime}$-distance
 between~$\Sgrid(u^{\mathrm{op}}_h, \, y, \, h)$
 and~$\Sgrid(u, \, h, \, y)$:
 \begin{equation*} %\label{opening-S}
  \begin{split}
   \F_{\Omega^\prime}\!\left(
   \Sgrid(u^{\mathrm{op}}_h, \, h, \, y)
   - \Sgrid(u, \, h, \, y)\right)
   &\leq h \sum_{H\in\GG_{p}(h, \, y, \, \Omega)}
   \abs{\sigma(u, \, \partial H^\prime)} \\
   &\lesssim
    h \int_{R_{\partial\mathscr{C}}}
    \left(\abs{\nabla u}^p + \abs{\nabla u_h^{\mathrm{op}}}^p\right)\d\H^{p}
  \end{split}
 \end{equation*}
 The integral of~$\abs{\nabla u}^p$ on~$R_{\partial\mathscr{C}}$
 can be bounded from above using the assumption~\eqref{hp:goodgrid},
 while the contribution of~$u_h^{\mathrm{op}}$ can be estimated using
 properties~(a) and~(c) above.
 Overall, we obtain the inequality
 \begin{equation*} %\label{opening-S}
  \begin{split}
   \F_{\Omega^\prime}\!\left(
   \Sgrid(u^{\mathrm{op}}_h, \, h, \, y)
   - \Sgrid(u, \, h, \, y)\right)
   &\lesssim \int_\Omega \abs{\nabla u}^p \d x,
  \end{split}
 \end{equation*}
 which completes the proof.
\end{proof}

The second tool, which we quote from~\cite{BousquetPonceVanSchaftingen-II},
is ``adaptive smoothing''. Let
\begin{equation} \label{Omega_*}
 \Omega_* := \left\{x\in\R^{p+1} \colon
 \dist(x, \, \partial\Omega) <
 \left(1 - \frac{1}{2\sqrt{n}}\right) \dist(\Omega, \, \partial\Omega^\prime) \right\}
\end{equation}
We have~$\overline{\Omega_*}\subseteq\Omega^\prime$ and,
for all positive value~$h\leq h_0 = \frac{1}{2\sqrt{n}}\dist(\Omega, \, \partial\Omega^\prime)$, all grid~$\GG(h, \, y)$, and all cube~$K\in\GG(h, \, y)$
such that~$K\cap\Omega\neq\varnothing$, $K\subseteq\Omega_*$.
Let~$\varphi \in C_{\mathrm{c}}^\infty(B^{p+1})$ be such that
\begin{equation} \label{mollifier}
 \varphi \geq 0\ \textrm{in } B^{p+1}, \qquad
 \int_{B^{p+1}} \varphi(x) \, \d x = 1.
\end{equation}
For every $u\in W^{1,p}(\Omega^\prime, \, \NN)$,
$s\in [0, \, h_0]$ and~$x \in \Omega_*$, we define
\begin{equation} \label{convolution}
 (\varphi_s \ast u)(x) := \int_{B^{p+1}} \varphi(z) \, u(x + s z) \, \d z.
\end{equation}
With this notation, $\varphi_0 \ast u$ is well-defined and we have
$(\varphi_0 \ast u)(x) = u(x)$ for all~$x\in\Omega_*$.
For any non-negative function~$\psi \in C^\infty(\Omega^\prime)$
such that~$\norm{\psi}_{L^\infty(\Omega^\prime)}\leq h_0$,
Equation~\eqref{convolution} defines a
map~$\varphi_\psi \ast u\colon\Omega_*\to\R^m$.
Moreover, if~$\abs{D\psi (a)} < 1$ at some point $a \in \Omega_*$,
then we can perform a change of variable in~\eqref{convolution}
and  deduce that~$\varphi_\psi \ast u$ is smooth
in a neighbourhood of~$a$.

\begin{prop}[{\cite[Propositions~3.1 and~3.2]
{BousquetPonceVanSchaftingen-II}}] \label{prop:convolution}
 Let $\varphi\in C_{\mathrm{c}}^\infty(B^{p+1})$ satisfy~\eqref{mollifier},
 and let~$\psi \in C^\infty(\Omega^\prime)$ be such
 that~$0 \leq \psi\leq h_0$ in~$\Omega^\prime$,
 $\norm{\nabla\psi}_{L^\infty(\Omega^\prime)} < 1$.
 Then, for every $u\in W^{1,p}(\Omega; \NN)$
 (identified with a map~$W^{1,p}(\Omega^\prime, \, \NN)$
 satisfying~\eqref{extension}), the following properties hold:
 \begin{enumerate}[label=(\roman*)]
%   \item $\varphi_{\psi} \ast u \in L^p(\Omega_*, \, \R^m)$ and
%   \begin{equation*}
%    \norm{\varphi_\psi \ast u}_{L^p(\Omega_*)}
%    \leq \frac{1}{(1 - \norm{\nabla\psi}_{L^\infty(\Omega^\prime)})^\frac{1}{p}}
%     \norm{u}_{L^p(\Omega_*)};
%   \end{equation*}

  \item the map~$\varphi_{\psi} \ast u$ belongs
  to~$L^p(\Omega_*, \, \R^m)$ and we have
  \[
   \norm{\varphi_{\psi} \ast u - u}_{L^p(\Omega_*)}
   \leq \sup_{v \in B^{p+1}}{\norm{\tau_{\psi v}u - u}_{L^p(\Omega_*)}},
  \]
  where~$\tau_{\psi v} u (x) = u(x + \psi(x)v)$
  for all~$v\in\R^{p+1}$;

  \item the map $\varphi_{\psi} \ast u$ belongs
  to~$\in W^{1, p}(\Omega_*, \, \R^m)$ and, for any
  open set~$\omega\subseteq\Omega_*$ and~$t := \norm{\psi}_{L^\infty(\psi)}$,
  we have
  \[
   \norm{\nabla(\varphi_{\psi} \ast u)}_{L^p(\omega)}
   \leq \frac{C}{(1 - \norm{\nabla\psi}_{L^\infty(\Omega^\prime)})^{1/p}}
   \norm{\nabla u}_{L^p(U_t(\omega))}
  \]
  for some constant~$C>0$ depending only on~$p$;

  \item we have
  \[
   \begin{split}
    &\norm{\nabla(\varphi_{\psi} \ast u) - \nabla u}_{L^p(\Omega_*)}\\
    &\hspace{2cm}\leq \sup_{v \in B^{p+1}}{\norm{\tau_{\psi v}(\nabla u) - \nabla u}_{L^p(\Omega_*)}}
    + \frac{C'}{(1 - \norm{\nabla\psi}_{L^\infty(\Omega^\prime)})^{1/p}}
    \norm{\nabla u}_{L^p(A)}
   \end{split}
  \]
  for some constant~$C'>0$ depending only on~$p$, where
%   $\tau_{\psi v} (\nabla u) (x) := \nabla u(x + \psi(x)v)$
%   for all~$v\in\R^{p+1}$, $x\in\Omega_*$ and
  \[
   A = \bigcup_{x\in\Omega_*\cap\spt(\nabla\psi)} B_{\psi(x)}^{p+1}(x).
  \]
 \end{enumerate}
\end{prop}

\begin{proof}[Proof of Proposition~\ref{prop:RS}]
 Let~$u\in W^{1,p}(\Omega, \, \NN)$ be a given map,
 which we identify with a map in~$W^{1,p}(\Omega^\prime, \, \NN)$
 satisfying~\eqref{extension}. Thanks to Lemma~\ref{lemma:Fubini} and an averaging argument,
 for all~$h\in (0, \, h_0]$ we can find~$y=y(u, \, h)\in Q^{p+1}$ such that the estimate~\eqref{hp:goodgrid} is satified.
%  \begin{equation} \label{RS-goodgrid}
%   \int_{R^j(h, \, y)\cap\Omega^\prime} \abs{\nabla u}^p \d\H^j
%    \leq C h^{j - n} \int_\Omega \abs{\nabla u}^p \d x
%    \qquad \textrm{for all } j\in\{0, \, 1, \, \ldots, p\},
%  \end{equation}
%  and moreover
 Without loss of generality, we can also assume that
 \begin{equation} \label{RS0}
  \F_\Omega(\Sgrid(u, \, h, \, y))
   \lesssim \int_{Q^{p+1}} \F_\Omega(\Sgrid(u, \, h, \, z)) \, \d z.
 \end{equation}
 Therefore, the proposition will follow if we prove
 that there exists a family of maps
 $u_h\in R^{1,p}(\Omega, \, \NN)$ that converges
 $W^{1,p}(\Omega)$-strongly to~$u$ as~$h\to 0$ and satisfies
 \begin{equation} \label{RS-main}
  \F_\Omega\!\left(\Snice(u_h) - \Sgrid(u, \, h, \, y)\right)
  \lesssim \int_\Omega \abs{\nabla u}^p \d x
 \end{equation}
 for all~$h\in (0, \, h_0]$ small enough.
 This can be achieved by reasoning exactly
 as in~\cite[Theorem~3]{BousquetPonceVanSchaftingen-II};
 we reproduce some of the arguments for the reader's convenience.
 We fix parameters~$\tau$, $\theta$ with
 \[
  0 < \tau < \theta < \frac{1}{2}.
 \]
 We denote by~$C$ several different constants
 that depend on~$\Omega$, $\NN$, $p$, $\tau$, $\theta$ only.
 We split the rest of the proof into several steps.

 \setcounter{step}{0}
 \begin{step}[bad cubes]
  Let~$\delta_* = \delta_*(\NN) > 0$ be such that the
  nearest-point projection~\eqref{projN} onto~$\NN$
  is well-defined in a~$\delta_*$-neighbourhood of~$\NN$.
  Let~$C$ be a constant, whose value will be chosen later,
  depending only on~$\Omega$, $\NN$, $p$, and~$\theta$.
  We define~$\mathscr{B}_1$ --- the set of ``bad cubes of the first type'' ---
%   in the terminology of~\cite{Bethuel-Density} ---
  as the set of cubes~$K\in\GG_{p+1}(h, \, y, \, \Omega)$ such that
  \begin{equation} \label{bad_cubes}
   \frac{C}{h^{1/p}} \norm{\nabla u}_{L^p(U_{2\theta h}(K))} > \delta_*.
  \end{equation}
  Since the $(2\theta h)$-neighbourhood of~$K$ %$U_{2\theta h}(K)$,
  overlaps with at most $\kappa_p$ other cubes of the grid
  for some constant~$\kappa_p$ depending only on~$p$, we can estimate
  the number of cubes in~$\mathscr{B}_1$, denoted~$\#\mathscr{B}_1$,
  by raising both sides of~\eqref{bad_cubes} to the power
  of~$p$ and summing over~$K\in\mathscr{B}_1$:
  \begin{equation} \label{number_bad_cubes1}
   \#\mathscr{B}_1 \leq \frac{C \kappa_p}{h\delta_*^p} \int_\Omega\abs{\nabla u}^p \d x.
  \end{equation}
  We define~$\mathscr{B}_2$ %--- the set of ``bad cubes of the second type'' ---
  as the set of cubes~$K\in\GG(h, \, y, \, \Omega)$ such that
  the homotopy class of~$u$ on~$\partial K$, as defined in Section~\ref{sect:homotopy_W1p}, is non-trivial.
  The number of cubes in~$\mathscr{B}_2$ can be estimated
  using Proposition~\ref{prop:smallenergy} and~\eqref{hp:goodgrid}:
  \begin{equation} \label{number_bad_cubes2}
   \#\mathscr{B}_2 \leq \frac{2C}{h\alpha_p(\NN)} \int_\Omega \abs{\nabla u}^p \d x,
  \end{equation}
  where~$\alpha_p(\NN) > 0$ depends only on~$\NN$ and~$p$.
  We define the set of ``bad cubes''
  as~$\mathscr{B} := \mathscr{B}_1 \cup \mathscr{B}_2$.
  We also define~$\mathscr{C}$ as the collections of
  cubes that are ``either bad or next to bad ones''
  --- namely, the elements of~$\mathscr{C}$ are the
  cubes~$K\in\GG_{p+1}(h, \, y, \, \Omega)$
  such that~$K\cap K^\prime\neq\varnothing$
  for some~$K^\prime\in\mathscr{B}$. The numbers of cubes
  in~$\mathscr{C}$ is bounded from above as
  \begin{equation} \label{number_bad_cubes}
   \#\mathscr{C} \leq \kappa_p \, (\#\mathscr{B}_1 + \#\mathscr{B}_2)
   \lesssim \frac{1}{h} \int_\Omega \abs{\nabla u}^p \d x .
  \end{equation}
  We write $R_{\mathscr{C}} := \bigcup_{K\in\mathscr{C}} K$
  for the skeleton of~$\mathscr{C}$.
  It follows from~\eqref{number_bad_cubes} that the
  total volume occupied by cubes in~$\mathscr{C}$ is small.
  In fact, we have
  \begin{equation} \label{volume_bad_cubes}
%    \abs{R_{\mathscr{C}}} \leq h^{p + 1} (\#\mathscr{C}) \lesssim h^p, \qquad
   \abs{U_{2\theta h}(R_{\mathscr{C}})}
   \leq h^{p + 1} (\#\mathscr{C})
   \leq C h^p\int_\Omega \abs{\nabla u}^p \d x
  \end{equation}
  for some constant~$C$ that does not depend on~$h$, and hence,
  $\abs{U_{2\theta h}(R_{\mathscr{C}})}\to 0$ as~$h\to 0$.
 \end{step}

 \begin{step}[``opening'']
%   Without loss of generality, and by perturbing slightly
%   the value of~$y$ if necessary, we can assume that
%   the restriction of~$u$ to the~$p$-skeleton~$R_p(h, \, y, \, \Omega)$
%   is well-defined and belongs to~$W^{1,p}(R_p(h, \, y, \, \Omega), \, \NN)$.
  Let~$u^{\mathrm{op}}_h$ be the map given by
  Proposition~\ref{prop:opening}, given~$u$ and the grid~$\GG(h, \, y)$.
  Due to properties~(ii) and~(iii) in Proposition~\ref{prop:opening},
  this map satisfies
  \begin{align}
   \norm{\nabla u_h^{\mathrm{op}}}_{L^p(U_{2\theta h}
    (R_{\mathscr{C}}))}
   &\lesssim \norm{\nabla u}_{L^p(U_{2\theta h}
    (R_{\mathscr{C}}))} \! , \label{RS-opening-bis} \\
   \norm{\nabla u_h^{\mathrm{op}}}_{L^p(U_{2\theta h}
    (R_{\partial\mathscr{C}}))}
   &\lesssim \norm{\nabla u}_{L^p(U_{2\theta h}
    (R_{\partial\mathscr{C}}))} \!, \label{RS-opening-bd} \\
   \norm{\nabla u - \nabla u_h^{\mathrm{op}}}_{L^p(\Omega^\prime)}
   &\lesssim \norm{\nabla u}_{L^p(U_{2\theta h}
    (R_{\partial\mathscr{C}}))} \! . \label{RS-opening}
  \end{align}
%   where~$U_{2\theta h}(R_{\partial\mathscr{C}})$ is defined by~\eqref{U_t}.
  Moreover, property (i) in Proposition~\ref{prop:opening}
  implies that~$u$ ``depends at most on~$p$ variables'' in~$U_{2\theta h}(R_{\partial\mathscr{C}})$
  (more precisely, it is constant on segments of length~$2\theta h$
  orthogonal to one of the $p$-dimensional boundary faces of
  the cubes in~$\mathscr{U}$). By Sobolev embedding, it follows
  that~$u_h^{\mathrm{op}}\in \mathrm{VMO}(U_{2\theta h}(R_{\partial\mathscr{C}}), \, \NN)$.
  In fact, property~(iv) in Proposition~\ref{prop:opening}
  implies that for all~$s \in (0, \, 1)$ small enough, there holds
  \begin{equation} \label{RS-s}
   \begin{split}
    \sup_{a+Q_s^{p+1}\subseteq U_{\theta h}(R_{\partial\mathscr{C}})}
    \frac{1}{s^{2p +2}} \int_{a+Q_s^{p+1}} \int_{a+Q_s^{p+1}}
    \abs{u_h^{\mathrm{op}}(x) - u_h^{\mathrm{op}}(y)} \d x \d y \leq \delta_*.
   \end{split}
  \end{equation}
  while property~(v) reads
  \begin{equation} \label{openingS}
   \F_{\Omega^\prime}\!\left(\Sgrid(u, \, h, \, y)
    - \Sgrid(u^{\mathrm{op}}_h, \, h, \, y)\right)
   \lesssim \int_\Omega \abs{\nabla u}^p \d x.
  \end{equation}
  These properties will be useful in the next steps.
 \end{step}

 \begin{step}[``adaptive smoothing'']
  Let~$\zeta\in C^\infty_{\mathrm{c}}(\Omega^\prime)$
  be a cut-off function, such that~$0 \leq \zeta \leq 1$
  in~$\Omega^\prime$, $\zeta = 1$ on the union of the
  bad cubes~$R_{\mathscr{B}}$, $\zeta = 0$ out of~$R_{\mathscr{C}}$
  (the union of the cubes that are ``bad or next to bad ones''),
  and~$\norm{\nabla\zeta}_{L^\infty(\Omega^\prime)} \lesssim h^{-1}$.
  In particular, there exists~$\kappa \in (0, \, 1)$
  such that $\kappa h \norm{\nabla\zeta}_{L^\infty(\Omega^\prime)} < 1$.
  Let~$s > 0$ be as in~\eqref{RS-s}. By taking a smaller~$s$ if necessary,
  we can assume that $s < \kappa h$, $s < (\theta - \tau)h$.
  We define the function
  \[
   \psi_h := s\zeta + \min(\kappa h, \, (\theta - \tau)h) (1 - \zeta).
  \]
  We have $0 < s \leq \psi_h \leq (\theta - \tau)h \leq h_*$ in~$\Omega^\prime$,
  with~$\psi_h = s$ in~$R_{\mathscr{B}}$ and~$\psi_h = \min(\kappa, \, \theta - \tau)h$
  out of~$R_{\mathscr{C}}$,
  and~$\norm{\nabla\psi_h}_{L^\infty(\Omega^\prime)} < 1$.
  Let~$\eta\in C^\infty_{\mathrm{c}}(-1, \, 1)$ be
  a non-negative function, such that~$\int_{-1}^1 \eta(t) \, \d t = 1$,
  and let~$\varphi^{p+1}(x) := \eta(x_1) \eta(x_2) \ldots \eta(x_{p+1})$, $x\in\R^{p+1}$.
  If the support of~$\eta$ is small enough,
  then~$\varphi^{p+1}$ is compactly supported in~$B^{p+1}$
  and satisfies~\eqref{mollifier}, hence it is an admissible mollifying kernel.
  On the set~$\Omega_*$ given by~\eqref{Omega_*},
  we define~$u_h^{\mathrm{sm}} := \varphi^{p+1}_{\psi_h}\ast u_h^{\mathrm{op}}$.
  Proposition~\ref{prop:convolution} implies that
  \begin{equation*}
   \begin{split}
    \norm{\nabla u_h^{\mathrm{sm}} - \nabla u_h^{\mathrm{op}}}_{L^p(\Omega_*)}
    \lesssim \sup_{v \in B^{p+1}}{\norm{\tau_{\psi v}(\nabla u_h^{\mathrm{op}}) - \nabla u_h^{\mathrm{op}}}_{L^p(\Omega_*)}}
    + C\norm{\nabla u_h^{\mathrm{op}}}_{L^p(U_{2\theta h}(R_{\partial\mathscr{C}}))} \! .
   \end{split}
  \end{equation*}
  Combining this estimate with~\eqref{RS-opening-bd}, we obtain
  \begin{equation} \label{RS-smoothing}
   \begin{split}
    \norm{\nabla u_h^{\mathrm{sm}} - \nabla u_h^{\mathrm{op}}}_{L^p(\Omega_*)}
    \lesssim \sup_{v \in B^{p+1}}
     \norm{\tau_{\psi v}(\nabla u) - \nabla u}_{L^p(\Omega_*)}
    + \norm{\nabla u}_{L^p(U_{2\theta h}(R_{\partial\mathscr{C}}))} \! .
   \end{split}
  \end{equation}
  From property~(ii) in Proposition~\ref{prop:convolution}
  and~\eqref{RS-opening-bis}, we also deduce
  \begin{equation} \label{RS-smoothingbis}
   \begin{split}
    \norm{\nabla u_h^{\mathrm{sm}}}_{L^p(U_{\theta h}(R_{\mathscr{C}}))}
    \lesssim \norm{\nabla u_h^{\mathrm{op}}}_{L^p(U_{2\theta h}(R_{\mathscr{C}}))}
    \lesssim \norm{\nabla u}_{L^p(U_{2\theta h}(R_{\mathscr{C}}))} \! .
   \end{split}
  \end{equation}
  Let~$K\in\mathscr{B}$ and~$H\subseteq\partial K$ a boundary $p$-face
  of~$K$. Suppose, for instance, that~$H$ is parallel to
  the coordinate plane~$\{x_1 = 0\}$. By Proposition~\ref{prop:opening},
  we know that~$u_h^{\mathrm{op}}$ is constant on all straight
  line segments orthogonal to~$H$,
  with length~$2\theta h > 2s$ and midpoint in~$H$.
  Since~$u_h^{\mathrm{sm}} = \varphi_{sh}\ast u_h^{\mathrm{op}}$
  and~$\varphi^{p+1}_s(x) = \eta_s(x_1) \,
  \varphi^{p+1}_s(x_2, \, \ldots, \, x_{p+1})$, we deduce
  \begin{equation} \label{RSclassicalconvolution}
   \nabla u_h^{\mathrm{sm}}
   = \varphi_s^p \ast \nabla u_h^{\mathrm{op}}
  \end{equation}
  on each $p$-cell~$H\subseteq R_{\partial\mathscr{B}}$ and, hence,
  \begin{equation} \label{RS-smoothingtris}
   \begin{split}
    h \norm{\nabla u_h^{\mathrm{sm}}}_{L^p(\partial\mathscr{B})}
    \leq h \norm{\nabla u_h^{\mathrm{op}}}_{L^p(\partial\mathscr{B})}
    \lesssim \norm{\nabla u_h^{\mathrm{op}}}_{L^p(U_{2\theta h}
    (R_{\partial\mathscr{C}}))}
    \lesssim \norm{\nabla u}_{L^p(U_{2\theta h}
    (R_{\partial\mathscr{C}}))}\!,
   \end{split}
  \end{equation}
  thanks to~\eqref{RS-opening-bd}.
 \end{step}

 \begin{step}[``thickening'']
  Given a grid~$\GG(h, \, y)$ and a collection of
  cubes~$\mathscr{B}\subseteq\GG_{p+1}(h, \, y, \, \Omega)$,
  we denote by~$\mathscr{B}^\prime$ the dual skeleton of~$\mathscr{B}$,
  that is, the set of all the centres of the cubes of~$\mathscr{B}$.
  We define a map~$u_h^{\mathrm{th}}$ as follows.
  On ``good cubes''~$K\in\GG(h, \, y, \Omega)\setminus\mathscr{B}$,
  we set~$u_h^{\mathrm{th}} := u_h^{\mathrm{sm}}$.
  On a bad cube~$K\in\mathscr{B}$, we define~$u_h^{\mathrm{th}}$
  by homogeneous extension of the values of~$u_h^{\mathrm{sm}}$
  on~$\partial K$ --- that is,
  \[
   u_h^{\mathrm{th}}(x) :=
   u_h^{\mathrm{sm}}\left(c_K +
    \frac{x - c_K}{\abs{x - c_K}_\infty}\right) \! ,
    \qquad x\in K,
  \]
  where~$c_K$ is the centre of~$K$ and~$\abs{y}_\infty :=
  \max_{1 \leq j \leq p+1} \abs{y_j}$.
  The map~$u_h^{\mathrm{th}}$ is well-defined and locally Lipschitz
  in~$\overline{\Omega}\setminus\mathscr{B}^\prime$, with
  \begin{equation} \label{nicethicksing}
   \abs{\nabla u_h^{\mathrm{th}}(x)}
   \lesssim \frac{\norm{\nabla u_h^{\mathrm{sm}}}_{L^\infty(\Omega_*)}}
    {\dist(x, \, \mathscr{B}^\prime)},
    \qquad \textrm{for } x\in\overline{\Omega}\setminus\mathscr{B}^\prime.
  \end{equation}
  We have
  \[
   \begin{split}
    \norm{\nabla u_h^{\mathrm{th}} - \nabla u_h^{\mathrm{sm}}}_{L^p(\Omega_*)}
    &\leq \norm{\nabla u_h^{\mathrm{th}}}_{L^p(R_{\mathscr{B}})}
     + \norm{\nabla u_h^{\mathrm{sm}}}_{L^p(R_{\mathscr{B}})} \\
    &\lesssim h\norm{\nabla u_h^{\mathrm{sm}}}_{L^p(R_{\partial\mathscr{B}})}
     + \norm{\nabla u_h^{\mathrm{sm}}}_{L^p(R_{\mathscr{B}})}
   \end{split}
  \]
  and hence, keeping~\eqref{RS-smoothingbis} and~\eqref{RS-smoothingtris}
  into account,
  \[
   \begin{split}
    \norm{\nabla u_h^{\mathrm{th}} - \nabla u_h^{\mathrm{sm}}}_{L^p(\Omega_*)}
    \lesssim \norm{\nabla u}_{L^p(U_{2\theta h}(R_{\mathscr{C}}))} \! .
   \end{split}
  \]
  Combining this estimate with~\eqref{RS-smoothing}, we deduce
  \begin{equation} \label{RS-thickening}
   \begin{split}
    \norm{\nabla u_h^{\mathrm{th}} - \nabla u}_{L^p(\Omega_*)}
    \lesssim \sup_{v \in B^{p+1}}
     \norm{\tau_{\psi v}(\nabla u) - \nabla u}_{L^p(\Omega_*)}
    + \norm{\nabla u}_{L^p(U_{2\theta h}(R_{\mathscr{C}}))} \! .
   \end{split}
  \end{equation}
  Moreover, exactly as in~\cite[Theorem~3, Step~2]{BousquetPonceVanSchaftingen-II},
  there exists a constant~$C$, depending only on~$p$,
  $\tau$ and~$\theta$, such that
  \begin{equation} \label{RS-thick-bis}
   \begin{split}
    \dist(u^\mathrm{th}_h(x), \, \NN)
    &\leq C \max \biggl\{  \max_{K\in \GG(h, \, y, \, \Omega) \setminus \mathscr{B}} h^{-1/p} \norm{\nabla u}_{L^p(U_{2\theta h}(K))},\\
    &\qquad \sup_{a \in U_{\tau h}(R_{\partial\mathscr{C}})}
    \frac{1}{\abs{Q_{s}^m}^2} \int_{a + Q_s^m}\int_{a + Q_s^m} \abs{u^\mathrm{op}_h(x) - u^\mathrm{op}_h(y)} \, \d x \, \d y\biggr\}
   \end{split}
  \end{equation}
  for all~$x\in\Omega_*\setminus\mathscr{B}^\prime$.
  Combining~\eqref{RS-thick-bis} with~\eqref{bad_cubes} and~\eqref{RS-s}, we deduce
  \begin{equation} \label{RS-thickN}
   \begin{split}
    \dist(u^\mathrm{th}_h(x), \, \NN)\leq \delta_*
    \qquad \textrm{for all } x\in\Omega_*\setminus\mathscr{B}^\prime,
   \end{split}
  \end{equation}
  where~$\delta_* = \delta_*(\NN)$ is such that
  the nearest point projection~$\pi_{\NN}$ onto~$\NN$,
  given by~\eqref{projN}, is well-defined and smooth in~$U_{\delta_*}(\NN)$.
 \end{step}

 \begin{step}[projection onto~$\NN$]
  Let~$u_h := \pi_{\NN}\circ u^\mathrm{th}_h$.
  Thanks to~\eqref{RS-thickN}, this map is well-defined,
  locally Lipschitz in~$\overline{\Omega}\setminus \mathscr{B}^\prime$,
  and belongs to~$u_h\in R^{1,p}(\Omega, \, \NN)$
  thanks to~\eqref{nicethicksing}. Moreover, we have
  \begin{equation*} %\label{RS-thickening}
   \begin{split}
    \norm{\nabla u_h - \nabla u}_{L^p(\Omega_*)}
    \lesssim \sup_{v \in B^{p+1}}
     \norm{\tau_{\psi v}(\nabla u) - \nabla u}_{L^p(\Omega_*)}
    + \norm{\nabla u}_{L^p(U_{2\theta h}(R_{\mathscr{C}}))} \to 0
   \end{split}
  \end{equation*}
  as~$h\to 0$, because of~\eqref{RS-thickening}.
  We claim that~$u_h \to u$ in $W^{1,p}(\Omega)$ as~$h\to 0$.
  By construction, we have~$u_h := \pi_{\NN}\circ u^\mathrm{sm}_h$ out
  of~$R_{\mathscr{B}}$.
  Since~$\abs{R_{\mathscr{B}}}\to 0$ as~$h\to 0$
  (by~\eqref{volume_bad_cubes}) and~$u^\mathrm{sm}_h \to u$
  in measure (by property~(ii) in Proposition~\ref{prop:opening}
  and~(i) in Proposition~\ref{prop:convolution}),
  we conclude that $u_h \to u$ in measure as~$h\to 0$.
  This proves the claim.
 \end{step}

 \begin{step}[Proof of~\eqref{RS-main}]
  To conclude the proof, it only remains to prove~\eqref{RS-main}.
  In light of~\eqref{openingS}, it suffices to show that
  \begin{equation} \label{RS-whatremainstodo}
   \Snice(u_h) = \Sgrid(u^{\mathrm{op}}_h, \, h, \, y)
  \end{equation}
  for all small enough~$h$. By construction, both $\Snice(u_h)$
  and~$\Sgrid(u^{\mathrm{op}}_h, \, h, \, y)$ are supported on
  the centres of bad cubes (recall that, by definition of bad cubes,
  $u$ and hence~$u^{\mathrm{op}}_h$ are homotopically trivial
  on the boundary of each good
  cube~$K\in\GG(h, \, y, \, \Omega)\setminus\mathscr{B}$).
  For each~$K\in\mathscr{B}$, we have
  \[
   u_h = \pi_{\NN}\circ u^{\mathrm{th}}_h
   = \pi_{\NN}\circ (\varphi^p_{sh} \ast u^{\mathrm{op}}_h)
   \qquad \textrm{on } \partial K,
  \]
  due to~\eqref{RSclassicalconvolution}, and~$s$ is small
  enough that~\eqref{RS-s} holds. Therefore, the homotopy class
  of~$u^{\mathrm{op}}_h$ on~$\partial K$ coincides with the homotopy class
  of~$u_h$ on~$\partial K$, by the very definition of the latter.
  This implies~\eqref{RS-whatremainstodo}.
  \qedhere
 \end{step}
\end{proof}

\subsection{Proof of Proposition~\ref{prop:dipole}}
\label{sect:dipole}

The proof of this result is based on the so-called
``dipole construction'', introduced in~\cite[Theorem~2]{Bethuel-Density}
and revisited in several places (see e.g.~\cite[Proposition~5.1]{PakzadRiviere}).
We still provide a proof, because we were not able to find a statement in the literature that applies under the exact same hypothesis we consider here.

\begin{lemma} \label{lemma:removable}
 Let~$g\colon \partial B^{p+1}_r\to\NN$ be a Lipschitz map that
 is homotopic to a constant. Then, there exists
 a Lipschitz extension~$v\colon\overline{B}^{p+1}_r\to\NN$
 of~$g$ that satisfies
 \begin{equation} \label{removable}
  \int_{B^p_r} \abs{\nabla v}^p \d x
  \leq C r \int_{\partial B^{p+1}_r} \abs{\nabla g}^p\, \d\H^p
 \end{equation}
 for some constant~$C_p$ that depends only on~$p$.
\end{lemma}
\begin{proof}
 Since the estimate~\eqref{removable} is scale-invariant,
 up to a rescaling we can assume without loss of
 generality that~$r = 1$. By a regularisation argument
 (using convolution and the projection~\eqref{projN} onto~$\NN$),
 we can also assume that~$g$ is smooth.
 Let~$\hat{g}\colon \overline{B}^{p+1}\to\NN$
 be defined as~$\hat{g}(x) := g(x/\abs{x})$,
 $x\in \overline{B}^{p+1}\setminus\{0\}$.
 Then, \cite[Lemma~5]{BethuelZheng} yields
 \[
  \inf\left\{\int_{B^{p+1}} \abs{\nabla v}^p \d x\colon
   v\in C^\infty(\overline{B}^{p+1}, \, \NN),
   \, v = g \ \textrm{ on } \partial B^{p+1} \right\}
  \leq \frac{3}{2}\left(3^p + 1\right)
   \int_{B^{p+1}} \abs{\nabla \hat{g}}^p \d x
 \]
 and the lemma follows.
\end{proof}

\begin{remark} \label{rk:smoothremoved}
 If~$g$ is smooth, then we can construct a
 smooth extension~$v\colon\overline{B}^{p+1}_r\to\NN$
 that satisfies~\eqref{removable}, as in \cite[Lemma~5]{BethuelZheng}.
\end{remark}

For the next lemma, we set~$\Lambda(H, \, r) := [0, \, H]\times \overline{B}^{p}_r$ and~$\Gamma(H, \, r) := [0, \, H]\times \partial B^{p}_r$. We write~$x = (x_1, \, x^\prime)$ for the variable in~$\R^{p+1} = \R\times\R^p\supseteq\Lambda(H, \, r)$.

\begin{lemma} \label{lemma:emptycylinder}
 Let~$0 < r < H$, let~$h\colon [0, \, H]\to\NN$ be a Lipschitz function, and let~$\sigma\in\pN$. Then, there exists a Lipschitz map~$v\colon \Lambda(H, \, r)\to\NN$ that is such that~$v(x) = h(x_1)$ for all~$x = (x_1, \, x^\prime)\in\Gamma(H, \, r)$, has homotopy class~$\sigma$ on each slice~$\{x_1\}\times B^p_r$ with~$x_1\in (0, \, H)$, and satisfies
 \begin{align}
  \int_{\Lambda(H, \, r)} \abs{\nabla v}^p \d x
  &\leq C r^p H\norm{h^\prime}^p_{L^\infty(0, \, H)}
   + C H\abs{\sigma}, \label{emptycyl1} \\
   \int_{\{0, \, H\}\times B^p_r} \abs{\nabla v}^p \d\H^p
  &\leq C \abs{\sigma}, \label{emptycyl2}
 \end{align}
 for some constant~$C$ that depends only on~$p$ and~$\NN$.
\end{lemma}
\begin{proof}
 Let~$n\geq 1$ be the largest integer such that~$n\leq H/r$,
 and let~$t_j := jH/n$, $y_j := h(t_j)$ for~$j\in\{0, \, 1, \, \ldots, \, n\}$.
 For any such~$j$, there exists a smooth map~$v_j\colon\overline{B}^p_r\to\NN$
 such that~$v_j = y_j$ on~$\partial B^p_r$, $v_j$
 belongs to the homotopy class~$\sigma$, and
 \[
  \int_{B^p_r} \abs{\nabla v_j}^p\d\H^p
  \leq 2 \abs{\sigma}_{y_j} \! ,
 \]
 where~$\abs{\,\cdot\,}_{y_j}$ is the norm defined by~\eqref{norm_ball}.
 Statements~\eqref{item:equivalentnorms} and~\eqref{item:samenorms}
 in Lemma~\ref{lemma:norm_p} imply that
 \begin{equation} \label{cyl0}
  \int_{B^p_r} \abs{\nabla v_j}^p\d\H^p \leq C \abs{\sigma} \! ,
 \end{equation}
 where~$\abs{\,\cdot\,} = \abs{\,\cdot\,}_{z_0}$ is the norm
 relative to the base-point~$z_0$ we chose once and for all.
 Let
 \[
  \Sigma := \bigcup_{j=0}^n (\{t_j\}\times B^p_r)
  \cup\Gamma(H, \, r).
 \]
 We define a map $v\colon\Sigma\to\NN$ as 
 \[
  v(x) := 
  \begin{cases}
   h(x_1) 
    &\textrm{if } x = (x_1, \, x^\prime)\in\Gamma(H, \, r), \\
   v_j(x^\prime) 
    &\textrm{if } x = (x_1, \, x^\prime)\in\{t_j\}\times B^p_r.
  \end{cases}
 \]
 The map~$v$ is well-defined and Lipschitz continuous,
 and satisfies
 \begin{equation} \label{cyl1}
  \begin{split}
   \int_{\Sigma} \abs{\nabla v}^p \d\H^p
   &\leq \norm{h^\prime}_{L^\infty(0, \, H)} \abs{\Gamma(H, \, r)}
   + \sum_{j=0}^{n} \int_{B^p_r} \abs{\nabla v_j}^p\d\H^p \\
   &\leq Cr^{p-1}H\norm{h^\prime}^p_{L^\infty(0, \, H)}
    + \frac{CH}{r} \abs{\sigma}\!.
  \end{split}
 \end{equation}
 The last inequality follows from~\eqref{cyl0},
 taking into account that~$n \leq H/r$.
 For~$j\in\{0, \, 1, \, \ldots, \, n-1\}$,
 let~$Q_j := [t_j, \, t_{j+1}]\times B^p_r$.
 We claim that, for all~$j$, the homotopy class 
 of~$v_{|\partial Q_j}$ is trivial. Indeed,
 Lemma~\ref{lemma:actionofpi1} implies that
 the homotopy class of~$v$ restricted 
 to~$\partial Q_j \setminus(\{t_j\}\times B^p_r)$
 is the same as the homotopy class of~$v$ 
 on~$\{t_{j+1}\}\times B^p_r$,
 hence the same as the homotopy class of~$v$
 on~$\{t_{j}\}\times B^p_r$ --- that is, $\sigma$.
 Taking the orientation into account, it follows
 that the homotopy class of~$\partial Q_j$ is~$\sigma - \sigma = 0$.
 Moreover, the set~$Q_j$ is biLipschitz equivalent to a ball
 --- that is, there exists an invertible, Lipschitz map
 $\phi_j\colon Q_j\to \overline{B}^{p+1}_r$ with Lipschitz inverse.
 In fact, since~$t_{j+1} - t_j = H/n$
 and~$\mbox{$H/(n+1)$} < r \leq H/n$, we can choose~$\phi_j$ so that
 \[
  \norm{\nabla\phi_j}_{L^\infty(Q_j)}
   + \norm{\nabla(\phi_j^{-1})}_{L^\infty(\overline{B}_r^p)}
   \leq C,
 \]
 where the constant~$C$ depends only on~$p$.
 Then, up to composition with~$\phi_j$,
 we can apply Lemma~\ref{lemma:removable}
 to extend~$v$ in a Lipschitz way inside each~$Q_j$.
 The extended map --- still denoted~$v$, by abuse 
 of notation --- satisfies
 \[
  \int_{\Lambda(H, \, r)} \abs{\nabla v}^p \d x
  \leq C r \int_{\Sigma} \abs{\nabla v}^p \, \d\H^p
  \leq Cr^pH\norm{h^\prime}^p_{L^\infty(0, \, H)}
    + CH \abs{\sigma}\!,
 \]
 thanks to~\eqref{removable} and~\eqref{cyl1}.
 The inequality~\eqref{emptycyl1} follows.
 The estimate~\eqref{emptycyl2} is an immediate consequence of~\eqref{cyl0}.
\end{proof}

\begin{lemma} \label{lemma:fullcylinder}
 Let~$0 < r < H$, let~$u\colon\Lambda(H, \, r)\to\NN$ be a Lipschitz function, and let~$\sigma\in\pN$. Then, there exists a Lipschitz map~$v\colon\Lambda(H, \, r)\to\NN$ that satisfies the following properties:
 \begin{enumerate}[label=(\roman*)]
  \item $v = u$ on~$\Gamma(H, \, r)$;
  \item $v(x) = u(x_1, \, 0)$ for~$x \in\Gamma(H, \, r/2)$;
  \item $v$ has homotopy class~$\sigma$ on each
  slice~$\{x_1\}\times B_{r/2}^p$, with~$x_1 \in [0, \, H]$;
  \item there holds
  \begin{align}
   \int_{\Lambda(H, \, r)} \abs{\nabla v}^p \d x
   &\leq C r^p H\norm{\nabla u}^p_{L^\infty(\Lambda(H, \, r))}
    + C H\abs{\sigma}, \label{fullcyl1} \\
   \int_{\{0, \, H\}\times B^p_r} \abs{\nabla v}^p \d\H^p
   &\leq C r^p \norm{\nabla u}^p_{L^\infty(\Lambda(H, \, r))}
    + C \abs{\sigma}, \label{fullcyl2}
  \end{align}
  for some constant~$C$ that depends only on~$p$ and~$\NN$.
 \end{enumerate}
\end{lemma}
\begin{proof}
 We define~$v\colon\Lambda(H, \, r)\setminus\Lambda(H, \, r/2)\to\NN$ by
 \[
  v(x) := u\!\left(x_1, \, 
   \left(2 - \frac{r}{\abs{x^\prime}}\right)x^\prime \right)
 \]
 for~$x = (x_1, \, x^\prime)\in [0, \, H]\times
 \overline{B}^p_r\setminus B^p_{r/2}$.
 The map~$v$ is Lipschitz continuous and satisfies~(i), (ii).
 We extend~$v$ in~$\Lambda(H, \, r/2)$ by applying
 Lemma~\ref{lemma:emptycylinder}.
\end{proof}

\begin{proof}[Proof of Proposition~\ref{prop:dipole}]
 Let~$u\in R^{1,p}(\Omega, \, \NN)$ be a given map,
 locally Lipschitz in the complement of a finite set
 $\Sigma = \Sigma(u)\subseteq\Omega$.

 \setcounter{step}{0}
 \begin{step}[decomposing~$\Snice(u)$ as a sum of dipoles]
  By definition~\eqref{Stop_nice},
  the chain~$\Snice(u)$ is supported in~$\Sigma$.
  Let us consider a decomposition of~$\Snice(u)$ of the form
  \[
   \Snice(u) = \sum_{i=1}^N \sigma_i \, \lb x_i\rb
    + \sum_{j=1}^M \tau_j \, \partial \lb L_j \rb
   \qquad \textrm{in } \Omega,
  \]
  such that~$\sigma_i\in\G$, $\tau_j\in\G$, $x_i\in\Omega$, and each~$L_j$ is an oriented line segment in~$\overline{\Omega}$ that intersect~$\partial\Omega$ at most at one of its endpoints.
  Up to a subdivision, we can assume that, for~$j\neq k$, the intersection~$L_j\cap L_k$, the intersection is either empty or a common endpoint of~$L_j$, $L_k$.
  Up to a small perturbation, we can make sure that each segment~$L_j$ meeting~$\partial\Omega$ at an endpoint~$x\in L_j\cap\partial\Omega$ is transversal to~$\partial\Omega$, that is, $L_j$ approaches $x$ inside an interior cone for~$\Omega$ at $x$.
  Moreover, we can choose this decomposition in such a way that
  \begin{equation*} %\label{dipole0}
   \sum_{i=1}^N \abs{\sigma_i}
    + \sum_{j=1}^M \abs{\tau_j} \H^1(L_j)
   \leq 2 \, \F_\Omega(\Snice(u))
  \end{equation*}
  (see~\eqref{flat} and Remark~\ref{rk:dipolarOmega}).
  We can assume without loss of generality that
  all the~$\sigma_i$'s are zero, for otherwise, we
  replace~$\sigma_i \lb x_i\rb$ with~$\sigma_i \, \partial\lb L_i^*\rb$, where~$L_i^*$ is a line segment that connects~$x_i$
  with a point in~$\partial\Omega$ and is transversal to the boundary of the domain.
  Since $\H^1(L_i^*)\leq \mathrm{diam}\,\Omega$,
  where~$\mathrm{diam}\,\Omega$ is the diameter of~$\Omega$, after a relabelling of the indices we have
  \begin{equation} \label{dipole0}
   \Snice(u) = \sum_{j=1}^M \tau_j \, \partial\lb L_j \rb
    \quad \textrm{in } \Omega,
   \qquad \sum_{j=1}^M \abs{\tau_j} \H^1(L_j)
    \leq C \, \F_\Omega(\Snice(u))
  \end{equation}
  for some constant~$C$ depending only on~$\Omega$.
 \end{step}
 
 \begin{figure}%[h]
  \centering
  \includegraphics[height=0.3\textheight]{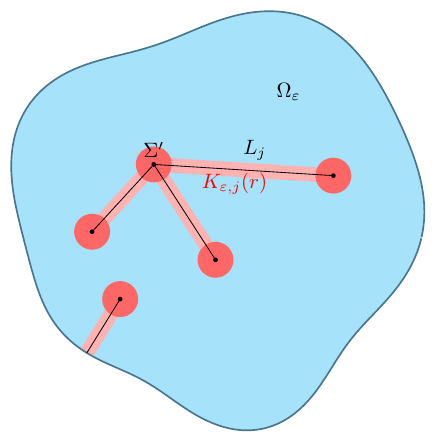}
  \caption{The geometry of the domain in Step~2 of Proposition~\ref{prop:dipole}. 
  Each~$K_{\eps,j}(r)$ is shown in light red, while~$\Omega_\eps$ is the union of the blue region and the~$K_{\eps,j}(r)$'s.}
  \label{fig:dipoles}
 \end{figure}

 \begin{step}[modifying~$u$ ``along the dipoles'']
  For each line segment~$L$ of endpoints~$a$, $b$,
  we denote~$\partial L := \{a, \, b\}$. Let
  \[
   \Sigma^\prime := \Sigma
   \cup \bigcup_{j=1}^M (\partial L_j\cap\Omega)
   \subseteq\Omega.
  \]
  Let~$\eps > 0$ be a parameter small enough that the closed
  balls~$\overline{B}^{p+1}_{2\eps}(x)$, for~$x\in\Sigma^\prime$,
  are pairwise disjoint and contained in~$\overline{\Omega}$, and let
  \[
   \Omega_\eps := \Omega \setminus
    \bigcup_{x\in\Sigma^\prime} \overline{B}_\eps^{p+1}(x),
   \qquad \Gamma_\eps :=
    \bigcup_{x\in\Sigma^\prime} \partial B_\eps^{p+1}(x)
  \]
  (see Figure~\ref{fig:dipoles}).
  The map~$u$ is Lipschitz continuous in~$\overline{\Omega}_\eps$ and satisfies
  \begin{equation} \label{dipole1}
   \norm{\nabla u}_{L^\infty(\Omega_\eps)} \leq \frac{C}{\eps}
  \end{equation}
  for some~$\eps$-independent constant~$C$ (see~\eqref{nicesing}).
  Given a parameter~$r\in (0, \, \eps/100)$, we define
  \begin{align}
   K_{\eps,j}(r) &:= \left\{x\in\overline{\Omega_\eps}\colon
    \dist(x, \, L_j) \leq r\right\} \! , \label{K_eps,j,r} \\
   \Gamma_{\eps,j}(r) &:= \left\{x\in\overline{\Omega_\eps}\colon
    \dist(x, \, L_j) = r\right\} \! , \label{Gamma_eps,j,r}
  \end{align}
  for~$j\in\{1, \, \ldots, \, d\}$.
  Since the segments~$L_j$ intersect at most along their boundaries,
  and reducing the value of~$r$ if necessary,
  we can assume that the sets~$K_{\eps,j}$ are pairwise disjoint.
  The sets~$K_{\eps,j}$ are not cylinders, but are homeomorphic
  to cylinders. More precisely, letting~$H_j$ denote the length
  of~$L_j$, for each~$j$ there exists an invertible, Lipschitz
  map~$\phi_j\colon K_{\eps,j}\to\Lambda(H_j-2\eps, \, r)$
  with Lipschitz inverse, such that
  \begin{equation} \label{Lipschitzgraph}
   \norm{\nabla\phi_j}_{L^\infty(K_{\eps,j})}
    + \norm{\nabla(\phi_j^{-1})}_{L^\infty(\Lambda(H_j-2\eps, \, r))}
    \leq C
  \end{equation}
  for some constant~$C$ that does not depend on~$\eps$, $r$.
  In case one of the endpoints of~$L_j$ lies on~$\partial\Omega$, the fact that~$K_{\eps,j}$ can be mapped to a cylinder in a biLipschitz way depends on the assumption that~$\Omega$ is Lipschitz, hence its boundary~$\partial\Omega$ can locally be written as a Lipschitz graph, and that~$L_j$ meets~$\partial\Omega$ transversally.
  In this case, the constant~$C$ in~\eqref{Lipschitzgraph} will depend on the Lipschitz regularity of~$\partial\Omega$, as well as the angle between~$L_j$ and~$\partial\Omega$, but can still be bounded uniformly with respect to~$j$ (since there are only finitely many segments), $\eps$ and~$r$ (so long as~$\eps$, $r$ are small enough).
  Up to composing with~$\phi_j$ and its inverse, we
  can apply Lemma~\ref{lemma:fullcylinder} and
  construct a map~$v_{\eps,j}$ that coincides with~$u$
  on~$\Gamma_{\eps,j}(r)$, depends only on the variable
  along the direction of~$L_j$ on~$\Gamma_{\eps,j}(r/2)$,
  has homotopy class~$\tau_j\in\pN$ on each
  transversal slice of radius~$r/2$ across~$L_j$
  (with~$\tau_j$ as in~\eqref{dipole0}), and satisfies
  \begin{align}
   \int_{K_{\eps,j}} \abs{\nabla v_{\eps,j}}^p \d x
   &\leq \frac{Cr^p H_j}{\eps^p}
    + C H_j \abs{\tau_j}, \label{dipole2} \\
   \int_{\Gamma_\eps\cap\overline{K}_{\eps,j}} \abs{\nabla v_{\eps,j}}^p \d\H^p
   &\leq \frac{Cr^p}{\eps^p}
    + C \abs{\tau_j}, \label{dipole2bis}
  \end{align}
  The inequalities~\eqref{dipole2}, \eqref{dipole2bis} follow
  from Statement~(iv) in Lemma~\ref{lemma:fullcylinder}
  and~\eqref{dipole1}.
  Note that the homotopy class of~$v_{\eps,j}$ on the
  transversal slices depend on the orientation of the slices.
  We choose the orientation of the slices consistent with that of~$L_j$.
  In other words, if~$v_j$ is a unit vector along the oriented direction of~$L_j$, we define positive bases $(\tau_1, \, \ldots, \, \tau_p)$ of the orthogonal~$p$-plane to~$L_j$
  as those such that~$(v_j, \, \tau_1, \, \ldots, \, \tau_p)$
  is a positive basis of~$\R^{p+1}$.
 \end{step}

 \begin{step}[definition of~$v_\eps$]
  We define a map~$v_\eps\colon\Omega_\eps\to\NN$
  as~$v_\eps := v_{\eps,j}$ in~$K_{\eps,j}$ for all~$j$,
  $v_\eps := u$ in~$\Omega_\eps\setminus\bigcup_{j=1}^d K_{\eps,j}$.
  Thanks to~\eqref{dipole0} and~\eqref{dipole2},
  the map~$v_\eps$ satisfies
  \begin{equation*}
   \int_{\Omega_\eps} \abs{\nabla v_\eps}^p \d x
   \leq \int_{\Omega} \abs{\nabla u}^p \d x
    + \frac{Cr^p}{\eps^p\,\alpha_p(\NN)} \F_\Omega(\Snice(u))
    + C \, \F_\Omega(\Snice(u)) ,
  \end{equation*}
  where~$\alpha_p(\NN)$ is given by Lemma~\ref{lemma:norm_p}.
  Since we have taken~$r < \eps/100$, we have
  \begin{equation} \label{dipole3}
   \int_{\Omega_\eps} \abs{\nabla v_\eps}^p \d x
   \leq \int_{\Omega} \abs{\nabla u}^p \d x
    + C \, \F_\Omega(\Snice(u)).
  \end{equation}
  Moreover, the inequalities~\eqref{dipole0} and~\eqref{dipole2bis} yield
  \begin{equation} \label{dipole3bis}
   \int_{\Gamma_\eps} \abs{\nabla v_\eps}^p \d\H^p
   \leq C + C\sum_{j=1}^M \abs{\tau_j}
   \leq C\sum_{j=1}^M \abs{\tau_j} \! ,
  \end{equation}
  where the second inequality is a consequence of Statement~\eqref{item:alpha_p} in Lemma~\ref{lemma:norm_p}.
  Finally, since~$v_\eps$ agrees with~$u$
  out of~$\bigcup_{j=1}^M K_j$ and takes values
  in the compact manifold~$\NN$, choosing~$r$ small enough
  we can make sure that
  \begin{equation} \label{dipole4}
   \norm{u - v_\eps}_{L^p(\Omega_\eps)}
   \leq \eps .
  \end{equation}
 \end{step}

 \begin{figure}%[h]
  \centering
  \includegraphics[height=0.3\textheight]{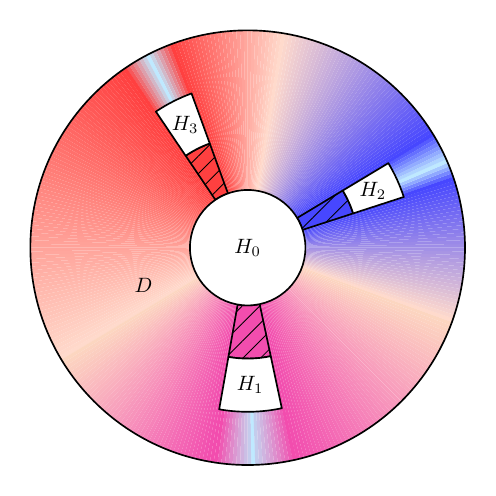}
  \caption{Proof of~\eqref{dipole-homotopy}: construction of
  the map~$w\colon D\to\NN$. The map~$w$ is obtained by
  extending radially the boundary data on the outer sphere,
  except in the three hatched regions, where~$w$ is constant. }
  \label{fig:homotopycolour}
 \end{figure}

 \begin{step}[$v_\eps$ is homotopically trivial
 on~$\partial B_{\eps}^{p+1}(x)$ for all~$x\in\Sigma^\prime$]
  Let~$x\in\Sigma^\prime$. Up to a relabelling of the indices,
  suppose that~$x$ is an endpoint of the segments~$L_1$, \ldots, $L_q$
  and does not belong to~$L_{q+1}$, \ldots, $L_M$.
  By possibly changing the sign of~$\tau_j$, we assume that
  $L_j$ is pointing away from~$x$ for all~$j\in\{1, \, \ldots, \, q\}$. Let~$\sigma(u, \, x)$ be the homotopy
  class of~$u$ restricted to~$\partial B^{p+1}_\eps(x)$,
  and let~$\sigma(v_\eps, \, x)$ be the homotopy
  class of~$v_\eps$ restricted to the same sphere.
  We claim that
  \begin{equation} \label{dipole-homotopy}
   \sigma(v_\eps, \, x) = \sigma(u, \, x) + \sum_{j=1}^q \tau_j.
  \end{equation}
  Since~$\sigma(u, \, x) = -\sum_{j=1}^q \tau_j$
  (because of~\eqref{dipole0}),
  from~\eqref{dipole-homotopy} we will deduce
  \begin{equation} \label{dipole5}
   \sigma(v_\eps, \, x) = 0.
  \end{equation}
  The proof of~\eqref{dipole-homotopy} is based on the definition of the
  operation on~$\GN = \pN$ and consists of a topological construction,
  illustrated in Figure~\ref{fig:homotopycolour}:
  we define a continuous map~$w\colon D\to\NN$,
  where~$D$ is the ball~$B_\eps^{p+1}(x)$
  with~$q+1$ holes~$H_0$, $H_1$ \ldots, $H_q$,
  such that~$w$ agrees with~$v_\eps$
  on~$\partial B^{p+1}_\eps(x)$, is homotopic to~$u$ on~$\partial H_0$,
  and has homotopy class~$\tau_j$ on~$\partial H_j$.
  More precisely, let
  \[
   \psi\colon \overline{B}^{p+1}_\eps(x)\to\partial B^{p+1}_\eps(x), \qquad
   \psi(y) := \frac{\eps(y - x)}{\abs{y - x}}.
  \]
  We define $H_0 := \overline{B}_{\eps/4}^{p+1}(x)$ and
  \[
   H_j := \left\{y\in B^{p+1}_\eps\colon
   \frac{\eps}{2} \leq \abs{y - x} \leq \frac{3\eps}{4}, \
   \psi(y)\in \Lambda_{\eps,j}(r/2) \right\}
  \]
  for~$j\in\{1, \, \ldots, \, q\}$, where~$\Lambda_{\eps,j}(r/2)$
  is defined in~\eqref{Gamma_eps,j,r}. The sets~$H_j$,
  for~$1\leq j \leq q$, are not disks but are homeomorphic to disks
  and moreover, $D := B^{p+1}_\eps(x) \setminus\bigcup_{j=0}^q H_j$
  is homeomorphic to a ball with spherical holes.
  For all~$y\in D$ satisfying either~$\abs{y - x} > \eps/2$
  or~$\psi(y)\notin\Lambda_{\eps,j}(r/2)$,
  we define $w(y) := v_\eps(\psi(y))$.
  Since~$v_\eps$ is constant on~$\Gamma_{\eps,j}(r/2)$
  (as it follows by~(ii) in Lemma~\ref{lemma:fullcylinder}),
  we can extend~$w$ continuously to~$D$,
  by setting~$w(y)$ equal to a suitable constant for all~$y$
  such that~$\eps/4 \leq \abs{y - x} \leq \eps/2$
  and~$\psi(y)\in\Lambda_{\eps,j}(r/2)$.
  The restriction of~$w$ to~$\partial B^{p+1}_{\eps/4}(x)$
  does not agree with~$u\circ\psi$, but it is homotopic to the latter;
  one defines a homotopy by shrinking the disks
  $\psi^{-1}(\Lambda_{\eps,j}(r/2))\cap \partial B^{p+1}_{\eps/4}(x)$,
  on which~$w$ is constant, to points.
 \end{step}

 \begin{step}[conclusion]
  Smooth functions are dense in~$W^{1,\infty}(\Omega_\eps, \, \NN)$
  (as can be proved, e.g., using convolution
  and the projection~\eqref{projN} onto~$\NN$).
  Therefore, we can replace~$v_\eps$ with a smooth function~$\tilde{v}_\eps\colon\overline{\Omega}_\eps\to\NN$
  that still satisfies~\eqref{dipole3}, \eqref{dipole3bis},
  \eqref{dipole4}, \eqref{dipole5}.
  It only remain to extend~$\tilde{v}_\eps$ inside each
  ball~$B_\eps^{p+1}(x)$, $x\in\Sigma^\prime$.
  Due to~\eqref{dipole5}, the restriction of~$\tilde{v}_\eps$
  on~$\partial B(x, \, \eps)$ is homotopically trivial,
  for all~$x\in\Sigma^\prime$. Therefore,
  applying Lemma~\ref{lemma:removable} in each~$B(x, \, \eps)$
  (and a smoothing argument near $\partial B(x, \, \eps)$),
  we construct a smooth map~$u_\eps\colon\overline{\Omega}\to\NN$
  that coincides with~$\tilde{v}_\eps$ in~$\overline{\Omega}_\eps$ and satisfies
  \begin{equation*}
   \begin{split}
    \int_{\Omega} \abs{\nabla u_\eps}^p \d x
     &\leq \int_{\Omega_\eps} \abs{\nabla \tilde{v}_\eps}^p \d x
      + C\eps \int_{\Gamma_\eps} \abs{\nabla \tilde{v}_\eps}^p \d\L^p \\
     &\leq \int_{\Omega} \abs{\nabla u}^p \d x + C \, \F_\Omega(\Snice(u))
      + C\eps\sum_{j=1}^M \abs{\tau_j} \! ,
    \end{split}
  \end{equation*}
  where the last inequality follows from~\eqref{dipole3} and~\eqref{dipole3bis}.
  Thanks to~\eqref{dipole4}, we also deduce
  \[
   \norm{u - u_\eps}_{L^p(\Omega)} \leq C\eps.
  \]
  The proposition follows. \qedhere
 \end{step}
\end{proof}

\section{Proof of Theorem~\ref{th:Stop} and Proposition~\ref{prop:Orliczflat}}
\label{sect:proofmainthm}

We are finally ready to prove Theorem~\ref{th:Stop}.
The bulk of the argument is the proof of the following statement.

\begin{lemma} \label{lemma:subStop}
 Let~$(u_j)_{j\in\N}$ be a Cauchy sequence
 in~$R^{1,p}(\Omega, \, \NN)$, with respect to the $W^{1,p}$-distance.
 Then, $(\Snice(u_j))_{j\in\N}$ is a Cauchy
 sequence in~$\I_0(\Omega; \, \pN)$.
\end{lemma}
\begin{proof}
 Let~$u$ be the strong~$W^{1,p}(\Omega)$-limit of the~$u_j$'s,
 and let~$(x, \, r)\in\Omega\times(0, \, +\infty)$ be such that
 $Q^{p+1}_r(x)\subseteq\Omega$.
 By Fubini theorem, there exists a subsequence~$j_m\to+\infty$ such that the trace of~$u_{j_m}$ on the slice~$\{x\in\Omega\colon x_i = c\}$, for all~$i\in\{1, \, \ldots, \, p+1\}$ and almost every~$c\in\R$, converge $W^{1,p}$-strongly to that of~$u$.
 As a consequence, the trace of~$u_{j_m}$ on~$\partial Q^n_r(x)$ converges~$W^{1,p}$-strongly to that of~$u$, for almost every~$(x, \, r)$ such that~$Q^n_r(x)\subseteq\Omega$.
 Then, Equation~\eqref{SI} and Proposition~\ref{prop:smallenergy} imply
 \[
  \begin{split}
   \ind(\Snice(u_{j_m}), \, Q^{p+1}_r(x))
   &= \textrm{homotopy class of } u_{j_m} \textrm{ on } \partial Q^{p+1}_r(x) \\
   &= \textrm{homotopy class of } u \textrm{ on } \partial Q^{p+1}_r(x)
  \end{split}
 \]
 for all~$m$ large enough, say~$m\geq m_*(x, \, r)$.
 Lemma~\ref{lemma:Iae} now implies that~$(\Snice(u_{j_m}))_{m\in\N}$ is a Cauchy sequence in~$\I_0(\Omega; \, \G)$, and its limit~$S$ satisfies
 \begin{equation} \label{subStop1}
  \begin{split}
   \ind(S, \, Q^{p+1}_r(x))
   &= \textrm{homotopy class of } u \textrm{ on } \partial Q^{p+1}_r(x)
  \end{split}
 \end{equation}
 for almost all~$(x, \, r)$ such that~$Q^{p+1}_r(x)\subseteq\Omega$, by Lemma~\ref{lemma:ICauchy}.
 In order to complete the proof, it suffices to show that any other subsequence~$j^\prime_m\to+\infty$ admits a further subsequence
 converging to the \emph{same} limit~$S$.
 Given~$(j^\prime_m)_{m\in\N}$, by reasoning exactly as above
 we can extract a (non-relabelled) subsequence and find~$S^\prime\in\I_0(\Omega; \, \G)$ such that $\I_\Omega(\Snice(u_{j^\prime_m}) - S^\prime) \to 0$ as $m\to+\infty$.
 Moreover, the limit~$S^\prime$ satisfies
 \begin{equation*}
  \begin{split}
   \ind(S^\prime, \, Q^{p+1}_r(x))
   &= \textrm{homotopy class of } u \textrm{ on } \partial Q^{p+1}_r(x)
  \end{split}
 \end{equation*}
 for almost all~$(x, \, r)$ such that~$Q^{p+1}_r(x)\subseteq\Omega$.
 Together with~\eqref{subStop1}, this equality implies $\ind(S - S^\prime, \,  Q^{p+1}_r(x)) = 0$ for almost all admissible~$(x, \, r)$, and hence~$\I_\Omega(S - S^\prime) = 0$.
 The lemma follows.
\end{proof}

In the proof of Theorem~\ref{th:Stop}, we will also
use the following characterisation of strong approximability
by smooth maps:

\begin{prop} \label{prop:H}
 A map~$u\in W^{1,p}(\Omega, \, \NN)$ belongs to~$\Hs^{1,p}(\Omega, \, \NN)$
 if and only if, for almost all~$(x, \, r)$ such that~$Q^{p+1}_r(x)\subseteq\Omega$,
 the trace of~$u$ on~$\partial Q^{p+1}_r(x)$ is homotopic to a constant
 (in the sense defined in Section~\ref{sect:homotopy_W1p}).
\end{prop}
\begin{proof}
 The statement is a direct consequence of the characterization of strong approximability by generic skeleta due to Hang and Lin (see~\cite[Remark~6.1]{HangLin-II}, together with \cite[Lemma~4.5 and Proposition~4.1]{HangLin-II}; see also~\cite[Theorems~1.10 and 1.12, Proposition~5.19]{BousquetPonceVanSchaftingen2025}).
 Indeed, Hang and Lin work with rectilinear cell complexes and show, when $p$ is a positive integer, that a map $u\in W^{1,p}(\Omega,N)$ belongs to $\Hs^{1,p}(\Omega,N)$ if and only if the restriction of~$u$ to a generic $p$-skeleton extends to $\Omega$.
 For a cubical complex in dimension~$p+1$ this extension property is equivalent, cell by cell, to the homotopy class of~$u$ on the boundary of each~$(p+1)$-cube being trivial.
 Alternatively, the result could be obtained by (slightly) adapting the proof of Proposition~\ref{prop:RDbar}.
%  In Hang and Lin's construction, genericity is intended with respect to the parameters of the cell decomposition; in the Euclidean cubical setting this amounts to almost every translation.
\end{proof}

% Proposition~\ref{prop:H} follows, e.g., from~\cite[Theorems~1.10 and 1.12, Proposition~5.19]{BousquetPonceVanSchaftingen2025}.
Note that in Proposition~\ref{prop:H}, the only obstruction to approximability by smooth maps is \emph{local}, as it can be computed by restricting~$u$ to contractible sets (cubes).
This is a special feature of~$(p+1)$-dimensional domains; for domains in dimension~$n\geq p+2$, there might be \emph{global} obstructions to approximability by smooth maps.

\begin{proof}[Proof of Theorem~\ref{th:Stop}]
 Lemma~\ref{lemma:subStop}, combined with standard density arguments, implies that~$\Snice$ can be extended to a continuous map
 $\S\colon W^{1, p}(\Omega, \, \NN) \to \I_0(\Omega; \, \pN)$.
 By~\eqref{subStop1}, this map satisfies
 \begin{equation} \label{Stop-ind}
  \begin{split}
   \ind(\S(u), \, Q^{p+1}_r(x))
   &= \textrm{homotopy class of } u \textrm{ on } \partial Q^{p+1}_r(x)
  \end{split}
 \end{equation}
 for all~$u\in W^{1,p}(\Omega, \, \NN)$ and almost all~$(x, \, r)$ such that~$Q^{p+1}_r(x)\subseteq\Omega$.
 Then, comparing~\eqref{deformation} with~\eqref{Sgrid}, we obtain
 \begin{equation} \label{Stop-grid}
  \Sgrid(u, \, h, \, y) = P(\S(u), \, h, \, y)
 \end{equation}
 for all~$u\in W^{1,p}(\Omega, \, \NN)$ and almost all~$(h, \, y)\in (0, \, +\infty)\times Q^{p+1}$.

%  {\BBB We claim that, for all~$u\in W^{1,p}(\Omega, \, \NN)$, the chain~$\S(u)$ satisfies the condition~\eqref{AC}.
%  Indeed, let~$(Q^{p+1}_{r_i}(x_i))_{i=1}^N$ be pairwise disjoint cubes contained in~$\Omega$.
%  For all~$i\in\{1, \, \ldots, \, N\}$ and almost all~$t\in (0, \, r_i)$, the condition~\eqref{Stop-ind} and Lemma~\ref{lemma:cubebounds} imply
%  \[
%   \abs{\ind(\S(u), \, Q^{p+1}_t(x_i))}
%   \leq C \int_{\partial Q^{p+1}_t(x_i)} \abs{\nabla u}^p \d\H^{p},
%  \]
%  for some constant~$C$ depending only on~$\NN$ and~$p$.
%  Since the cubes are pairwise disjoint, we have
%  \[
%   \sum_{i=1}^N \int_0^{r_i} \abs{\ind(\S(u), \, Q^{p+1}_t(x_i))} \d t
%   \leq C \int_{\bigcup_{i=1}^N Q^{p+1}_{r_i}(x_i)} \abs{\nabla u}^p \d x
%  \]
%  and the right-hand side of this inequality can be made arbitrarily small if the total volume of the cubes is small enough. This proves~\eqref{AC}.}

 Let~$u\in\Hw^{1,p}(\Omega, \, \NN)$.
 Applying Proposition~\ref{prop:RS}, we find a family of maps~$(u_h)_{0 < h\leq h_*}\subseteq R^{1,p}(\Omega, \, \NN)$ that converges~$W^{1,p}(\Omega)$-strongly to~$u$ as~$h\to 0$ and satisfies
 \[
  \overline{D}_p(u_h) \lesssim \overline{D}_p(u)
 \]
 for all~$h\in (0, \, h_*]$.
 From Lemma~\ref{lemma:Sbdd}, we obtain
 \[
  \F_\Omega(\S(u_h))
  = \F_\Omega(\Snice(u_h)) \lesssim \overline{D}_p(u)
 \]
 and, hence,
 \[
  \overline{\F}_{\Omega}(\S(u))
  \leq \liminf_{h\to 0} \F_\Omega(\S(u_h))
  \lesssim \overline{D}_p(u) < +\infty.
 \]
 Conversely, let~$u\in W^{1,p}(\Omega, \, \NN)$ be such that~$\overline{\F}_\Omega(\S(u)) < +\infty$.
 By Lemma~\ref{lemma:relaxeddeformation} and~\eqref{Stop-grid}, we know that for almost every~$h > 0$,
 \[
  \int_{Q^n} \F_\Omega(\Sgrid(u, \, h, \, y)) \, \d y
  \lesssim \overline{\F}_\Omega(\S(u)).
 \]
 Then, using Proposition~\ref{prop:RS}, for almost every small enough~$h > 0$ we construct~$u_h\in R^{1,p}(\Omega, \, \NN)$, in such a way that~$u_h\to u$ strongly in~$W^{1,p}(\Omega)$ as~$h\to 0$ and
 \[
  \F_\Omega(\Snice(u_h))
  \lesssim \int_{\Omega} \abs{\nabla u}^p \, \d x
  + \overline{\F}_\Omega(\S(u))
 \]
 for a.e.~$h$. Applying Proposition~\ref{prop:dipole} and the lower semicontinuity of~$\overline{D}_p$, we obtain
 \[
  \overline{D}_p(u)
  \leq \liminf_{h\to 0} \overline{D}_p(u_h)
  \lesssim \int_{\Omega} \abs{\nabla u}^p \, \d x
  + \overline{\F}_\Omega(\S(u)) < +\infty.
 \]
 In particular, $u\in\Hw^{1,p}(\Omega, \, \NN)$.

 It only remains to prove that a map~$u\in W^{1,p}(\Omega, \, \NN)$
 belongs to~$\Hs^{1,p}(\Omega, \, \NN)$ if and only if~$\S(u) = 0$.
 One side of this implication is an immediate consequence of the continuity of~$\S$, since~$\S(\varphi) = \Snice(\varphi) = 0$ for any~$\varphi\in C^\infty(\overline{\Omega}, \, \NN)$.
%  Let~$u\in\Hs^{1,p}(\Omega, \, \NN)$
%  an let~$(\varphi_k)_{k\in\N}$ be a sequence of smooth maps
%  $\overline{\Omega}\to\NN$ that converges $W^{1,p}(\Omega)$-strongly to~$u$.
%  Since~$C^\infty(\overline{\Omega}, \, \NN)\subseteq R^{1,p}(\Omega, \, \NN)$,
%  the definition~\eqref{Stop_nice} of~$\Snice$ immediately
%  gives~$\S(\varphi_k) = 0$. By continuity of~$\S$ and~\eqref{restrS},
%  we obtain $\S(u)\mres B = \S(u_{|B}) = 0$ and, since the ball~$B$
%  is arbitrary, we conclude that~$\S(u) = 0$.
%  using Lemma~\ref{goal:Srestr0}.
 For the opposite implication, assume
 that~$u\in W^{1,p}(\Omega, \, \NN)$ satisfies~$\S(u) = 0$.
 In light of Proposition~\ref{prop:H}, proving that~$u\in\Hs^{1,p}(\Omega, \, \NN)$
 is equivalent to proving that, for almost every~$(x, \, r)$ such that~$Q^{p+1}_r(x)\subseteq\Omega$, the trace of~$u$ on~$\partial Q^{p+1}_r(x)$ is homotopic to a constant.
 However, since~$\S(u) = 0$, this property follows from~\eqref{Stop-ind}.
\end{proof}

Next, we proceed to the proof of Proposition~\ref{prop:Orliczflat}.
Let~$\phi\colon [0, +\infty)\to\R$ be a nonnegative, nondecreasing, convex function that satisfies~$\phi(0) = 0$ and
\begin{equation} \label{hp:Orlicz}
 \int_1^{+\infty} \frac{\d t}{\phi(t)} < +\infty.
\end{equation}
This assumption implies that
\begin{equation} \label{Orlicz-growth}
  \int_{s_0}^{+\infty} \frac{\phi^{-1}(s)}{s^2} \, \d s < +\infty,
\end{equation}
where~$\phi^{-1}$ is the functional inverse of~$\phi$, which exists in some interval of the form~$[s_0, \, +\infty)$.
Indeed, by convexity, $\phi$ is strictly increasing on $[t_0, \, +\infty)$ for some $t_0\geq 1$ and locally absolutely continuous.
Therefore, by the change of variables~$s=\phi(t)$ and integration by parts, we obtain
\[
 \int_{\phi(t_0)}^{\phi(T)}\frac{\phi^{-1}(s)}{s^2}\,\d s
 = \int_{t_0}^{T}\frac{t \, \phi'(t)}{\phi(t)^2}\,\d t
 = \int_{t_0}^{T}\frac{\d t}{\phi(t)}
 + \frac{t_0}{\phi(t_0)} - \frac{T}{\phi(T)}
 <+\infty.
\]
Since~$T/\phi(T)$ is bounded in a neighbourhood of~$+\infty$ (by convexity of~$\phi$), by taking the limit as~$T\to +\infty$ we obtain~\eqref{Orlicz-growth}.

\begin{proof}[Proof of Proposition~\ref{prop:Orliczflat}]
 Let~$u\in\Hw^{1,p}(\Omega, \, \NN)$ satisfy
 \begin{equation} \label{hp:Orlicz-Dini}
  \int_{\Omega} \phi(\abs{\nabla u}^p) \, \d x < +\infty
 \end{equation}
 for some function~$\phi$ as above.
 Given~$r > 0$, let~$\Omega_r$ be the set of points~$x\in\Omega$ such that~$Q^{p+1}_r(x)\subseteq\Omega$.
 All we have to do is to show that
 \begin{equation} \label{Diniwewant}
  \int_0^{r_0}\frac{1}{r^n}
   \int_{\Omega_r}\abs{\ind(\S(u), \, Q_r^n(x))} \d x\,\d r
  < +\infty
 \end{equation}
 for some~$r_0 > 0$.
 Then, Proposition~\ref{prop:Dini-flatness} will imply that~$\S(u)$ belongs to~$\F_0(\Omega; \, \pN)$, and the rpoposition will follow, because~$\F_\Omega$ and~$\overline{\F}_\Omega$ are equivalent on flat chains (Corollary~\ref{cor:relaxeddeformation}).

 We proceed to the proof of~\eqref{Diniwewant}.
 First, we define
 \[
  A_r := \left\{x\in\Omega_r \colon \ind(\S(u), \, Q_r^{p+1}(x))\neq 0 \right\} \! .
 \]
 Since every nonzero element of~$\pN$ has norm at least~$\alpha_p(\NN) > 0$
 (see Lemma~\ref{lemma:norm_p}), Corollary~\ref{cor:fixed-scale-index} gives
 \begin{equation} \label{Orlicz-support}
  \abs{A_r} \leq \frac{1}{\alpha_p(\NN)}
  \int_{\Omega_r}\abs{\ind(\S(u), \, Q_r^{p+1}(x))}\,\d x
  \leq C r^p \, \overline{\F}_\Omega(\S(u))
 \end{equation}
 for almost every~$r > 0$.
 The right-hand side is finite when~$u\in\Hw^{1,p}(\Omega, \, \NN)$,
 thanks to Theorem~\ref{th:Stop}.

 Next, Statement~(i) in Theorem~\ref{th:Stop} and Lemma~\ref{lemma:cubebounds} imply
 \[
  C \abs{\ind(\S(u), \, Q^{p+1}_r(x))}
  \leq \int_{\partial Q_r^{p+1}(x)} \abs{\nabla u}^p \d\H^p
 \]
 for almsot all~$(x, \, r)$ such that~$Q^{p+1}_r(x)\subseteq\Omega$ and some constant~$C$ depending on~$\NN$, $p$ only.
 By Jensen's inequality, we deduce
 \[
  \phi\!\left(\kappa r^{-p}\abs{\ind(\S(u), \, Q_r^{p+1}(x))}\right)
  \leq \frac{1}{2(p+1) r^p}
   \int_{\partial Q_r^{p+1}(x)} \phi\!\left(\abs{\nabla u}^p\right) \d\H^p
 \]
 for almost every~$x\in\Omega_r$, where~$\kappa > 0$ depends only on~$\NN$ and~$p$.
 Integrating with respect to~$x$ and parametrising each face~$\partial Q_r^{p+1}(x)$ as~$x+\eps re_i/2  + z$ for some~$\eps\in\{-1, \, 1\}$, $z\in e_i^\perp\cap Q_r^{p+1}$, we obtain
 \begin{equation*} %\label{Orlicz-modular-index}
  \begin{split}
   &\int_{\Omega_r}
    \phi\!\left(\kappa r^{-p}
    \abs{\ind(\S(u), \, Q_r^{p+1}(x))}\right) \d x \\
   &\qquad \leq \sum_{i=1}^{p+1} \sum_{\eps\in\{-1, \, 1\}}
    \frac{1}{2(p+1) r^p} \int_{\Omega_r}
    \int_{e_i^\perp\cap Q_r^{p+1}}
    \phi\!\left(\abs{\nabla u\left(x+\frac{\eps r}{2} e_i + z\right)}^p\right) \d\H^p(z) \, \d x.
  \end{split}
 \end{equation*}
 For~$x\in\Omega_r$, $z\in e_i^\perp\cap Q_r^{p+1}$ and~$\eps\in\{-1, \, 1\}$, we have~$x+\eps re_i/2  + z\in\Omega$.
 Therefore, by applying Fubini theorem at the right-hand side and making the change of variable~$y = x+\eps re_i/2  + z$ to evaluate the integral in~$\d x$, we obtain
 \begin{equation*} %\label{Orlicz-modular-index}
  \begin{split}
   &\int_{\Omega_r}
    \phi\!\left(\kappa r^{-p}
    \abs{\ind(\S(u), \, Q_r^{p+1}(x))}\right) \d x
   \leq \sum_{i=1}^{p+1} \sum_{\eps\in\{-1, \, 1\}}
    \frac{\H^p(e_i^\perp\cap Q_r^{p+1})}{2(p+1) r^p} \int_{\Omega}
    \phi\!\left(\abs{\nabla u(y)}^p\right) \d y
  \end{split}
 \end{equation*}
 and hence,
 \begin{equation} \label{Orlicz-modular-index}
  \begin{split}
   \int_{\Omega_r} \phi\!\left(\kappa r^{-p}
    \abs{\ind(\S(u), \, Q_r^{p+1}(x))}\right) \d x
   \leq \int_{\Omega} \phi\!\left(\abs{\nabla u(x)}^p\right) \d x =: M.
  \end{split}
 \end{equation}
%  Let
%  \[
%   M:=1+\int_{\Omega} \phi\!\left(\abs{\nabla u(x)}^p\right) \d x.
%  \]
 Applying Jensen's inequality once more, this time on~$A_r$, and using~\eqref{Orlicz-modular-index}, we find
 \[
  \phi\!\left(\frac{\kappa r^{-p}}{\abs{A_r}}
   \int_{A_r} \abs{\ind(\S(u), \, Q_r^{p+1}(x))} \d x\right)
  \leq\frac{M}{\abs{A_r}} ,
 \]
 or equivalently
 \begin{equation} \label{Orlicz-prepointwise}
  r^{-p}\int_{\Omega_r}\abs{\ind(\S(u), \, Q_r^{p+1}(x))} \d x
  \leq \kappa\abs{A_r}\phi^{-1}\!\left(\frac{M}{\abs{A_r}}\right).
 \end{equation}
 Because~$\phi$ is convex and~$\phi(0)=0$, the function~$t\mapsto\phi(t)/t$ is nondecreasing.
 It follows that, for each given~$M>0$, the function
 $s\mapsto s\phi^{-1}(M/s)$
 is nondecreasing on~$(0,+\infty)$.
 Combining~\eqref{Orlicz-prepointwise} with~\eqref{Orlicz-support} gives
 \begin{equation*} %\label{Orlicz-Dini-pointwise}
  r^{-p}\int_{\Omega_r}\abs{\ind(\S(u), \, Q_r^{p+1}(x))} \d x
  \leq C r^p\phi^{-1}(C r^{-p}),
 \end{equation*}
 for some constant~$C$ that depends on~$M$ and~$\overline{\F}_\Omega(\S(u))$.
 Finally, by dividing both sides of this inequality by~$r$ and integrating over~$r\in (0, \, r_0)$, we obtain
 \[
  \begin{split}
   \int_0^{r_0} \frac{1}{r^{p+1}}
   \int_{\Omega_r}\abs{\ind(\S(u), \, Q^{p+1}_r(x))}\,\d x\,\d r
   \leq C\int_0^{r_0} r^{p-1}\phi^{-1}(Cr^{-p})\,\d r
   = C\int_{Cr_0^{-p}}^{+\infty} t^{-2}\phi^{-1}(t) \,\d t,
  \end{split}
 \]
 after the change of variables~$t = C r^{-p}$.
 Chosing~$r_0$ small enough that~$C r_0^{-p}\geq s_0$,
 we see that the desired conclusion~\eqref{Diniwewant}
 follows from~\eqref{Orlicz-growth}.
\end{proof}

\paragraph{Acknowledgments}
The authors are grateful to Antoine D\'etaille for his very useful comments and suggestions on a previous version of this manuscript.
% and in particular for drawing their attention to Ref.~\cite{VanSchaftingen2020}, which provides a proof of Proposition~\ref{prop:compactnorm}, and for pointing out several typos.
The authors are members of GNAMPA--INdAM.
G.C.'s work has been partially supported by GNAMPA project E53\-C2500\-2010001.
G.C. was also part of the ANR project ``Singularities of energy-minimizing vector-valued maps'', ref.~ANR-22-CE40-0006.

\bibliographystyle{plain}
\bibliography{singular_set}

\newcommand{\noop}[1]{}
\begin{thebibliography}{10}

\bibitem{ABO1}
G.~Alberti, S.~Baldo, and G.~Orlandi.
\newblock Functions with prescribed singularities.
\newblock {\em J. Eur. Math. Soc. (JEMS)}, 5(3):275--311, 2003.

\bibitem{ABO2}
G.~Alberti, S.~Baldo, and G.~Orlandi.
\newblock Variational convergence for functionals of {G}inzburg-{L}andau type.
\newblock {\em Indiana Univ. Math. J.}, 54(5):1411--1472, 2005.

\bibitem{Bethuel-Density}
F.~Bethuel.
\newblock The approximation problem for {S}obolev maps between two manifolds.
\newblock {\em Acta Math.}, 167(3-4):153--206, 1991.

\bibitem{Bethuel-DensityTrace}
F.~Bethuel.
\newblock Approximations in trace spaces defined between manifolds.
\newblock {\em Nonlinear Anal.}, 24(1):121--130, 1995.

\bibitem{Bethuel-Extension}
F.~Bethuel.
\newblock A new obstruction to the extension problem for {S}obolev maps between
  manifolds.
\newblock {\em J. Fixed Point Theory Appli.}, 15(1):155--183, Mar 2014.

\bibitem{Bethuel-Inventiones}
F.~Bethuel.
\newblock A counterexample to the weak density of smooth maps between manifolds
  in {S}obolev spaces.
\newblock {\em Invent. Math.}, 219(2):507--651, 2020.

\bibitem{BethuelChiron}
F.~Bethuel and D.~Chiron.
\newblock Some questions related to the lifting problem in {S}obolev spaces.
\newblock {\em {Contemporary Mathematics}}, 446:125--152, 2007.

\bibitem{BethuelCoronDemengelHelein}
F.~Bethuel, J.-M. Coron, F.~Demengel, and F.~H\'elein.
\newblock A cohomological criterion for density of smooth maps in {S}obolev
  spaces between two manifolds.
\newblock In {\em Nematics ({O}rsay, 1990)}, volume 332 of {\em NATO Adv. Sci.
  Inst. Ser. C: Math. Phys. Sci.}, pages 15--23. Kluwer Acad. Publ., Dordrecht,
  1991.

\bibitem{BethuelDemengel}
F.~Bethuel and F.~Demengel.
\newblock Extensions for {S}obolev mappings between manifolds.
\newblock {\em Cal. Var. Partial Differential Equations}, 3(4):475--491, 1995.

\bibitem{BethuelZheng}
F.~Bethuel and X.~Zheng.
\newblock Density of smooth functions between two manifolds in {S}obolev
  spaces.
\newblock {\em J. Funct. Anal.}, 80(1):60 -- 75, 1988.

\bibitem{BourgainBrezisMironescu}
J.~Bourgain, H.~Brezis, and P.~Mironescu.
\newblock Lifting in {S}obolev spaces.
\newblock {\em Journal d'Analyse Math\'ematique}, 80(1):37--86, 2000.

\bibitem{Bousquet2007}
P.~Bousquet.
\newblock Topological singularities in ${W}^{s,p}({S}^{N},{S}^1)$.
\newblock {\em Journal d'Analyse Math\'ematique}, 102:311--346, 2007.

\bibitem{BousquetPonceVanSchaftingen-I}
P.~Bousquet, A.~C. Ponce, and J.~Van~Schaftingen.
\newblock Density of smooth maps for fractional {S}obolev spaces ${W}^{s,p}$
  into $\ell$-simply connected manifolds when $s\geq 1$.
\newblock {\em Confluentes Mathematici}, 5(2):3--24, 2013.

\bibitem{BousquetPonceVanSchaftingen2014}
P.~Bousquet, A.~C. Ponce, and J.~Van~Schaftingen.
\newblock Strong approximation of fractional {S}obolev maps.
\newblock {\em J. Fixed Point Theory Appl.}, 15(2):133--153, 2014.

\bibitem{BousquetPonceVanSchaftingen-II}
P.~Bousquet, A.~C. Ponce, and J.~Van~Schaftingen.
\newblock Strong density for higher order {S}obolev spaces into compact
  manifolds.
\newblock {\em J. Eur. Math. Soc.}, 17(4):763--817, 2015.

\bibitem{BousquetPonceVanSchaftingen2025}
P.~Bousquet, A.~C. Ponce, and J.~Van~Schaftingen.
\newblock Generic topological screening and approximation of sobolev maps.
\newblock Preprint arXiv 2501.18149, 2025.

\bibitem{BrezisMironescu2015}
H.~Brezis and P.~Mironescu.
\newblock Density in ${W}^{s,p}({\Omega}; {N})$.
\newblock {\em Journal of Functional Analysis}, 269(7):2045--2109, 2015.

\bibitem{BrezisMironescu2018}
H.~Brezis and P.~Mironescu.
\newblock {Gagliardo--Nirenberg} inequalities and non-inequalities: The full
  story.
\newblock {\em Ann. Inst. H. Poincaré C Anal. Non Linéaire},
  35(5):1355--1376, 2018.

\bibitem{BN1}
H.~Brezis and L.~Nirenberg.
\newblock Degree theory and {BMO}. {I}. {C}ompact manifolds without boundaries.
\newblock {\em Selecta Math. (N.S.)}, 1(2):197--263, 1995.

\bibitem{BN2}
H.~Brezis and L.~Nirenberg.
\newblock Degree theory and {BMO}. {II}. {C}ompact manifolds with boundaries.
\newblock {\em Selecta Math. (N.S.)}, 2(3):309--368, 1996.
\newblock With an appendix by the authors and Petru Mironescu.

\bibitem{CLOBO}
G.~Canevari, V.~P.~C. Le, R.~Oliver-Bonafoux, and G.~Orlandi.
\newblock ${\Gamma}$-convergence of the $p$-{D}irichlet energy for
  manifold-valued maps.
\newblock Preprint, arXiv:2505.21257 [math.AP], 2025.

\bibitem{CO1}
G.~Canevari and G.~Orlandi.
\newblock Topological singular set of vector-valued maps, {I}: {A}pplications
  to manifold-constrained {S}obolev and {BV} spaces.
\newblock {\em Calc. Var. Partial Differ. Equ.}, 58(2):72, Mar 2019.

\bibitem{CO2}
G.~Canevari and G.~Orlandi.
\newblock Topological singular set of vector-valued maps, {II}:
  {$\Gamma$}-convergence for ginzburg--landau type functionals.
\newblock {\em Arch. Rational Mech. Anal.}, 241(2):1065--1135, 2021.

\bibitem{CaselliFregugliaPicenni2025}
M.~Caselli, M.~Freguglia, and N.~Picenni.
\newblock Coercivity and gamma-convergence of the {$p$}-energy of sphere-valued
  sobolev maps.
\newblock Preprint, arXiv:2503.14142, 2025.

\bibitem{CrossFields}
A.~Chemin, F.~Henrotte, J.-F. Remacle, and J.~van Schaftingen.
\newblock {\em Representing Three-Dimensional Cross Fields Using Fourth Order
  Tensors}, pages 89--108.
\newblock Springer International Publishing, Cham, 2019.

\bibitem{DePauwHardt}
T.~De~Pauw and R.~Hardt.
\newblock Rectifiable and flat {$G$} chains in a metric space.
\newblock {\em Amer. J. Math.}, 134(1):1--69, 2012.

\bibitem{Detaille2026}
A.~Detaille.
\newblock A complete answer to the strong density problem in sobolev spaces
  with values in compact manifolds.
\newblock {\em J. Eur. Math. Soc.}, 2026.
\newblock Published online first.

\bibitem{DetailleMironescuXiao2025}
A.~Detaille, P.~Mironescu, and K.~Xiao.
\newblock Pullback of closed forms by low regularity maps to manifolds, and
  applications, 2025.

\bibitem{DetailleVanSchaftingen}
A.~Detaille and J.~Van~Schaftingen.
\newblock Analytical obstructions to the weak approximation of {S}obolev
  mappings into manifolds.
\newblock Preprint, arXiv:2412.12889, 2024.

\bibitem{DuzaarKuwert}
F.~Duzaar and E.~Kuwert.
\newblock Minimization of conformally invariant energies in homotopy classes.
\newblock {\em Calc. Var. Partial Differential Equations}, 6(4):285--313, 1998.

\bibitem{EellsLemaire}
J.~Eells and L.~Lemaire.
\newblock A report on harmonic maps [{B}ull.\ {L}ondon {M}ath.\ {S}oc.\ {\bf
  10} (1978), no.\ 1, 1--68; {MR}0495450 (82b:58033)].
\newblock In {\em Two reports on harmonic maps}, pages 1--68. World Sci. Publ.,
  River Edge, NJ, 1995.

\bibitem{Federer}
H.~Federer.
\newblock {\em Geometric measure theory}.
\newblock Die Grundlehren der mathematischen Wissenschaften, Band 153.
  Springer-Verlag New York Inc., New York, 1969.

\bibitem{FedererFleming}
H.~Federer and W.~H. Fleming.
\newblock Normal and integral currents.
\newblock {\em Ann. Math. (2)}, 72:458--520, 1960.

\bibitem{Fleming}
W.~H. Fleming.
\newblock Flat chains over a finite coefficient group.
\newblock {\em Trans. Amer. Math. Soc.}, 121:160--186, 1966.

\bibitem{Hajlasz}
P.~Haj{\l}asz.
\newblock Approximation of {S}obolev mappings.
\newblock {\em Nonlinear Anal.}, 22(12):1579--1591, 1994.

\bibitem{HangLin-II}
F.~Hang and F.-H. Lin.
\newblock Topology of {S}obolev mappings. {II}.
\newblock {\em Acta Math.}, 191(1):55--107, 2003.

\bibitem{HangLin-III}
F.~Hang and F.-H. Lin.
\newblock Topology of {S}obolev mappings. {III}.
\newblock {\em Comm. Pure Appl. Math.}, 56(10):1383--1415, 2003.

\bibitem{HKL}
R.~Hardt, D.~Kinderlehrer, and F.-H. Lin.
\newblock Existence and partial regularity of static liquid crystal
  configurations.
\newblock {\em Comm. Math. Phys.}, 105(4):547--570, 1986.

\bibitem{HardtRiviere}
R.~Hardt and T.~Rivi\`ere.
\newblock Connecting rational homotopy type singularities.
\newblock {\em Acta Math.}, 200(1):15--83, 2008.

\bibitem{Harrison}
J.~Harrison.
\newblock Continuity of the integral as a function of the domain.
\newblock {\em The Journal of Geometric Analysis}, 8(5):769--795, 1998.

\bibitem{Harrison-Isomorphism}
J.~Harrison.
\newblock Isomorphisms of differential forms and cochains.
\newblock {\em The Journal of Geometric Analysis}, 8(5):797--807, 1998.

\bibitem{Hatcher}
A.~Hatcher.
\newblock {\em {A}lgebraic {T}opology}.
\newblock Cambridge University Press, Cambridge, 2002.

\bibitem{HeleinWood}
F.~H{\'e}lein and J.~C. Wood.
\newblock Harmonic maps.
\newblock In {\em {H}andbook of {G}lobal {A}nalysis}, pages 417--491, 1213.
  Elsevier Sci. B. V., Amsterdam, 2008.

\bibitem{Luckhaus-PartialReg}
S.~Luckhaus.
\newblock Partial {H}\"older continuity for minima of certain energies among
  maps into a {R}iemannian manifold.
\newblock {\em Indiana Univ. Math. J.}, 37(2):349--367, 1988.

\bibitem{Mermin}
N.~D. Mermin.
\newblock The topological theory of defects in ordered media.
\newblock {\em Rev. Modern Phys.}, 51(3):591--648, 1979.

\bibitem{MironescuVanSchaftingen-Lifting}
P.~Mironescu and J.~Van~Schaftingen.
\newblock Lifting in compact covering spaces for fractional {S}obolev mappings.
\newblock {\em Anal. PDE}, 14(6):1851--1871, 2021.

\bibitem{MironescuVanSchaftingen-Trace}
P.~Mironescu and J.~Van~Schaftingen.
\newblock Trace theory for {S}obolev mappings into a manifold.
\newblock {\em Ann. Fac. Sci. Toulouse Math. (6)}, 30(2):281--299, 2021.

\bibitem{Mucci2009}
D.~Mucci.
\newblock Strong density results in trace spaces of maps between manifolds.
\newblock {\em Manuscripta Mathematica}, 128(4):421--441, 2009.

\bibitem{PakzadRiviere}
M.~R. Pakzad and T.~Rivi{\`e}re.
\newblock Weak density of smooth maps for the {D}irichlet energy between
  manifolds.
\newblock {\em Geom. Funct. Anal.}, 13(1):223--257, 2003.

\bibitem{SchoenUhlenbeck2}
R.~Schoen and K.~Uhlenbeck.
\newblock Boundary regularity and the {D}irichlet problem for harmonic maps.
\newblock {\em J. Diff. Geom.}, 18(2):253--268, 1983.

\bibitem{Spanier}
E.~H. Spanier.
\newblock {\em Algebraic Topology}.
\newblock Springer-Verlag, New York, 1994.
\newblock Corrected reprint of the 1966 original.

\bibitem{VanSchaftingen2020}
J.~Van~Schaftingen.
\newblock Estimates by gap potentials of free homotopy decompositions of
  critical {S}obolev maps.
\newblock {\em Adv. Nonlinear Anal.}, 9(1):1214--1250, 2020.

\bibitem{White86}
B.~White.
\newblock Infima of energy functionals in homotopy classes of mappings.
\newblock {\em J. Diff. Geom.}, 23(2):127--142, 1986.

\bibitem{White-Deformation}
B.~White.
\newblock The deformation theorem for flat chains.
\newblock {\em Acta Math.}, 183(2):255--271, 1999.

\bibitem{White-Rectifiability}
B.~White.
\newblock Rectifiability of flat chains.
\newblock {\em Ann. Math. (2)}, 150(1):165--184, 1999.

\bibitem{Whitney-GIT}
H.~Whitney.
\newblock {\em Geometric integration theory}.
\newblock Princeton University Press, Princeton, N. J., 1957.

\end{thebibliography}

\end{document}